\documentclass[11pt]{article}
 \usepackage{benstyle}
 
\usepackage{fancyhdr}
\usepackage{cite}

\usepackage{hyperref}

\title{Near-optimal learning of Banach-valued, high-dimensional functions via deep neural networks}

\author{Ben Adcock \thanks{Department of Mathematics, Simon Fraser University, 8888 University Drive, Burnaby BC, Canada, V5A 1S6. \emph{E-mail address:} {\tt ben\_adcock@sfu.ca} } \and Simone Brugiapaglia\thanks{Department of Mathematics and Statistics, Concordia University, J.W. McConnell Building, 1400 De Maisonneuve Blvd. W., Montréal, QC, Canada, H3G 1M8.  \emph{E-mail address:} {\tt simone.brugiapaglia@concordia.ca} } \and Nick Dexter\thanks{Department of Scientific Computing, Florida State University,
400 Dirac Science Library, Tallahassee, Florida, USA, 32306-4120.  \emph{E-mail address:}  {\tt nick.dexter@fsu.edu} } \and Sebastian Moraga \thanks{Department of Mathematics, Simon Fraser University, 8888 University Drive, Burnaby BC, Canada, V5A 1S6. 
Corresponding author: {\tt smoragas@sfu.ca}  }}

\begin{document}
\pagestyle{fancy}

\fancyhead{}

\fancyhead[OR]{B. Adcock, S. Brugiapaglia, N. Dexter, S. Moraga}
\fancyhead[OL]{Near-optimal learning via DNNs}
\setcounter{page}{1}

\maketitle

\begin{abstract}
The past decade has seen increasing interest in applying Deep Learning (DL) to Computational Science and Engineering (CSE). Driven by impressive results in applications such as computer vision, Uncertainty Quantification (UQ), genetics, simulations and image processing, DL is increasingly supplanting classical algorithms, and seems poised to revolutionize scientific computing. However, DL is not yet well-understood from the standpoint of numerical analysis. Little is known about the efficiency and reliability of DL from the perspectives of stability, robustness, accuracy, and{, crucially,} sample complexity. For example, approximating solutions to parametric PDEs is a key task in UQ for CSE. Yet, training data for such problems is often scarce and corrupted by errors. Moreover, the target function, while often smooth, is a potentially infinite-dimensional function taking values in the PDE solution space, which is generally an infinite-dimensional Banach space. This paper provides arguments for Deep Neural Network (DNN) approximation of such functions, with both known and unknown parametric dependence, that overcome the curse of dimensionality.  We establish \textit{practical existence theorems} that describe classes of DNNs with dimension-independent architecture widths and depths, and training procedures based on minimizing a (regularized) $\ell^2$-loss which achieve near-optimal algebraic rates of convergence in terms of the amount of training data $m$. These results involve key extensions of compressed sensing for recovering Banach-valued vectors and polynomial emulation with DNNs.  When approximating solutions of parametric PDEs, our results account for all sources of error, i.e., sampling, optimization, approximation and physical discretization, and allow for training high-fidelity DNN approximations from coarse-grained sample data. Our theoretical results fall into the category of non-intrusive methods, providing a theoretical alternative to classical methods for high-dimensional approximation. 
\end{abstract}

\noindent
{\bf Keywords.} high-dimensional approximation, uncertainty quantification, deep learning, deep neural networks,  Banach spaces

\smallskip\noindent
{\bf Mathematics subject classifications (2020)}: 65D40;  68T07; 68Q32

\section{Introduction}
Advances in modern hardware and software over the past decade have enabled the rapid development and training of increasingly complex and specialized Deep Neural Network (DNN) architectures. 
State-of-the-art DNNs for applications such as natural language processing now often feature hundreds of billions of model parameters. 
Impressive results have been achieved by training such models with vast datasets on large distributed computing resources. 
However, many still question their use in critical applications which require rigorous safety standards.
As Deep Learning (DL), i.e., the process of training DNNs on real-world or synthetic data, is increasingly being considered for applications in medicine, science, and engineering, it is important to quantify the efficiency and reliability of DL from both the theoretical and practical standpoints. 

A key challenge arising in many problems in Computational Science and Engineering (CSE) is the approximation of high-dimensional functions from sample data. The ever-expanding literature on the approximation theory of DNNs suggests they may offer a promising alternative for approximating such functions compared to standard techniques
\cite{Herrman2022b,Dung2022,Dung2021b,Dung2022a,Herrmann2022,Dung2021,kutyniok2020theoretical,Petersen2018,poggio2017,Daws2019b,elbrachter2021deep,Herrmann2020,
Schwab2017,Tripathy2018,e2018exponential,Montanelli2021,opschoor2019exponential,Schwab2021d,devore2021neural}. This is coupled by increasing empirical evidence that DL is a promising tool for solving challenging problems arising in CSE, for example Uncertainty Quantification (UQ) problems where the underlying model is described in terms of a parametric Differential Equation (DE) or Partial Differential Equation (PDE) 
\cite{AdcockEtAl2021MSML,DeHoop2022,Bhattacharya2021, dalsanto2020data,geist2020numerical,Kropfl2022,LeiZ2022,ThomasJGrady2022,Becker2022,Cicci2022,Heiss2021,
Khara2021,Li2021}. 
However, there is a large gap between the theory and practice of DL. Most theoretical results establish the existence of a DNN that overcomes the curse of dimensionality (in some suitable sense), but do not claim that it can be trained from data. Parametric modelling applications are also commonly \textit{data scarce}, since data is obtained through costly physical or numerical experiments. Hence it is crucial to understand the \textit{sample complexity} of training DNNs.

The purpose of this paper is to demonstrate that DNNs are effective at learning certain classes of holomorphic high-, or to be precise, infinite-dimensional functions taking values in Banach spaces. The specific classes we consider are motivated by the aforementioned applications, and supported by a large body of literature demonstrating that solution maps of parametric DEs belong to such classes (see, \cite{cohen2015high,Adcock2022} 
and references therein).   Our main results are \textit{practical existence theorems} showing that there are DNN architectures and training procedures similar to those used in practice that can efficiently (in terms of the sample
Crucially, we overcome the curse of dimensionality in both the number of DNN parameters \textit{and} the sample complexity. We also show stability with respect to the various errors arising in the training scheme, including {\textit{measurement error} (i.e., noise or other corruptions in the data)}, \textit{physical discretization error} (i.e., when the infinite-dimensional Banach space is discretized in order to perform computations) and \textit{optimization error} (i.e., inexact solution of the training problem).

As we discuss further in \S\ref{ss:theory-practice-gap}, our results use a particular \textit{handcrafted} DNN architecture. Herein only the final layer is trained, with the weights and biases in other layers being specifically chosen to emulate certain orthogonal polynomials. This differs from common practice, where generally, all layers are trained and the architectures are, arguably, simpler and smaller. However, these practical implementations currently lack theoretical guarantees. Our practical existence theorems offer a theoretical justification for the use of DL in parametric PDEs, by showing the existence of provably good architectures and training strategies for learning from limited datasets. While there are various results on the \textit{existence} of DNNs that efficiently approximate infinite-dimensional functions or operators between Banach spaces, our work is, to the best of our knowledge, the first to assert that such approximations can be efficiently learned in the scarce data regime {via DNN architectures and training procedures that are similar to those used in practice}.

\subsection{Motivations}

We now elaborate the main motivations behind this work.

\pbk
\textit{Parametric DEs.}  Although our theory applies more generally, this work is primarily motivated by \textit{parametric DEs} \cite{cohen2015high,Adcock2022,gunzburger2014stochastic}; that is to say, functions arising as parametric maps of physical systems, which are modelled as (systems of) DEs or PDEs.
 {In such problems, one considers a function $u=u(\x,\y)$, depending on parametric and physical variables  $\y \in \cU$ and $\x \in \Omega$, respectively, that arises as the solution of a DE system}
\begin{equation}
\label{para-DEs}
\cD_{\bm{x}}(u,\y)=0.
\end{equation}
{Here} $\cD_x(\cdot,\y)$ is a differential operator in $\x$ that depends on the parameters $\y$. {A classic example is the} parametric elliptic {diffusion equation}
\begin{equation}\label{eq:diffusion}
- \nabla_x (a(\x,\y) \nabla_x u(\x,\y) ) = F(\x), \quad \x \in \Omega, \qquad
u(\x,\y ) = 0, \quad \x \in \partial \Omega.
\end{equation}
Our work belongs to the class of \textit{nonintrusive methods} \cite{Pawar2019,Eigel2021b,Nikolopoulos2022,dexter2019mixed,AdcockEtAl2021MSML,adcock2022efficient,Nobile2008,Adcock2022}. We consider approximating the parametric solution map $\bm{y} \mapsto u(\cdot,\bm{y})$ from a set of $m$ noisy \textit{sample values}
\bes{
d_i = u(\cdot,\bm{y}_i) + n_i,\quad i = 1,\ldots,m.
}
This application illustrates four key challenges that motivate this work (see also \cite{Adcock2022,AdcockEtAl2021MSML}):
\bulls{
\item[(i)] The problem is high- or infinite-dimensional, since parametric DEs arising in applications may depend on many parameters to describe the modeling scenario, e.g., through the boundary or initial conditions. Infinite-dimensional approximation problems arise when, for instance, Karhunen--Lo\`eve expansions are used to model random fields describing forcing terms or diffusion coefficients in the DE.
\item[(ii)] Data is expensive to generate, and consequently scarce, since acquiring each sample $d_i$ may involve numerically resolving the solution of a PDE to sufficient accuracy.
\item[(iii)] The data $d_i$ is corrupted by errors -- for example, numerical errors from optimization procedures, measurement errors, or discretization errors from the PDE solver.
\item[(iv)] The function to be approximated $\bm{y} \mapsto u(\cdot,\bm{y})$ takes values in an infinite-dimensional function space, typically a Hilbert or Banach space.
}

\noindent
\textit{The gap between DNN approximation theory  and practical performance.}
We noted above that there is a large and growing literature on DNN approximation theory for high-dimensional approximation \cite{Dung2021,Dung2021b,MontanelliEtAl2019,opschoor2019exponential,Schwab2017,Schwab2021d,Li2019,Ryck2021tanh}. Many such results assert the existence of DNNs that achieve desirable approximation rates in relation to their complexity (width, depth, number of nonzero weights and biases, etc).  However, such results typically do not address how to learn such a DNN from data, and crucially (in view of (ii)), the \textit{sample complexity} needed to do so. On the other hand, as also mentioned, there is a growing body of evidence that DL can be an effective tool for high-dimensional parametric DE problems. This work is motivated by the desire to {narrow} this \textit{theory-to-practice gap}. It is known that there are function classes for which such gap cannot be closed \cite{Grohs2021}. 
 Our work shows that for certain classes related to parametric DEs, this is indeed possible.
 
 \pbk
\textit{From {scalar-} to Banach-valued recovery.}   
Much of the aforementioned DNN approximation theory and early uses of DL in scientific computing pertain to scalar-valued functions. However, motivated by problems such as parametric DEs, there is now a growing interest in DL for approximating operators between Banach spaces \cite{Bhattacharya2021,Lanthaler2022,Li2021,WangWang2021,Lu2021,Cai2021c}. {In the parametric DE context, one can in practice approximate the parametric solution map by separately approximating $K$ scalar quantities-of-interest $\bm{y} \mapsto u(\bm{x}_i,\bm{y})$, where $\{ \bm{x}_i \}^{K}_{i=1}$ is a finite mesh over the physical domain (e.g., a finite element mesh). However, key structure can be lost when doing this, leading to a potentially worse approximation  \cite{dexter2019mixed}. This motivates one to develop techniques for simultaneously approximating the whole solution map. Spurred by these observations, in previous work we considered approximating high-dimensional, Hilbert-valued functions using polynomials \cite{dexter2019mixed,adcock2022efficient} and DNNs \cite{AdcockEtAl2021MSML}. However, there is an increasing demand to develop numerical methods for functions taking values in Banach spaces. Banach-valued function approximation problems arise naturally in the setting parametric ODEs \cite{hansen2013sparse}. Moreover, while standard PDEs such as \R{eq:diffusion} are naturally posed in weak form in Hilbert spaces,  the efficient numerical solution of more complicated PDEs increasingly involve weak formulations in Banach spaces \cite{Cai2010,Howell2016,Gatica2020a,Farhloul2009,Colemanres2016,rojas2022banach}. More generally, Banach-valued function approximation problems arise naturally in the context of UQ within the setting of \textit{parametric operator equations} \cite{Dick2014,Dick2016b,Rauhut2017,Schwab2011b}.} Our work can also be a viewed as a contribution to the topic of learning in Banach spaces -- a area of growing interest in the machine learning community \cite{Song2013,Xu2019,Ye2022,Sri2011,Zhang2009}.

\pbk
\textit{Known versus unknown anisotropy.}   
{In general, high-dimensional approximation is only possible under some type of \textit{anisotropy} assumption, i.e., the function in question depends more strongly on certain variables than others. In some cases the nature of this anisotropy -- i.e., the order of importance of the variables and the relative strengths of the interactions between them -- may be known in advance. However, in more realistic scenarios, it is unknown. In this work, we consider both the \textit{known} and \textit{unknown anisotropy} settings. In each case, we develop DNN architectures and training procedures for learning the underlying function from limited data.}

\subsection{Contributions}\label{s:contributions}

We consider learning functions of the form $f : \cU \rightarrow \cV$, where $\cU = [-1,1]^{\bbN}$ is the centered infinite-dimensional hypercube of side length $2$ and $\cV$ is a   Banach space. Given sample points $\bm{y}_1,\ldots,\bm{y}_m \sim_{\mathrm{i.i.d.}} \varrho$, where $\varrho$ is either the uniform or Chebyshev (arcsine) measure over $\cU$, we assume the measurements take the form
\bes{
d_i= f(\y_i)+n_i,\quad i = 1,\ldots,m,
}
where the $n_i$ represent measurement errors. Motivated by applications to parametric DEs, we assume that $f$ satisfies a standard holomorphy assumption, termed \textit{$(\bm{b},\varepsilon)$-holomorphy} \cite{chkifa2015breaking,Schwab2017,Adcock2022} (see \S \ref{ss:holomorphy} for the definition).

Our main contribution is four practical existence theorems, dealing with the known and unknown anisotropy settings and the case where $\cV$ is either a Banach space or a Hilbert space (the additional structure of the latter yields somewhat improved estimates). In each theorem, we assert the existence of a DNN architecture with explicit width and depth bounds and a training procedure based on a (regularized) $\ell^2$-loss function from which any resulting DNN approximates $f$ to within an explicit error bound, with high probability. Specifically, our error bounds take the form
\be{
\label{error-intro}
\nmu{f - \hat{f}}_{L^2_{\varrho}(\cU ; \cV)} \lesssim E_{\mathsf{app}} + m^{\theta} \left ( E_{\mathsf{disc}} + E_{\mathsf{samp}} + E_{\mathsf{opt}} \right ),
}
where $\hat{f}$ is the learned approximation to $f$ and $\nm{\cdot}_{L^2_{\varrho}(\cU ; \cV)}$ is the norm on the Lebesgue-Bochner space of functions $g : \cU \rightarrow \cV$ square-integrable with respect to the measure $\varrho$. {Here $\theta = 0$ if $\cV$ is a Hilbert space and $\theta = 1/2$ (unknown anisotropy case) or $\theta =1$ (known anisotropy case) if $\cV$ is a Banach space.}  
There a several distinguishing features of our analysis that we now highlight:
\bulls{
\item[1.] Our error bound \R{error-intro} accounts for all main sources of error in the problem. Specifically:
\bulls{
\item[(i)] $E_{\mathsf{app}}$ is the \textit{approximation error}, and depends on $m$ and the smoothness of $f$ (see next). 
\item[(ii)] $E_{\mathsf{disc}}$ is a \textit{physical discretization error}. It accounts for the fact that we cannot typically perform computations in $\cV$, since it is an infinite-dimensional function space. Instead, we perform computations in some finite-dimensional subspace $\cV_K \subset \cV$, e.g., a finite element space in the case of parametric PDEs.
\item[(iii)] $E_{\mathsf{samp}}$ is the \textit{sampling error}. It is equal to
 $\sqrt{\frac1m \sum^{m}_{i=1} \nm{n_i}^2_{\cV}}$, and it accounts for the measurement errors $(n_i)_i$.
\item[(iv)] $E_{\mathsf{opt}}$ is the \textit{optimization error}. It accounts for the fact the training problem is never solved exactly, but only to within some tolerance. It is proportional to the error in the loss function.
}
\item[2.] We overcome the curse of dimensionality in the approximation error. The term $E_{\mathsf{app}}$ decays algebraically fast in $m/L$, where $L$ is a (poly)logarithmic factor in $m$. Specifically, when $f$ is $(\bm{b},\varepsilon)$-holomorphic with $\bm{b} \in \ell^p(\bbN)$ for some $0 < p <1$, then
\bes{
E_{\mathsf{app}} \lesssim \pi_K \cdot \left ( \frac{m}{L} \right )^{-\sigma(p)},
}
where $\sigma(p) > 0$ is given by $\sigma(p) = \frac1p-\frac12$ if $\cV$ is a Hilbert space or, if $\cV$ is a Banach space,  {$\sigma(p) = \frac1p - 2$ with $p\leq1/2$}  (known anisotropy) or {$\sigma(p) = \frac12(\frac1p-1)$ with $p\leq 1/2$} (unknown anisotropy), and $\pi_K$ a constant depending on the finite-dimensional subspace $\cV_K$.
\item[3.]  In the Hilbert space case {our error bounds are optimal in terms of the number of samples $m$,} up to constants and (poly)logarithmic factors. Specifically, the rate $m^{1/2-1/p}$ is the best achievable for the class of functions considered, regardless of sampling strategy or learning procedure \cite{Adcock2022learning}. Our results for Banach-valued functions are near-optimal. For instance, in the known anisotropy case they are suboptimal only by a factor of {${m}^{3/2}$}. We conjecture that optimal rates (up to constants and log factors) also hold for Banach-valued functions.
\item[4.] We overcome the curse of dimensionality in the DNN architecture $\cN$. We consider either the Rectified Linear Unit (ReLU), Rectified Polynomial Unit (RePU) or hyperbolic tangent (tanh) activation function. In  the ReLU case the depth of the fully-connected architecture has explicit dependence on the smoothness of $f$ and polylogarithmic-linear scaling in $m$. For the latter two, the width and depth satisfy
\be{
\label{depth-width-bounds}
\mathrm{depth}(\cN) \lesssim \log(m),\qquad \mathrm{width}(\cN) \lesssim \begin{cases} m^2 & \text{known anisotropy} \\ m^{3 + \log_2(m)} & \text{unknown anisotropy} \end{cases} .
}
\item[5.] We analyze both the known and unknown anisotropy settings. In the Hilbert space case, the only differences between the two are the width of the DNN architecture and the (poly)-logarithmic term $L$. Both the depth and the approximation error $E_{\mathsf{app}}$ have the same bounds.
\item[6.] We analyze Banach-valued functions. As observed, previous work has generally considered either scalar- or Hilbert-valued functions. To the best of our knowledge, these are first theoretical results on learning Banach-valued functions from samples with DNNs.
}
Similar to past works in DNN approximation theory \cite{Herrman2022b,Dung2021b,Dung2021,AdcockEtAl2021MSML,adcock2020gap,opschoor2019exponential,Daws2019b,Herrman2022b,Schwab2017,Schwab2021d,Opschoor2022b}, 
our main results are proved by drawing a connection between DNNs and algebraic polynomials. We use DNNs to emulate polynomial approximation via least squares (in the known anisotropy case) and compressed sensing (in the unknown anisotropy case). As a by-product, we also show guarantees for polynomial approximation to infinite-dimensional Banach-valued functions from limited samples. To the best of our knowledge, these results are also new.

\subsection{The gap between theory and practice}\label{ss:theory-practice-gap}

To be precise, in our networks only the parameters in the final layer are trained, which results in a convex optimization problem. The other parameters are \textit{handcrafted} and designed to (approximately) emulate suitable orthonormal polynomials. A consequence of this approach is that the ensuing DNN training strategies are not expected to yield superior performance over the corresponding (least-squares or compressed sensing-based) polynomial approximation procedures. Thus, our main results should be interpreted as a primarily theoretical contribution, i.e., showing the existence of DNN architectures and training strategies  for learning such functions from limited datasets that are near-optimal in terms of the amount of training data $m$. 

Note that this approach is not exclusive to polynomial-based methods. Due to universal approximation theory, one could take other types of methods for approximating functions from samples, emulate the constituent functions via DNNs, and thus obtain DNN architectures and training procedures (i.e., maps from training data to DNNs) with theoretical guarantees. See, e.g., \cite{kutyniok2020theoretical} for one such approach based on reduced basis methods. What differentiates our results is that the ensuing DNN architectures and training procedures are similar to those used in practice. Therefore, our work provides some justification for the application of DL to parametric PDEs, where superior performance over state-of-art techniques, including polynomial-based methods, has been recently observed \cite{AdcockEtAl2021MSML}. Moreover, our results also provide credence to various empirical observations about DL for these applications. First, it has been observed that ReLU activations often lead to worse practical performance with similar-sized architectures than smoother activations. In our setting, we require deeper and wider ReLU networks to obtain the same rates. Second, it has been observed that width is more important than depth in such applications. This broadly agrees with our width and depth bounds \R{depth-width-bounds}. See also \cite{Ryck2021tanh} for similar discussion.

However, it is important to stress that there still remains a substantial gap between theory and practice. Practical DL strategies train all (or most) layers via a nonconvex optimization problem and typically employ simpler and smaller architectures (see, e.g., the bounds in \R{depth-width-bounds}). Yet, they currently lack theoretical guarantees. Further narrowing this gap is an interesting objective for future work.

\subsection{Other related work}\label{ss:otherwork}

We now describe how our work relates to several existing areas of research not previously discussed.

{The class of $(\bm{b},\varepsilon)$-holomorphic function was introduced and formalized in the context of parametric DEs \cite{chkifa2015breaking,Schwab2017,cohen2015high}. Many works have developed approximation methods for such functions, often involving multivariate polynomials. Early approaches include interpolation schemes using sparse grids (see \cite{cohen2015high,chkifa2014high,cohen2018multivariate} and \cite[Chpt.\ 1]{Adcock2022} and references therein). As discussed in \cite[\S 1.7]{Adcock2022}, these methods are best suited to the known anisotropy setting, since they generally require a priori knowledge of a good polynomial subspace in which to construct the approximation. They also do not generally obtain {optimal stability bounds in terms of} $m$, due to growing Lebesgue constants {\cite{chkifa2014high,chkifa2015breaking,cohen2018multivariate}}. More recently, there has been significant focus on least-squares methods \cite{cohen2013stability,chkifa2015discrete, migliorati2014analysis} and methods based on compressed sensing \cite{doostan2011nonadapted,mathelin2012compressed,rauhut2012sparse}. See \cite[Chpts.\ 5 \& 7]{Adcock2022} for reviews. The former are suitable for the known anisotropy case, since they likewise requires knowledge of a good polynomial subspace. In contrast, the latter can handle the unknown anisotropy setting. As noted above, our main results are established by emulating such methods with DNNs.
}

Using polynomial techniques to establish theoretical guarantees for DNN training from limited samples was previously considered in \cite{adcock2020gap,AdcockEtAl2021MSML}. These works consider either scalar- or Hilbert-valued functions in finite dimensions and approximation via ReLU DNNs. Our work can be considered an extension to the (significantly more challenging) infinite-dimensional and Banach-valued setting, while using also other families of DNNs. As noted above, emulation of DNNs by polynomials is a well-established technique in DNN approximation theory. In this paper, we use ideas from \cite{opschoor2019exponential,Schwab2017,Li2019,Ryck2021tanh} to establish our main results. 
For more on polynomial-based methods for high-dimensional approximation, see \cite{adcock2022efficient,Adcock2022,cohen2015high,chkifa2018polynomial} and references therein.
See also \cite{kutyniok2020theoretical} for a different approach based on reduced bases to derive DNN approximation results for parametric PDEs.

Another line of recent DL research involves learning operators between function spaces \cite{Bhattacharya2021,WangWang2021,Lu2021,Cai2021c,Li2021,Lanthaler2022,Nelsen2021}.
 This is motivated in great part by parametric PDEs, where the operator is, for example in the case of \R{eq:diffusion}, the mapping from the diffusion coefficient to the PDE solution $a \in L^{\infty}(\Omega) \rightarrow u = u(\cdot,a) \in H^1_0(\Omega)$. Our work is related to this line of investigation in that we assume a parametrization of $a$ in terms of an infinite vector $\bm{y} \in [-1,1]^{\bbN}$ for which the map $\bm{y} \mapsto u(\cdot,\bm{y})$ is holomorphic. However, it is also different in scope, as we consider approximating an arbitrary Banach-valued, holomorphic function $f : \cU \rightarrow \cV$ which may or may not arise as the solution map of a parametric DE. We note also that many of the aforementioned works assume a Hilbert space formulation, whereas we consider Banach spaces. See \S \ref{s:conclusions} for some further discussion in relation to this line of work.
 
\subsection{Outline}
 
The outline of the remainder of this paper is as follows. In \S \ref{S:Prelim}, we introduce notation and other necessary concepts, including relevant spaces and DNNs.
 Next, \S \ref{S:problem_statement} gives the formal definition of the approximation problem,  the holomorphy assumption on $f$ along with a precise definition of  known and unknown anisotropy. In \S \ref{s:mainres}, we state our four main results and   discuss  their implications.  The next four sections, \S \ref{ss:vector_recovery}--\ref{S:proofs}, are devoted to the proofs of these results. See \S \ref{ss:vector_recovery} for an outline of these sections.
  Finally, in \S \ref{s:conclusions} we present conclusions and topics for future work.

\section{Preliminaries}\label{S:Prelim}

In this section, we introduce key preliminary material needed later in the paper.

\subsection{Notation}\label{S:Notation}
We first introduce some notation. {Let $\bbN$ and $\bbN_0$ be the set of positive and nonnegative integers, respectively}. We write {$\bbN_0^\bbN$ ($\bbR^{\bbN}$) } for the vector space of {nonnegative (real)} sequences indexed over $\bbN$, and $\bbR^n$  for the vector space of real vectors of length $n$. In either space, we write $\bm{e}_j$ for the standard basis vectors, where $j \in \bbN$ or $j \in [n]$. Here and throughout, we write $[n] = \{1,\ldots,n\}$ and {$\bbR_+$ for positive numbers}.

For $1 \leq p \leq \infty$, we write $\nm{\cdot}_{p}$ for the usual vector $\ell^p$-norm and for the induced matrix $\ell^p$-norm. When $0 < p <1$, we use the same notation to denote the $\ell^p$-quasinorm. We also consider infinite multi-indices in $\bbN^{\bbN}_{0}$ with at most finitely-many nonzero terms. Let
\be{
\label{l0norm}
\nm{\bm{\nu}}_0 = | \supp(\bm{\nu}) |,\quad \text{where }\supp(\bm{\nu}) = \{ k : \nu_k \neq 0 \} \subseteq \bbN
}
for a multi-index $\bm{\nu} \in \bbN^{\bbN}_0$. Then we define
\begin{equation}
\label{Fdef-2}
\cF := \lbrace  \bnu = (\nu_k)^{\infty}_{k=1} \in \bbN_0^{\bbN} : \nm{\bm{\nu}}_{0} < \infty\rbrace  
\end{equation}
as the set of multi-indices with at most finitely-many nonzero entries.
We also write $\bm{0}$ and $\bm{1}$ for the multi-indices consisting of all zeros and all ones, respectively.  Finally, we use the inequality $\bm{\mu} \leq \bm{\nu}$ in a componentwise sense, i.e., $\bm{\mu} \leq \bm{\nu}$ means that $\mu_k \leq \nu_k$ for all $k \in \bbN$.

\subsection{Function spaces}\label{S:Setup}
Throughout this paper, we consider the domain 
\bes{
\cU = [-1,1]^{\bbN}
}
with variable $\bm{y} = (y_j)^{\infty}_{j=1} \in \cU$.
We consider probabilities measures on $\cU$ that arise as tensor products of probability measures on the interval $[-1,1]$. Specifically, we focus on the uniform and Chebyshev (arcsine) measures
\begin{equation}\label{meas-unif}
\D \varrho(y) =  2^{-1}\D y,
\quad \mathrm{and} \quad \D \varrho(y) =   \dfrac{1}{\pi \sqrt{1-y}} \D y,\quad \forall y \in [-1,1],
\end{equation}
respectively. The Kolmogorov extension theorem guarantees the existence of a probability measure on $\cU$ arising as the infinite tensor-product of such a one-dimensional measure (see, e.g. \cite[\S 2.4]{Tao2011}). Abusing notation, we denote this measure as 
\begin{equation}\label{measure_infty}
\varrho = \varrho \times \varrho \times \cdots ,
\end{equation}
where on the right-hand side $\varrho$ denotes the one-dimensional measure.

We let $(\cV,\nm{\cdot}_{\cV})$ be a Banach space over $\bbR$ and write $(\cV^*,\nm{\cdot}_{\cV^*})$ for its dual. We write $v \in \cV$ for an arbitrary element of $\cV$ and $v^* \in \cV^*$ for an arbitrary element of $\cV^*$. We let $\ip{\cdot}{\cdot}_{\cV}$ be the duality pairing between $\cV$ and $\cV^*$, i.e.\ $\ip{v^*}{v}_{\cV} = v^*(v)$. Note that
\begin{equation}\label{norm_Vector}
\|v\|_{\cV} 
= \max_{\substack{ v^* \in \cV^* \\ \|v^*\|_{\cV^*} = 1 }} |\ip{v^*}{v}_{\cV} |,\quad \forall v \in \cV,
\end{equation}
(see, e.g., \cite[Cor.\ 1.4]{brezis2010functional}). 
Given $1 \leq p \leq \infty$, we define weighted Lebesgue-Bochner space $L^p_{\varrho}(\cU;\cV)$
as the space consisting of (equivalence classes of) strongly $\varrho$-measurable functions $f: \cU \rightarrow \cV$ for which $\nm{f}_{L^{p}_{\varrho}(\cU ; \cV)} < \infty$, where 
\be{
\nm{f}_{L^p_{\varrho}(\cU;\cV)} : = 
\begin{cases} 
\left( \int_{\cU} \nm{f( \y)}_{\cV}^p \D \varrho (\y) \right)^{1/p} & 1 \leq p < \infty 
\\
\mathrm{ess} \sup_{\y \in \cU} \nm{f(\y)}_{\cV}  & p = \infty
\end{cases}
,
\label{L_p_U_V}
}
and the essential supremum is taken with respect to $\varrho$ in the case $p=\infty$.
We will use these norms to measure the various approximation errors in our subsequent theory.

In general, we cannot work directly in the space $\cV$, since it is usually infinite dimensional. Hence, we also consider a finite-dimensional subspace 
\begin{equation}
\label{eq:conforming}
\cV_K  \subseteq \cV.
\end{equation}
Assuming  \eqref{eq:conforming} corresponds to considering so-called conforming discretizations in the context of finite elements. We let $\{ \varphi_k \}^{K}_{k=1}$ be a (not necessarily orthonormal) basis of $\cV_K$, where $ K  = \dim(\cV_{K})$.  
We also assume that there is a bounded linear operator 
\be{
\label{Ph-def}
\cP_K : \cV \rightarrow \cV_K.
}
To simplify several of the subsequent bounds, we define
\be{
\label{ch-def}
\pi_{K} = \max \left \{ \nm{\cP_K}_{\cV \rightarrow \cV} , 1 \right \}
}
(the assumption $\pi_{K} \geq 1$  is convenient and is of arguably of little consequence for practical purposes). Note that we do not specify a particular form for this operator -- see Remark \ref{rem:Ph} for some further discussion. In particular, we do not use $\cP_{K}$ in our various DNN approximations. It only appears in the resulting error bounds.

For convenience, if $f \in L^2_{\varrho}(\cU ; \cV)$ then we write $\cP_K(f)$ for the function defined pointwise as
\be{
\label{Phf-def}
\cP_K(f)(\y) = \cP_K(f(\y)),\quad \y \in \cU.
}
Later, our various error bounds involve $f - \cP_K(f)$ measured in a suitable Lebesgue-Bochner norm.

\rem{\label{rem:Ph}
When $\cV$ is a Hilbert space, it is natural to choose $\cP_K$ as the orthogonal projection onto $\cV_{K}$. Then $\cP_K(v)$ is the best approximation in $\cV_K$ of $v \in \cV$ and \R{ch-def} holds with $\pi_{K} = 1$. 
The case of Banach spaces is more delicate. First, the best approximation problem
\bes{
\inf_{z \in \cV_K} \|v-z\|_{\cV}
}
may not have a unique solution (a solution always exists since $\cV_K$ is a finite-dimensional subspace). Thus the best approximation map
\be{
\label{def_proxmap}
\cP_{\cV_K} : v \mapsto \{v_K \in \cV_K: \|v-v_K\|_{\cV} = \inf_{z \in \cV_K} \|v-z\| \}
}
is set-valued. Furthermore there does not generally exist a linear operator $\cP_K : \cV \rightarrow \cV_K$ with $\cP_K(v) \in \cP_{\cV_K}(v)$, $\forall v \in \cV$.
The works \cite{Deutsch1982,Holmes1968} establish conditions on the set $\cP_{\cV_K}(v)$  that imply linearity of such operators in a general normed linear space $\cV$. From \cite[Lem.\ 2.1]{Deutsch1982}, such an operator is bounded with $\|\cP_K(v)\|_{\cV} \leq 2 \|v\|_{\cV} $. From \cite[Thm.\ 2.2]{Deutsch1982}, a necessary and sufficient condition for the existence of such a linear operator is that  $\mathrm{ker}(\cP_{\cV_K})$ contains a closed subspace $\cW$ such that $\cV=\cV_K+\cW$, where
\bes{
\mathrm{ker}(\cP_{\cV_K}):= \{v \in \cV:  0 \in  \cP_{\cV_K}(v) \}.
}
Note further that if $\cV$ is strictly convex and $\mathrm{ker}(\cP_{\cV_K})$ is a subspace of $\cV$, then $\cP_{\cV_K}$ is linear \cite[Cor. 2.5]{Deutsch1982}. Moreover, if $\cV=L^p(\Omega)$ with $1<p<\infty$, the operator $\cP_K$ is linear if and only if the quotient space $L^p(\Omega)/\cV_K$ is isometrically isomorphic to some other $L^q(\Omega)$ space \cite[Thm.\ 5]{Ando1966}. 
}

\subsection{Deep neural networks}\label{S:DNN}
Let $f : \cU \rightarrow \cV$ be the unknown function to recover. Consider first the approximation of $f(\bm{y})$ in $\cV_K$. Using the basis $\{ \varphi_k \}^{K}_{k=1}$ we can write this as
\bes{
f(\y) \approx  \sum^{K}_{k=1} c_{k}(\bm{y}) \varphi_k.
}
Notice that the coefficients $c_k$ in this approximation are scalar-valued functions of $\bm{y}$, i.e., $c_k : \cU \rightarrow \bbR$.
Our objective is to approximate the coefficients using a DNN. 
To do this, we consider standard feedforward DNN architectures of the form
\be{
\label{Phi_NN_layers}
\Phi : \bbR^n \rightarrow \bbR^K,\ \bm{z} \mapsto \Phi(\bm{z}) = \cA_{D+1} ( \sigma ( \cA_{D} ( \sigma ( \cdots \sigma ( \cA_0 (\bm{z}) ) \cdots ) ) ) ).
}
Here $\cA_l : \bbR^{N_{l}} \rightarrow \bbR^{N_{l+1}}$, $l = 0,\ldots,D+1$ are affine maps and $\sigma$ is the activation function, which we assume acts componentwise, i.e.\ $\sigma(\bm{z}) := (\sigma(z_i))^{n}_{i=1} $ for $\bm{z} = (z_i)^{n}_{i=1}$. In this work, we consider either the Rectified Linear Unit (ReLU)
\bes{
\sigma_1(z):= \max\{ 0,z\},
}
Rectified Polynomial Unit (RePU)
\bes{
\sigma_{\ell}(z):= \max\{ 0,z \}^{\ell},\qquad \ell = 2,3,\ldots
}
or hyperbolic tangent (tanh)
\bes{
\sigma_0(z)= \frac{\E^{z}-\E^{-z}}{\E^{z}+\E^{-z}}
}
activation function. See Remark \ref{rem:other-activation} below for some discussion on other activation functions.

The values $\{ N_l \}^{D{+1}}_{l=1}$ are the widths of the hidden layers. We also write $N_0 = n$ and $N_{D+{2}} = K$. We write $\cN$ for a class of DNNs of the form \eqref{Phi_NN_layers} with a fixed architecture (i.e., fixed activation function, depth and widths). For convenience, we define
\bes{
\mathrm{width}(\cN) = \max \{ N_1,\ldots, N_{D+1} \},\qquad \mathrm{depth}(\cN) = D.
}
Observe that the DNN \R{Phi_NN_layers} has domain $\bbR^n$, whereas the coefficients $c_k$ have domain $\cU \subset \bbR^{\bbN}$. In order to use a DNN to approximate these functions, we also require a certain restriction operator. Let $\Theta \subset \bbN$ with $|\Theta| = n$. Then we define the \textit{variable restriction operator}
\be{
\label{def:cT}
\cT_{\Theta} : \bbR^{\bbN} \rightarrow \bbR^{n},\ \bm{y} = (y_j)^{\infty}_{j=1} \mapsto (y_j)_{j \in \Theta}.
}
Given a DNN of the form \eqref{Phi_NN_layers}, we now consider the approximation
\be{
\label{fPhi_DNN_finite}
f(\y) \approx f_{\Phi,\Theta}(\bm{y}) = \sum^{K}_{k=1} (\Phi  \circ \cT_{\Theta} (\bm{y}))_k \varphi_k.
}
In our main theorems, besides describing the DNN architecture $\cN$ we also describe a suitable choice of set $\Theta$ defining the variable restriction operator.

\rem{
[Other activation functions]
\label{rem:other-activation}
Our results for ReLU automatically transfer over to the case of the leaky ReLU (see, e.g., \cite[\S 5.1]{AdcockEtAl2021MSML} and \cite[Rem.\ 3.3]{geist2020numerical}). Moreover,  the proofs of our main results only require the existence of a DNN with activation function  $\sigma$  capable of approximating the product of two numbers. Then, following the proof of \cite[Prop.\ 2.6]{opschoor2019exponential}, one can construct an approximation to the multiplication of $n$ numbers as a binary tree of DNNs and subsequently use this to emulate orthogonal polynomials with DNNs. 
This leaves the door open to future extensions with other activation functions such as the logistic function or exponential linear unit, and variations thereof. For more details, we refer to \cite{Maas2013,geist2020numerical,Li2018,LeCun2012}. We restrict our analysis to the three activation functions in this section since they are widely used both in theory and in practice. 
}

\section{The learning problem}\label{S:problem_statement}

Let $f : \cU \rightarrow \cV$ be a Banach-valued function. We consider sample points $(\y_i)_{i=1}^{m}$, where $\y_i \sim_{\mathrm{i.i.d.}} \varrho$, and assume noisy evaluations of $f$ of the form
\be{
\label{samples_of_f}
d_i = f(\bm{y}_i) + n_i \in \cV,\quad \forall i = 1,\ldots,m,
}
where $n_i \in \cV$ is the $i$th noise term. Observe that, in contrast to \cite{AdcockEtAl2021MSML}, we do not assume the samples $d_i$ to be elements of a finite-dimensional  subspace. In typical applications, the samples are computed via some numerical routine, which employs a discretization  of $\cV$. For instance, in the case of parametric DEs, such a discretization could arise from a finite element method. In this paper we do not consider how the evaluations $f(\bm{y}_i)$ are obtained. It may be done by approximating the DE solution with parameter value $\bm{y}_i$, where $n_i$ represents the simulation error. However, it is important to note that we do not assume any structure to the noise $n_i$, other than it be small in norm.

With this in hand, the problem we consider in this paper is as follows:
\begin{problem}\label{prob:main}
Given a truncation operator $\cT_{\Theta}$ with $| \Theta | = n$ and a class of DNNs $\cN$ of the form $\Phi : \bbR^n \rightarrow \bbR^K$, use the training data $\{ (\bm{y}_i , f(\bm{y}_i) + n_i ) \}^{m}_{i=1} \subset \cU \times \cV$ to learn a DNN $ \hat{\Phi} \in \cN$, and therefore an approximation to $f$ of the form 
\be{
\label{fPhi_DNN}
f \approx f_{\hat{\Phi} , \Theta }(\bm{y}) = \sum^{K}_{k=1} (\hat{\Phi} \circ \cT_{\Theta}(\bm{y}))_k \varphi_k\quad \forall \y \in \cU.
}
\end{problem}

\subsection{$(\pmb{b},\varepsilon)$-holomorphic functions and the class $\mathcal{H}(\pmb{b},\varepsilon)$}\label{ss:holomorphy}

We now introduce the precise class of functions considered in this work. We consider Banach-valued functions $f : \cU \rightarrow \cV$ that are holomorphic in specific regions of $\bbC^{\bbN}$ defined by unions of Bernstein (poly)ellipses. We refer to, for example, \cite[Sec.\ 2.3]{adcock2022efficient} for the definition of holomorphy for Hilbert-valued functions, which readily extends to Banach-valued functions.  
In one dimension the \textit{Bernstein ellipse} of parameter $\rho > 1$ is defined by 
\begin{equation*}
\cE(\rho) = \left\lbrace \tfrac{1}{2} (z +z^{-1}): z\in \bbC, 1 \leq |z|\leq \rho\right\rbrace \subset \bbC.
\end{equation*}
Here, given $\brho = (\rho_j) _{j \in \bbN}  $ with $\brho> \bm{1}$, we define the \textit{Bernstein polyellipse} as the Cartesian product
\[
\cE(\brho) =\cE({\rho_1}) \times \cE({\rho_2}) \times \cdots   \subset \bbC^{\bbN}.
\]
Now let $0 < p <1$, $\varepsilon > 0$ and $\bm{b} = (b_{j})_{j \in \bbN} \in \ell^p(\bbN)$ with {$\bm{b} \geq \bm{0}$}. We define
\begin{equation}\label{def_R_b_e}
\cR({\bm{b},\varepsilon}) =  \bigcup \left\lbrace \cE(\brho): \brho \geq \bm{1},\  \sum_{j=1}^{\infty} \left( \dfrac{\rho_j+\rho_j^{-1}}{2} -1 \right) b_j \leq  \varepsilon \right\rbrace \subseteq \bbC^{\bbN}
\end{equation}
and
\be{
\label{B-b-eps-def}
\cH(\bm{b},\varepsilon) = \left \{ f : \cU \rightarrow \cV, \mbox{$f$ holomorphic in $\cR({\bm{b},\varepsilon})$},\ \nm{f}_{L^{\infty}_{\varrho}(\cR({\bm{b},\varepsilon}) ; \cV)} \leq 1\right \}.
}
This is the class of so-called \textit{$(\bm{b},\varepsilon)$-holomorphic} functions \cite{chkifa2015breaking,Schwab2017,Adcock2022} with $L^\infty$-norm at most one over the domain of holomorphy.

\rem{
\label{r:elliptic-case}
This class was developed in context of parametric DEs. It is known that broad families of parametric DEs \R{para-DEs} possess solution maps $\bm{y} \mapsto u(\cdot,\bm{y})$ that are $(\bm{b},\varepsilon)$-holomorphic for suitable $\bm{b}$ depending on the DE. For instance, consider the case of a parametric diffusion equation
\bes{
-\nabla \cdot (a(\bm{x},\bm{y}) \nabla u(\bm{x},\bm{y})) = F(\bm{x}),\ \bm{x} \in \Omega,\qquad u(\bm{x},\bm{y}) = 0,\ \bm{x} \in \partial \Omega,
}
with affine diffusion term given by
\begin{equation}
a(\x,\y) = a_0(\x)+ \sum_{j \in \bbN} y_j \psi_j(\x), \quad \forall \x \in \Omega, \quad \forall \y \in \cU,
\end{equation}
for functions $a_0 \in L^{\infty}(\Omega)$ and $\{ \psi_j\}_{j \in \bbN} \subset L^{\infty}(\Omega)$ satisfying $\sum_{j \in \bbN} | \psi_j(\bm{x}) | \leq a_0(\bm{x}) - r$, $\forall \bm{x} \in \Omega$, for some $r > 0$. Then it is well known that the map $\bm{y} \in \cU \mapsto u(\cdot,\bm{y}) \in H^1_0(\Omega)$ is $(\bm{b},\varepsilon)$-holomorphic with $b_j\geq  \|\psi_j\|_{L^{\infty}(\Omega)}$ and $\varepsilon < r$ (see, e.g., \cite[Sec.\ 4.2.2]{Adcock2022}). See \cite{cohen2015high,Adcock2022} and references therein for further discussion.
}

Note that the parameter  $\varepsilon$ in \R{def_R_b_e} is technically redundant, as it could be absorbed into the term $\bm{b}$. However, as mentioned in \cite[\S 3.8]{Adcock2022}, it is conventional to keep it, due to the connection with parametric DEs.

\subsection{Known versus unknown anisotropy}\label{KnownvsUnknown}

The parameter $\bm{b}\geq \bm{0}$ controls the \textit{anisotropy} of functions in the class $\cH(\bm{b},\varepsilon)$. Specifically, large $b_j$ implies that the condition
\begin{equation*}
\sum_{j=1}^{\infty} \left( \dfrac{\rho_j+\rho_j^{-1}}{2} -1 \right) b_j \leq  \varepsilon
\end{equation*}
holds only for smaller values of the parameter $\rho_j$, meaning that functions in $\cH(\bm{b},\varepsilon)$ are less smooth with respect to the variable $y_j$: namely, they have holomorphic extensions in this variable only to relatively small Bernstein ellipses. Conversely, if $b_j$ is small (or in the extreme, $b_j = 0$), then functions in $\cH(\bm{b},\varepsilon)$ possess holomorphic extensions to larger Bernstein ellipses, and are therefore smoother with respect to the variable $y_j$

In simple cases such as in Remark \ref{r:elliptic-case}, the parameter $\bm{b}$ may be known. We refer to this as the \textit{known anisotropy} setting. Here, we have a specific prior understanding about the behaviour of the target function with respect to its variables in terms of the parameter $\bm{b}$. Therefore, we can strive to use this information to design an approximation scheme.

However, in many practical settings, $\bm{b}$ may be unknown. We refer to this as \textit{unknown anisotropy}. This is often the case in practical UQ settings, where $f$ is considered a black box (i.e., the underlying DE model, if one exists, is hidden). In fact, even if the DE model is known, it may be difficult to establish holomorphy with tight estimates on the parameter $\bm{b}$. See, e.g., \cite[Chpt.\ 4]{Adcock2022}. Additionally, it is worth noting that $\bm{b}$ only conveys limited information and knowledge of it may not lead to efficient methods. For example, the functions $f(\bm{y}) = 1/(1.1-y_1)$ and $f(\bm{y}) = \sin(1000 y_2) / (1.1-y_1)$ are both $(\bm{b},\varepsilon)$-holomorphic with $\bm{b} = (10,0,0,\ldots)$. The latter varies rapidly with $y_2$, while the former does not -- information that cannot be gleaned from $\bm{b}$ alone. 
Therefore, any approximation scheme that uses $\bm{b}$ may not be able to efficiently approximate both functions simultaneously.

Motivated by this discussion, in this paper we consider both the known and unknown anisotropy settings. Specifically, we describe DNN architectures and training procedures for each case and explicit error bound for the resulting DNN approximations. In the unknown anisotropy setting, our tanh and RePU DNN architectures and training procedures are completely independent of $\bm{b}$.

\section{Main results: practical existence theorems}\label{s:mainres}

We now present the main results of this paper. 
In order to state our results, we require two additional  concepts. First, let
\bes{
\min_{\Phi \in \cN} \cG(\Phi) 
}
be a DNN training problem with associated objective function $\cG$. Then we say that $\hat{\Phi} \in \cN$ is an $E_{\mathsf{opt}}$-\textit{approximate minimizer} of this problem, for some $E_{\mathsf{opt}} \geq 0$, if
\begin{equation}\label{def:Eopt}
\cG(\hat{\Phi})  \leq  E_{\mathsf{opt}} +  \min_{\Phi \in \cN} \cG(\Phi).
\end{equation}
Second, let $\bm{b} = (b_i)_{i \in \bbN} \in \bbR^{\bbN}$ be a sequence. We define its \textit{minimal monotone majorant} as 
\be{
\label{min-mon-maj}
\tilde{\bm{b}} = (\tilde{b}_i)_{i \in \bbN},\quad \text{where }
\tilde{b}_i = \sup_{j \geq i} | b_{j}|,\ \forall i \in \bbN.
}
Then, given $0 < p < \infty$, we define the \textit{monotone $\ell^p$} space $\ell^p_{\mathsf{M}}(\bbN)$ as
\bes{
\ell^p_{\mathsf{M}}(\bbN) = \{ \bm{b} \in \ell^{\infty}(\bbN) : \nmu{\bm{b}}_{p,\mathsf{M}} : = \nmu{\tilde{\bm{b}}}_{p} < \infty  \}.
}

\subsection{Learning in the case of unknown anisotropy}\label{S:unknown_theorems}

\begin{theorem}[Banach-valued learning; unknown anisotropy]\label{t:mainthm1}
There are universal constants $c_0$, $c_1 \geq 1$ such that the following holds.
 Let $m\geq 3$,  $0 < \epsilon < 1$, $0<p \leq 1/2$, $\varepsilon > 0$, $\varrho$ be either the uniform or Chebyshev probability measure over $\cU = [-1,1]^{\bbN}$, $\cV$ be a Banach space, $\cV_K \subseteq \cV$ be a subspace of dimension $K$, $\cP_K: \cV \rightarrow \cV_K$ be a bounded linear operator, $\pi_K$ be as in \R{ch-def},
 \be{
 \label{def_L}
 L = L(m,\epsilon)=\log^4(m)  +\log( \epsilon^{-1})
 }
 and 
 \be{
 \label{def_n}
 \Theta = [n],\quad \text{where } n =  \left \lceil \frac{m}{c_0 L} \right \rceil.
 }
Then there exist
\begin{itemize}
\item[(a)] a  class $\cN^j$ of neural networks $\Phi : \bbR^n \rightarrow \bbR^K$ with either the ReLU ($j=1$), RePU $(j=\ell)$ or   tanh ($j=0$) 
activation function with $\ell = 2,3,\ldots$ and bounds for its depth and width given by
\begin{align*}\label{size_depth}
   \mathrm{width}(\cN^{1}  )  \leq   c_{1,1}   \cdot   m^{ 3+ \log_{2}(m)}, \qquad 
    \mathrm{depth}( \cN^{1}  )   \leq c_{1,2}  \cdot    \log(m)  \Big[  \log^2(m) +  p^{-1}\log(m)+  m    \Big],
\end{align*}
in the ReLU case and
  \begin{align*}
  \mathrm{width}(\cN^j  )  \leq c_{j,1}  \cdot m^{ 3+ \log_{2}(m)},\qquad
  \mathrm{depth}(\cN^j  ) \leq c_{j,2}  \cdot  \log_2(m ),
\end{align*}
in the tanh ($j = 0$) or RePU ($j ={\ell}$) cases, where $c_{j,1}$, $c_{j,2}$ are universal constants in the ReLU and tanh  cases and $c_{j,1},c_{j,2}$ depend on $\ell = 2,3,\ldots$ in the RePU case;
\item[(b)] a regularization function $\cJ : \cN^j \rightarrow [0,\infty)$ equivalent to a certain norm of the trainable parameters;
\item[(c)] a choice of regularization parameter $\lambda$ involving only ${m}$ and $\epsilon$;
\end{itemize} 
 such that the following holds for every  $\bm{b} \in \ell^{p}_{\mathsf{M}}(\bbN)$. Let $f \in \cH(\bm{b},\varepsilon)$, where $\cH(\bm{b},\varepsilon)$ is as in \eqref{B-b-eps-def}, draw $\y_1,\ldots, \y_m \sim_{\mathrm{i.i.d.}} \varrho$ and consider    noisy evaluations $d_i = f(\bm{y}_i) + n_i \in \cV$, $i = 1,\ldots,m$,  as in \eqref{samples_of_f}. Then, with probability at least $1- \epsilon$, every $E_{\mathsf{opt}}$-approximate minimizer $\hat{\Phi}$, $E_{\mathsf{opt}} \geq 0$, of the training problem
\begin{equation}\label{trainingprob1}
 \min_{\Phi \in {\cN^j}}  \cG(\Phi),\qquad \text{where }\cG(\Phi) = \sqrt{\frac1m \sum^{m}_{i=1} \nm{f_{\Phi,\Theta}(\bm{y}_i) - d_i }^2_{\cV} } +\lambda \cJ(\Phi),
\end{equation}
satisfies
\be{
\label{main_err_bd}
\| f- f_{\hat{\Phi},\Theta} \|_{L^2_{\varrho}(\cU;\cV)} \leq c_1 \left(    E_{\mathsf{app},\mathsf{UB}}+{m^{1/2}} \cdot (E_{\mathsf{disc}} +E_{\mathsf{samp}}+ E_{\mathsf{opt}})\right ),
}
where  $f_{\hat{\Phi},\Theta}$ is as in \eqref{fPhi_DNN},
\be{
\label{E123_def_1}
E_{\mathsf{app},\mathsf{UB}}= C \cdot \pi_K \cdot  \left ( \frac{m}{ L} \right )^{{\frac12 (1- \frac{1}{p})}},
\quad
E_{\mathsf{samp}} =  \sqrt{\frac1m \sum^{m}_{i=1} \nm{n_i}^2_{\cV} } ,
\quad
E_{\mathsf{disc}} = \nm{f - \cP_K(f)}_{L^{\infty}_{\varrho}(\cU ; \cV)},
}
and $C = C(\bm{b},\varepsilon,p)$ depends on $\bm{b}$, $\varepsilon$ and $p$ only.  
\end{theorem}

The previous theorem applies to general Banach spaces. However, in the Hilbert space case we are able to improve the error bound in several ways.

\begin{theorem}[Hilbert-valued learning; unknown anisotropy]\label{t:mainthm1H}
Consider the setup of Theorem \ref{t:mainthm1}, except where $0<p<1$ and $\cV$ is a Hilbert space. Then the same result holds, except with \eqref{main_err_bd} replaced by
\begin{equation}\label{main_err_bd_H}
\| f- f_{\hat{\Phi},\Theta}\|_{L^2_{\varrho}(\cU;\cV)} \leq c_1 \left(    E_{\mathsf{app},\mathsf{UH}}+E_{\mathsf{disc}} +E_{\mathsf{samp}}+ E_{\mathsf{opt}}\right ),
\end{equation}
$E_{\mathsf{app},\mathsf{UB}}$ replaced by  $E_{\mathsf{app},\mathsf{UH}} = C  \cdot \pi_K \cdot   \left ( {m}/{ L} \right )^{\frac12- \frac{1}{p}}$ and potentially different values of the constants $c_0$, $c_1$, $c_{j,1}$, $c_{j,2}$ and $C(\bm{b},\varepsilon,p)$.
\end{theorem}

This result is an extension of that shown in \cite{AdcockEtAl2021MSML}, which considered the finite-dimensional, Hilbert-valued case with the ReLU activation function only.

Note that Theorem \ref{t:mainthm1} holds for a fixed $0 < p \leq 1/2$ and Theorem \ref{t:mainthm1H} {holds} for a fixed $0 < p < 1$. However, this assumption is only needed for the ReLU case, where, as we see, the depth of the DNN architecture behaves likes $\ord{1/p}$ for small $p$. In the RePU and tanh cases, the depth of the architecture is independent of $p$. This means that the results in fact hold simultaneously for all $0 < p \leq 1/2$ and $0 < p < 1$, respectively, in these cases. Therefore, these activations fully address the unknown anisotropy case: the architectures and training procedures are completely independent of $\bm{b}$, with the assumption $\bm{b} \in \ell^p_{\mathsf{M}}(\bbN)$ being used only to assert a bound for the approximation error. ReLU activations lead to schemes that depend on $p$, but are otherwise independent of $\bm{b}$ as well.

Upon inspection of the proof of Theorem \ref{t:mainthm1} (see, e.g., \eqref{eq:def_p}), we notice that allowing $0<p\leq p^*$ for some $p^*<1$ (instead of $0 < p \leq 1/2$)  is possible if we enlarge the search space in \eqref{def_n} to  $n= \lceil(m/c_0 L)^{\frac{1}{2(1-p^*)}}\rceil$. This results in larger width and depth bounds. However, to keep these bounds as small as possible, we have abstained from doing so in the statement of Theorem \ref{t:mainthm1}.
 
\subsection{Learning in the case of known anisotropy}

The previous two theorems address the case of unknown anisotropy, since the  DNN architecture and training strategy do not require any knowledge of the anisotropy parameter $\bm{b} \in \ell^p_{\mathsf{M}}(\bbN)$ (and, as noted, in the RePU and tanh cases, the parameter $p$). We now consider the case of known anisotropy, in which we assume knowledge of $\bm{b}$ to design the architecture and training strategy.

\begin{theorem}[Banach-valued learning; known anisotropy]\label{t:mainthm2}
There are universal constants $c_0,c_1 \geq 1$ such that the following holds. Let $m\geq 3$, $0 < \epsilon < 1$, $0<p \leq 1/2$, $\varepsilon > 0$, $\bm{b} \in \ell^p(\bbN)$, $\varrho$ be either the uniform or Chebyshev probability measure over $\cU = [-1,1]^{\bbN}$, $\cV$ be a Banach space, $\cV_K \subseteq \cV$ be a subspace of dimension $K$, $\cP_K: \cV \rightarrow \cV_K$ be a bounded linear operator, $\pi_K$ be as in \R{ch-def} and 
 \be{
 \label{def_L-known}
 L = L(m,\epsilon)=\log(m) + \log(\epsilon^{-1}) .
 }
 Then there exist
 \bulls{
 \item[(a)] a set $\Theta \subset \bbN$ of size
 \be{
 \label{def_n_known}
 |\Theta| = n : = \left \lceil \frac{m}{c_0} \right \rceil ,
 }
 \item[(b)]  a class $\cN^{j}$ of neural networks $\Phi : \bbR^n \rightarrow \bbR^K$ with either ReLU ($j = 1$), RePU ($j = {\ell}$) or tanh ($j = 0$) activation function and bounds for its depth and with given by 
\begin{align*}\label{size_depth}
   \mathrm{width}(\cN^{1} ) \leq c_{1,1} \cdot      m^2 , \qquad
    \mathrm{depth}( \cN^{1}  )  \leq c_{1,1} \cdot    \log(m) \left ( p^{-1} \log(m)+ m\right ) {,}
\end{align*}
in the ReLU case and
  \begin{align*}
  \mathrm{width}(\cN^j  )  \leq c_{j,1}  \cdot m^{ 2},\qquad
  \mathrm{depth}(\cN^j  ) \leq c_{j,2}  \cdot  \log_2(m ),
\end{align*}
in the tanh ($j = 0$) or RePU ($j =  {\ell}$) cases, where $c_{j,1}$, $c_{j,2}$ are universal constants in the ReLU and tanh  cases and $c_{j,1},c_{j,2}$ depend on $\ell = 2,3,\ldots$ in the RePU case;   
}
such that the following holds. Let $f \in \cH(\bm{b},\varepsilon)$, where $\cH(\bm{b},\varepsilon)$ is as in \eqref{B-b-eps-def}, draw $\y_1,\ldots, \y_m \sim_{\mathrm{i.i.d.}} \varrho$ and consider noisy evaluations $d_i = f(\bm{y}_i) + n_i \in \cV$, $i = 1,\ldots,m$,  as in \eqref{samples_of_f}. Then, with probability at least $1- \epsilon$, every $E_{\mathsf{opt}}$-approximate minimizer $\hat{\Phi}$, $E_{\mathsf{opt}} \geq 0$, of the training problem
\begin{equation}\label{trainingprob_NN}
 \min_{\Phi \in {\cN^j}}\cG(\Phi),\qquad \text{where } \cG(\Phi) = \sqrt{ \frac1m \sum^{m}_{i=1} \nm{f_{\Phi,\Theta}(\bm{y}_i) - d_i }^2_{\cV} },
\end{equation}
satisfies  
\be{
\label{main_err_bd_2}
\| f- f_{\hat{\Phi},\Theta} \|_{L^2_{\varrho}(\cU;\cV)} \leq c_1 \Big( E_{\mathsf{app},\mathsf{KB}} +  {m} \left( E_{\mathsf{disc}} +E_{\mathsf{samp}}+ E_{\mathsf{opt}} ,\right ) \Big),
}
where  $f_{\hat{\Phi},\Theta}$ is as in \eqref{fPhi_DNN},
\be{
\label{E123_def_2}
E_{\mathsf{app},\mathsf{KB}}= C \cdot \pi_K  \cdot   \left ( \frac{m}{ L} \right )^{{2- \frac1p}},
\quad
E_{\mathsf{samp}}= \sqrt{\frac1m \sum^{m}_{i=1} \nm{n_i}^2_{\cV} } ,
\quad
E_{\mathsf{disc}} = \nm{f - \cP_K(f)}_{L^{\infty}_{\varrho}(\cU ; \cV)}{,}
}
and $C = C(\bm{b},\varepsilon,p)$ depends on $\bm{b}$, $\varepsilon$ and $p$ only.  Moreover, if $\bm{b} \in \ell^p_{\mathsf{M}}(\bbN)$, then the set $\Theta$ in (a) may be chosen explicitly as
\be{
\label{Theta-n-choice-monotone}
\Theta = [n],\qquad \text{where }n = \left \lceil \frac{m}{c_0 L} \right \rceil.
}
\end{theorem}

\begin{theorem}[Hilbert-valued learning; known anisotropy]\label{t:mainthm2H}
Consider the same setup as Theorem \ref{t:mainthm2}, except where $0<p<1$ and $\cV$ is a Hilbert space. Then the same result holds, except \eqref{main_err_bd_2} replaced by 
\begin{equation}\label{main_err_bd_2H}
\| f- f_{\hat{\Phi},\Theta}\|_{L^2_{\varrho}(\cU;\cV)} \leq c_1 \left(    E_{\mathsf{app},\mathsf{KH}}+E_{\mathsf{disc}} +E_{\mathsf{samp}}+ E_{\mathsf{opt}}\right ),
\end{equation}
$E_{\mathsf{app},\mathsf{KB}}$ replaced by  $E_{\mathsf{app},\mathsf{KH}}= C  \cdot \pi_K \cdot  \left ( \frac{m}{ L} \right )^{ \frac12- \frac1p}$ and potentially different values of the constants $c_0$, $c_1$, $c_{j,1}$, $c_{j,2}$ and $C(\bm{b},\varepsilon,p)$.
\end{theorem}

\subsection{Discussion}

These results demonstrate the existence of sample-efficient training procedures for approximating infinite-dimensional, holomorphic functions taking values in Hilbert or Banach spaces using DNNs. As discussed in \S \ref{s:contributions}, they account for all main sources of error in the problem through the approximation errors $E_{\mathsf{app},\cdot}$, the physical discretization error $E_{\mathsf{disc}}$, the sampling error $E_{\mathsf{samp}}$ and the optimization error $E_{\mathsf{opt}}$. Note that the second error $E_{\mathsf{disc}}$ is given in terms of the linear operator $\cP_K(f)$.  However, we reiterate that this operator is only used to provide a bound for $E_{\mathsf{disc}}$. It is not used in the training procedure itself.

The various approximation errors warrant further discussion. Observe that these take the form
\begin{align*}
E_{\textsf{app},\mathsf{UB}} &= C  \cdot \pi_K  \cdot \left ( \frac{m}{ L} \right )^{  {\frac12(1- \frac{1}{p})}}, \qquad  
&E_{\textsf{app},\mathsf{KB}} = C  \cdot \pi_K \cdot \left ( \frac{m}{L} \right )^{ { 2- \frac{1}{p}}}, \\
E_{\textsf{app},\mathsf{UH}} &= C  \cdot \pi_K  \cdot \left ( \frac{m}{ L} \right )^{  \frac12- \frac{1}{p}}, \qquad
&E_{\textsf{app},\mathsf{KH}} = C  \cdot \pi_K \cdot  \left ( \frac{m}{L} \right )^{  \frac12- \frac{1}{p}}, 
\end{align*}
where $L$ is roughly $\log(m)$ in the known anisotropy case and $\log^4(m)$ in the unknown anisotropy case. This discrepancy arises because of the proof strategy. In the known anisotropy setting we emulate a polynomial least-squares scheme via a DNN, whereas in the unknown anisotropy setting we emulate a polynomial (weighted) $\ell^1$-minimization scheme, with significantly more intricate analysis via compressed sensing techniques.

In all cases, we overcome the curse of dimensionality in the sample complexity. Note that the term $E_{\textsf{app},\mathsf{UH}}$ is the same as that shown in \cite{adcock2022efficient} for polynomial approximation via compressed sensing of Hilbert-valued functions. Further, both the terms $E_{\textsf{app},\mathsf{KH}}$ and $E_{\textsf{app},\mathsf{UH}}$ are optimal up to constants and the logarithmic factors \cite{Adcock2022learning}. By contrast, the bounds for Banach-valued functions are worse in both cases. This arises as a consequence of just having the duality pair and not the proper inner product -- see the discussion after  Lemma \ref{l:from_R_to_V_rNSP}. We expect it may be possible to improve these rates via a different, and more involved, argument.

Another key difference between the known and unknown anisotropy cases is the width of the DNN architecture. For the RePU and tanh activation functions, it is polynomial in $m$ in the former case; specifically, $\ordu{m^2}$ for large $m$. In the latter case, it behaves like $\ordu{m^{3+\log_2(m)}}$, i.e., faster than polynomial, but still subexponential in $m$. This discrepancy arises from having to include many more coordinate variables in the unknown anisotropy case to guarantee the desired approximation error bound, whereas in the known anisotropy case we may restrict only to those variables that are known to be important.

The unknown anisotropy case also involves the stronger assumption $\bm{b}\in \ell^p_{\mathsf{M}}(\bbN)$. This assumption means that the variables are, on average, ordered in terms of importance. In fact, it is impossible in the unknown anisotropy setting to learn functions with only the assumption $\bm{b}\in \ell^p(\bbN)$ \cite{Adcock2022learning}. By contrast, in the known anisotropy case we may in fact assume that $\bm{b}\in \ell^p(\bbN)$. Yet we do not have any control over the set $\Theta$ that defines the truncation operator in this case,  except for its size. However, if we suppose that $\bm{b}\in \ell^p_{\mathsf{M}}(\bbN)$ then we may choose $\Theta$ explicitly as in \eqref{Theta-n-choice-monotone}.

\section{Formulating the training problems}\label{ss:vector_recovery}

The remainder of this paper is devoted to the proofs of the main theorems. This follows a similar approach to that of \cite{AdcockEtAl2021MSML}, but with significant extensions to handle the infinite-dimensional and Banach-valued settings considered in this paper. As mentioned, we first formulate learning problems for Banach-valued functions using orthogonal polynomials, and then use DNNs to emulate these problems. The main theorems are then obtained using techniques from compressed sensing theory.  

Specifically, we proceed as follows. In this section, we introduce orthogonal polynomials and reformulate the problem as a recovery problem for Banach-valued vectors. We then introduce the class of DNNs considered and formulate separate learning problems in the known and unknown anisotropy cases. Next, in \S \ref{ss:HilbertCS} we introduce extensions of various compressed sensing concepts from the scalar- or Hilbert-valued setting to the Banach-valued setting. Then in \S \ref{sec:DNN-approx} we describe how to emulate polynomials using DNNs, and give bounds for the width and depths of the corresponding architectures. After that, in \S \ref{s:inf-dim} we describe various polynomial approximation error bounds for infinite-dimensional, holomorphic, Banach-valued functions. These bounds will be used to obtain the desired algebraic rates of convergence. Finally, with the necessary tools in place, in \S \ref{S:proofs} we give the proofs of the main results.

\subsection{Orthogonal polynomial expansions of Banach-valued functions}

{We first require some notation. L}et $\Lambda \subset \cF$ be a multi-index set and consider $\cV$-valued sequences indexed by $\Lambda$, i.e., $\bm{v} = (v_{\bnu})_{\bnu \in \Lambda}$. Now let $\bm{w} = (w_{\bm{\nu}})_{\bm{\nu} \in \Lambda} \in \bbR^{|\Lambda|}$, with $\bm{w} > \bm{0}$ (recall that $\bm{0}$ denotes the vector of zeroes and the inequality $\bm{w} > \bm{0}$ is understood componentwise) be a vector of positive weights. Then, for $0<p<2$ we define the weighted $\ell_w^p(\Lambda;\cV)$ space as the set of $\bm{v} = (v_{\bm{\nu}})_{\bm{\nu} \in \Lambda} \in \cV^{|\Lambda|}$ for which  
\bes{
\nmu{\bm{v}}_{p,\bm{w};\cV} : = \left ( \sum_{\bnu \in \Lambda} w^{2-p}_{\bnu} \nm{v_{\bnu}}^p_{\cV} \right )^{1/p} < \infty.
}
Notice that $\ell^2_{\bm{w}}(\Lambda ; \cV)$ coincides with the unweighted space $\ell^2(\Lambda ; \cV)$, where its norm $\nm{\cdot}_{p;\cV}$ corresponds to the case $\bm{w}= \bm{1}$. 

Let $\varrho$ be either the uniform or Chebyshev measure \R{meas-unif} on $[-1,1]$ and $\{ \Psi_{\nu} \}_{\nu \in \bbN_0} \subset L^2_{\varrho}([-1,1])$ be the corresponding orthonormal basis of $L^2_{\varrho}([-1,1])$. Now let $\varrho$ be the tensor-product measure \R{measure_infty} on $\cU = [-1,1]^{\bbN}$. Then the set of functions $\{ \Psi_{\bm{\nu}} \}_{\bm{\nu} \in \cF} \subset L^2_{\varrho}(\cU)$ given by
\be{
\label{Psi-product}
\Psi_{\bm{\nu}}(\bm{y}) = \prod_{k \in \bbN} \Psi_{\nu_k}(y_k),\quad \bm{y} \in \cU,\ \bm{\nu} \in \cF,
}
is an orthonormal basis of $L^2_{\varrho}(\cU)$. We note in passing that \R{Psi-product} is well-defined since any $\bm{\nu} \in \cF$ has only finitely-many nonzero terms. {Now consider a function $f \in L^2_{\varrho}(\cU ; \cV)$ and define the expansion}
\be{
\label{f-exp}
{\sum_{\bm{\nu} \in \cF} c_{\bm{\nu}} \Psi_{\bm{\nu}}},\qquad \text{where }c_{\bm{\nu}} = \int_{\cU} f(\y) \Psi_{\bm{\nu}}(\bm{y}) \D \varrho(\bm{y}) \in \cV.
}
Notice that the \textit{coefficients} $c_{\bm{\nu}}$ are elements of the Banach space $\cV$. {We require conditions on $f$ under which this expansion converges in a suitable sense. To present these, we recall several concepts from \cite[Sec.\ 3.1]{cohen2015high}. First, a sequence $(\Lambda_n)_{n \in \bbN}$ of finite subsets of $\cF$ is termed an \textit{exhaustion} if, for any $\bnu \in \cF$, there is an $n_0$ such that $\bnu \in \Lambda_n$ for all $n \geq n_0$. Second, the series \R{f-exp} converges \textit{unconditionally} to $f$ in a norm $\nm{\cdot}$ if $\lim_{n \rightarrow \infty} \nm{f - \sum_{\bm{\nu} \in \Lambda_n} c_{\bm{\nu}} \Psi_{\bm{\nu}} } = 0$ for every exhaustion $(\Lambda_n)_{n \in \bbN}$.}

{When $\cV$ is a Hilbert space, the expansion \R{f-exp} converges unconditionally to $f$ in the $L^2_{\varrho}(\cU ; \cV)$-norm, whenever $f \in L^2_{\varrho}(\cU ; \cV)$ \cite[Thm.\ 3.3]{cohen2015high}. We now give a result on convergence in the $L^{\infty}(\cU ; \cV)$-norm, which also holds when $\cV$ is a Banach space.}

{
\begin{lemma}
Suppose that $f \in L^2_\varrho(\cU;\cV)$ and its coefficients $\bm{c} = (c_{\bm{\nu}})_{\bnu \in \cF} $ in \eqref{f-exp} satisfy $\bm{c} \in \ell^1_{\bm{u}}(\cF;\cV)$ for weights $(u_{\bnu})_{\bnu \in \cF}=(\|\Psi_{\bnu}\|_{L^{\infty}(\cU)})_{\bnu \in \cF}$. Then the series
\begin{equation*}
\sum_{\bnu \in \cF} c_{\bnu} \Psi_{\bnu}
\end{equation*}
converges unconditionally to $f$ in $L^{\infty}(\cU;\cV)$ and for any $\Lambda \subseteq \cF$ we have the error bound
\begin{equation}\label{eq:truncweight}
\|f-f_{\Lambda}\|_{L^{\infty}(\cU;\cV)} \leq \sum_{\bnu \not\in \Lambda} u_{\bnu} \|c_{\bnu}\|_{\cV},\quad \text{where }f_{\Lambda} = \sum_{\bnu \in \Lambda} c_{\bnu} \Psi_{\bnu}.
\end{equation}
\end{lemma}
\begin{proof}
Let $\cV^*$ be the continuous dual of $\cV$ and consider an arbitrary $v^* \in \cV^*$ with unit norm. Then $v^* \circ f \in L^2_{\varrho}(\cU)$ (since $f \in L^2_{\varrho}(\cU ; \cV)$) and its coefficients are given by
\begin{equation*}
d^*_{\bnu} = \int_{\cU} (v^* \circ f)(\y) \Psi_{\bm{\nu}}(\bm{y}) \D \varrho(\bm{y}) = v^* \left ( \int f(\bm{y}) \Psi_{\bnu}(\y) \D \varrho(\y) \right ) =  v^*(c_{\bnu}).
\end{equation*}
Here we  used \cite[eqn.\ (1.2)]{hytonen2016analysis}. Since $\bm{c} \in \ell^1_{\bm{u}}(\cF ; \cV)$, we also have that $\bm{d}^* = (d_{\bm{\nu}})_{\bnu \in \cF} \in \ell_{\bm{u}}^1(\cF)$. Therefore, \cite[Thm.\ 3.5]{cohen2015high} implies that the series
\begin{equation}
\label{d-series}
\sum_{\bnu \in \cF} d_{\bnu}^* \Psi_{\bnu}
\end{equation}
converges unconditionally to $v^* \circ f$ in $L^{\infty}(\cU)$.  Now let $(\Lambda_{n})_{n \in \bbN}$ be any exhaustion of $\cF$ and consider $f_n = \sum_{\bnu \in \Lambda_n} c_{\bnu} \Psi_{\bnu}$. Then
\begin{align*}
\lim_{n \rightarrow \infty} \nm{f - f_n}_{L^{\infty}(\cU ; \cV)} &= \lim_{n \rightarrow \infty} \esssup_{\bm{y} \in \cU} \sup_{v^* : \nm{v^*}_{\cV^*} = 1} \left | v^* \circ f(\y)  - v^* \circ f_n(\y) \right | 
\\
& = \lim_{n \rightarrow \infty} \esssup_{\bm{y} \in \cU} \sup_{v^* : \nm{v^*}_{\cV^*} = 1} \left | v^* \circ f(\y)  - \sum_{\bnu \in \Lambda_n} d^*_{\bnu} \Psi_{\bnu}(\y) \right |
\\
& = \lim_{n \rightarrow \infty} \sup_{v^* : \nm{v^*}_{\cV^*} = 1}\esssup_{\bm{y} \in \cU}  \left | v^* \circ f(\y)  - \sum_{\bnu \in \Lambda_n} d^*_{\bnu} \Psi_{\bnu}(\bm{y}) \right | 
\\
& = \lim_{n \rightarrow \infty} \sup_{v^* : \nm{v^*}_{\cV^*} = 1} \nm{v^* \circ f - \sum_{\bnu \in \Lambda_n} d^*_{\bnu} \Psi_{\bnu}}_{L^{\infty}(\cU)} = 0,
\end{align*}
Here, in the last equality. we used unconditional convergence of the series \R{d-series}. This gives the first result. For the error bound, we use this, the choice of weights and the triangle inequality to get
\begin{equation*}
\|f-f_{\Lambda}\|_{L^{\infty}(\cU;\cV)} \leq \sum_{\bnu \not\in \Lambda} \|c_{\bnu} \Psi_{\bnu} \|_{L^{\infty}(\cU;\cV)}
\leq \sum_{\bnu \not\in \Lambda} \|c_{\bnu}\|_{\cV} \|\Psi_{\bnu} \|_{L^{\infty}(\cU)} =  \sum_{\bnu \not\in \Lambda} {u}_{\bnu}\|c_{\bnu}\|_{\cV},
\end{equation*}
as required.
\end{proof}
Note that the assumptions of this lemma always hold in the setting of this paper. Indeed, if $f \in \cH(\bm{b},\varepsilon)$, where $\cH(\bm{b},\varepsilon)$ is as in \R{B-b-eps-def}, then its coefficients $\bm{c} \in \ell^1_{\bm{u}}(\cF ; \cV)$ (see, e.g., \cite[Lem.\ 7.23]{Adcock2022}).
}

 \subsection{Formulation as a vector recovery problem}\label{S:formulation_recovery}

{With this in hand, let} $\Lambda \subset \cF$ be a finite multi-index set of size $|\Lambda | = N$ with the ordering $\Lambda = \{ \bm{\nu}_1,\ldots,\bm{\nu}_N \}$. 
{Let $f \in L^2_{\varrho}(\cU ; \cV)$ with coefficients $\bm{c} \in \ell^1_{\bm{u}}(\cF ; \cV)$ as in \R{f-exp}. We now define the truncated expansion of $f$ based on the index set $\Lambda$ and its corresponding vector of coefficients as
\be{
\label{f_exp_trunc}
f_{\Lambda} = \sum_{\bm{\nu} \in \Lambda} c_{\bm{\nu}} \Psi_{\bm{\nu}},
 \quad
\bm{c}_{\Lambda} = (c_{\bm{\nu}_j})^{N}_{j=1} \in \cV^N.
}
}
Let $m \in  \bbN$ and $\bm{y}_1,\ldots,\bm{y}_m \in \cU$ be the sample points.
We now define the normalized measurement matrix and the measurement and error vectors by 
\begin{equation}\label{def-measMatrix}
\bm{A} =\left(\frac{\Psi_{\bnu_j} (\y_i)}{\sqrt{m}} \right)^{m,N}_{i,j=1} \in \bbR^{m \times N}, \quad \bm{f} = \frac{1}{\sqrt{m}} \left ( f(\bm{y}_i) + n_i \right )^{m}_{i=1} \in \cV^m \quad \text{and}\quad\bm{e} = \frac{1}{\sqrt{m}} (n_i)^{m}_{i=1} \in \cV^m.
\end{equation} 
Notice that the matrix $\bm{A} = (a_{ij})^{m,N}_{i,j=1}$  immediately extends to a bounded linear operator $\bm{A}:\cV^N \rightarrow \cV^m$. Specifically,  $\bm{A} \in \cB(\cV^N,\cV^m)$ is   given by 
\be{
\label{def-inducedBLO}
\bm{x} = (x_i)^{N}_{i=1} \in \cV^N \mapsto \bm{A} \bm{x} = \left ( \sum^{N}_{j=1} a_{ij} x_j \right )^{m}_{i=1} \in \cV^m.
}
For ease of notation, we make no distinction henceforth between a matrix $\bm{A} \in \bbR^{m \times N}$ and the corresponding linear operator in $ \cB(\cV^N,\cV^m)$ (or $\cB(\cV^N_K,\cV^m_K)$). With this in hand, notice that
\bes{
\bm{A} \bm{c}_{\Lambda} = \frac{1}{\sqrt{m}} \left ( f_{\Lambda}(\bm{y}_i) \right )^{m}_{i=1} =  \frac{1}{\sqrt{m}} (f(\bm{y}_i))^{m}_{i=1} - \frac{1}{\sqrt{m}} \left ( f(\bm{y}_i) - f_{\Lambda}(\bm{y}_i) \right )^{m}_{i=1},
}
and therefore
\be{
\label{linsys_for_cLambda}
\bm{A} \bm{c}_{\Lambda} + \bm{e} + \widetilde{\bm{e}} = \bm{f}, \quad \text{ where} \quad  \widetilde{\bm{e}}= \frac{1}{\sqrt{m}}  \left ( f(\bm{y}_i) - f_{\Lambda}(\bm{y}_i) \right )^{m}_{i=1}.
}
Hence, the vector $\bm{c}_{\Lambda}$ of unknown coefficients is a solution of the noisy linear system \R{linsys_for_cLambda}.

We now define the class of DNNs $\cN$. 
To this end, fix $\Theta \subset \bbN$, $|\Theta|=n$ and let  $ \Phi_{\bnu,\delta, \Theta} =   \Phi_{\bnu,\delta} \circ \cT_{\Theta}$ be a DNN approximation to the basis function $\Psi_{\bm{\nu}}$ for $\bm{\nu} \in \Lambda$, where $\Phi_{\bm{\nu},\delta} : \bbR^n \rightarrow \bbR$ is a DNN of the form \eqref{Phi_NN_layers} and $\cT_{\Theta}$ is as in \eqref{def:cT}. The term $\delta>0$ is a parameter that controls the accuracy of the approximation $\Phi_{\bm{\nu},\delta,\Theta} \approx \Psi_{\bm{\nu}}$. These definitions will be useful later in \S \ref{S:proofs}.  Now let 
\bes{
\Phi_{\Lambda,\delta} : \bbR^n \rightarrow \bbR^N, \bar{\y}=(\bar{y}_j)_{j \in \Theta} \mapsto (\Phi_{\bm{\nu},\delta}(\bar{\y}) )_{\bm{\nu} \in \Lambda},\qquad \Phi_{\Lambda,\delta,\Theta} = \Phi_{\Lambda,\delta} \circ \cT_{\Theta},
}
and define the class of DNNs $\cN$ by
\begin{equation}\label{class_DNN}
\cN = \left \{ \Phi : \bbR^n \rightarrow \bbR^K : \Phi(\bar{\bm{y}}) = \bm{Z}^{\top} \Phi_{\Lambda,\delta}(\bar{\bm{y}}),\ \bm{Z} \in \bbR^{N\times K}, 
\bar{\y}=(\bar{y}_j)_{j \in \Theta} \in \bbR^n \right \},
\end{equation}
where  $\bm{Z} \in \bbR^{N \times K}$ is the matrix of trainable parameters. We also define the approximate measurement matrix $\bm{A}' \approx \bm{A}$ by
 \begin{equation}\label{def_A_NN_general}
\bm{A}' = \left(\frac{\Phi_{\bnu_j,\delta, \Theta} (\y_i)}{\sqrt{m}} \right)^{m,N}_{i,j=1} \in \bbR^{m \times N}.
\end{equation}

\subsection{Unknown anisotropy recovery}\label{S:Unk_recov}
We now consider the unknown anisotropy setting of \S \ref{S:unknown_theorems}. Due to the
discussion in \cite[\S 2.5-2.7]{adcock2022efficient} we expect the vector $\bm{c}_{\Lambda}   \in \ell^{p}(\Lambda;\cV)$ to be well-approximated by a (weighted) sparse vector.
Hence, to exploit this structure, we use weighted $\ell^1$-minimization techniques.  Following\cite{adcock2018infinite,chkifa2018polynomial,adcock2018compressed} (see also \cite[Rem.\ 2.14]{Adcock2022}), we let $\bm{w}= \bm{u}\geq \bm{1}$  be the so-called \textit{intrinsic weights}, given by
\be{
\label{weights_def}
 \bm{u} = (u_{\bm{\nu}})_{\bm{\nu} \in \Lambda},\quad \text{where }u_{\bm{\nu}} = \nm{\Psi_{\bm{\nu}}}_{L^{\infty}_{\varrho}(\cU)}, \ \forall\bnu \in \Lambda. 
}
In particular, for Chebyshev and Legendre polynomials these are given explicitly by
\begin{equation}\label{def:int_weights}
u_{\bm{\nu}} =  \nm{\Psi_{\bm{\nu}}}_{L^{\infty}_{\varrho}(\cU)} = \begin{cases} 
      \prod^{d}_{j=1} \sqrt{2 \nu_j +1}  , & \mathrm{Legendre} \\
      2^{\nm{\bnu}_0/2} ,& \mathrm{Chebyshev} \\
   \end{cases}
\end{equation}
where $\nmu{\bm{\nu}}_0$ is as in \eqref{l0norm}.  Recall that $\{ \bm{\nu}_1,\ldots,\bm{\nu}_N \}$  is an ordering of $\Lambda$. In what follows, we often write $u_{i}$ instead of $u_{\bm{\nu}_i}$.

As discussed in  \cite{AdcockEtAl2021MSML,Adcock2022}, choosing an appropriate index set is a vital step towards obtaining the desired approximation rates in \S \ref{s:mainres}. Following \cite{AdcockEtAl2021MSML,adcock2022efficient}, we first recall the following definition:

\defn{
\label{d:lower-anchored-set}
A set  $\Lambda \subseteq \cF$ is \textit{lower} if the following holds for every $\bm{\nu},\bm{\mu} \in \cF$:
\begin{equation}
\label{lowet_set_cond}
(\bm{\nu} \in \Lambda \text{ and } \bm{\mu} \leq \bm{\nu} )\Longrightarrow \bm{\mu} \in \Lambda.
\end{equation}
A set $\Lambda \subseteq \cF$ is \textit{anchored} if it is lower and if the following holds for every $j \in \mathbb{N}$:
\begin{equation*}
\bm{e}_j \in \Lambda \Longrightarrow 
\{\bm{e}_1,\bm{e}_2, \ldots, \bm{e}_{j} \}\subseteq \Lambda.
\end{equation*}
}

See, e.g., \cite{Adcock2022,cohen2015high,chkifa2015discrete} for more details about lower and anchored sets. 
Following a similar argument to \cite{adcock2022efficient}, given $n \in \bbN$, we now let
\be{
\label{HC_index_set_inf}
\Lambda = \Lambda^{\mathsf{HCI}}_{n} = \left \{ \bm{\nu} = (\nu_k)^{\infty}_{k=1} \in \cF : \prod_{k:\nu_{k}\neq 0} (\nu_k + 1) \leq n,\ \nu_k = 0,\ k > n \right \} \subset \cF.
}
Note that $\Lambda^{\mathsf{HCI}}_{n}$ is isomorphic to the $n$-dimensional \textit{hyperbolic cross} index set of order $n-1$. A key property of this set is that it contains every anchored set (Definition \ref{d:lower-anchored-set}) of size at most $n$ (see, e.g., \cite[Prop.\ 2.18]{Adcock2022}).
Notice also that 
\be{
\label{N_bound}
N := | \Lambda^{\mathsf{HCI}}_{n} | \leq  \E n^{2 + \log_2(n)},\quad \forall n \in \bbN.
}
See \cite[Lem.\ B.5]{Adcock2022} (this result is based on \cite[Thm.\ 4.9]{kuhn2015approximation}).

We now construct the DNN training problem considered in Theorem \ref{t:mainthm1}. First, consider the  Banach-valued, weighted Square-Root LASSO (SR-LASSO) optimization problem \cite{belloni2011square,adcock2019correcting,sun2012scaled} 
\be{
\label{wsr-LASSO}
\min_{\bm{z} \in \cV^N_K} \cG(\bm{z}),\qquad \cG(\bm{z}) : = \lambda \nm{\bm{z}}_{1,\bm{u};\cV} + \nmu{\bm{A} \bm{z} - \bm{f} }_{2;\cV}.
}
 Here $\lambda>0$ is a tuning parameter and $\bm{A}$ and $\bm{f}$ are as in \eqref{def-measMatrix}.  To obtain a DNN training problem, we replace $\bm{A}$ with its approximation $\bm{A}'$, defined by \eqref{def_A_NN_general}, giving the optimization problem
\be{
\label{wsr-LASSO_NN}
\min_{\bm{z} \in \cV^N_K} \cG'(\bm{z}),\qquad \cG'(\bm{z}) : = \lambda \nm{\bm{z}}_{1,\bm{u};\cV} + \nmu{\bm{A}' \bm{z} - \bm{f} }_{2;\cV} .
}
To show that \R{wsr-LASSO_NN} is equivalent to a DNN training problem we argue as follows. Let $\{ \varphi_i \}_{i=1}^K$ be the basis of $\cV_K$ and $\bm{z} = (z_{\bm{\nu}_j})_{j=1}^N$ be an arbitrary element of $\cV^N_K$. Now, recall that $\cN$ is the class of DNNs defined in \eqref{class_DNN}. 
Then, we can associate $\bm{z}$ with a DNN  $\Phi \in \cN$ via the relation
\be{
\label{zZrelation}
\Phi = \bm{Z}^{\top} \Phi_{\Lambda,\delta},\quad \text{where $\bm{Z} =\left(  Z_{j,k} \right)_{j,k=1}^{N,K} \in \bbR^{N \times K}$ is such that }z_{\bnu_j} = \sum_{k=1}^K  Z_{j,k} \varphi_k, \ \forall j \in [N].
}
Now observe that
\bes{
f_{\Phi,\Theta}(\bm{y}) \hspace{-0.1cm}
=\hspace{-0.2cm} \sum^{K}_{k=1} ((\Phi \circ \cT_{\Theta})(\bm{y}))_k \varphi_k 
=\hspace{-0.2cm} \sum^{K}_{k=1} (\bm{Z}^{\top} \Phi_{\Lambda,\delta,{\Theta}}(\bm{y}))_k \varphi_k 
=\hspace{-0.2cm} \sum^{K}_{k=1} \sum_{j=1}^N Z_{j,k} \Phi_{\bm{\nu}_j,\delta,{\Theta}}(\bm{y}) \varphi_k 
=\hspace{-0.2cm} \sum_{\bm{\nu} \in \Lambda} z_{\bm{\nu}} \Phi_{\bm{\nu},\delta,{\Theta}}(\bm{y}). }
Hence, if $d_i = f(\bm{y}_i) + n_i \in \cV$ are the noisy evaluations of $f$, then
\bes{
\nm{\bm{A}' \bm{z} - \bm{f}}_{2;\cV} = \sqrt{\frac1m \sum^{m}_{i=1} \nm{ \sum_{\bm{\nu} \in \Lambda} z_{\bm{\nu}} \Phi_{\bm{\nu},\delta , {\Theta}}(\bm{y}_i) - d_i }^2_{\cV} } = \sqrt{\frac1m \sum^{m}_{i=1} \nm{f_{\Phi,\Theta}(\bm{y}_i) - d_i }^2_{\cV} }.
}
Now let $\cJ : \cN \rightarrow [0,\infty)$ be the regularization functional defined by 
\bes{
\nm{\bm{z}}_{1,\bm{u};\cV} = \sum_{j=1}^{N} u_{\bm{\nu}_j}\|z_{\bnu_j} \|_{\cV} = \sum_{j=1}^{N}  u_{\bm{\nu}_j} \nm{\sum_{k=1}^K Z_{j,k} \varphi_k }_{\cV} : = \cJ(\Phi),
}
where $\Phi \in \cN$ is as in \eqref{zZrelation}. Clearly $\cJ$ is a norm over the trainable parameters, as claimed. Using this and the previous relation, we deduce that \eqref{wsr-LASSO_NN} is equivalent to the DNN training problem
\bes{
\min_{\Phi \in \cN} \sqrt{\frac1m \sum^{m}_{i=1} \nm{f_{\Phi,\Theta}(\bm{y}_i) - d_i }^2_{\cV} } +\lambda \cJ(\Phi).
}
By this, we mean that every minimizer $\hat{\Phi} \in \cN$ of this problem corresponds to a minimizer $\hat{\bm{z}}$ of \eqref{wsr-LASSO_NN} via the relation \eqref{zZrelation}, and vice versa.

\subsection{Known anisotropy recovery}\label{S:Know_recov}

In the previous case, we assumed that the {polynomial} coefficients were approximately (weighted) sparse, but, due to the unknown anisotropy, we do not know which {of these} coefficients are the most significant. Thus we choose the set $\Lambda \subset\cF$ as a large set in which we expect these significant coefficients to lie. Conversely, in the known anisotropy setting, we effectively have this knowledge. Hence, we can choose a smaller set $S \subset \cF$ of size $|S| = s$ which contains the indices of these coefficients. {Indeed, later in our proofs, we will see that it is possible to choose such a set $S$ depending on $\bm{b}$ and yielding the desired approximation rate in terms of $s$ (see Remark \ref{rem:S-idpt-f}).}

Analogously to \S \ref{S:formulation_recovery}, we define the normalized measurement matrix and the approximate normalized measurement matrix by 
 \begin{equation}
 \label{def-measMatrix-known}
 \bm{A} =\left(\frac{\Psi_{\bnu_j} (\y_i)}{\sqrt{m}} \right)^{m,s}_{i,j=1} \in \bbR^{m \times s}, \quad \text{and}\quad
 \bm{A}' := \left(\frac{\Phi_{\bnu_j,\delta,{\Theta}} (\y_i)}{\sqrt{m}} \right)^{m,s}_{i,j=1} \in \bbR^{m \times s},
 \end{equation}
where $\{ \bm{\nu}_1,\ldots,\bm{\nu}_s \}$ is an ordering of $S$.
Likewise, we  truncate the expansion of $f$ and its vector coefficients based on \eqref{f_exp_trunc} for the index set $S$. Defining the class of DNNs $\cN$ as in \eqref{class_DNN}, except with $\Lambda$ and $N$ replaced by $S$ and $s$, respectively, we now see that the training problem \eqref{trainingprob_NN} can be expressed as the Banach-valued minimization problem
\begin{equation}
\label{def:G_LS_probelmA}
\min_{\bm{z} \in \cV^s_K} \cG'(\bm{z}),\qquad \cG'(\bm{z}) : =  \nmu{\bm{A}' \bm{z} - \bm{f}}_{2;\cV},
\end{equation}
where $\bm{f} = \frac{1}{\sqrt{m}} (d_i)^{m}_{i=1} \in \cV^m$.  To be precise, any $\bm{\hat{z}}=(\hat{z}_{\bnu})_{\bnu \in S}$ that is a minimizer of  \eqref{def:G_LS_probelmA} defines a minimizer $\hat{\Phi}$ of \eqref{trainingprob_NN}  via the relation \eqref{zZrelation}, except with $\Lambda$ and $N$ replaced by $S$ and $s$, respectively. As before, we also consider \R{def:G_LS_probelmA} as an approximation to a minimization problem with matrix $\bm{A}$ for the polynomial coefficients $\bm{c}_S$:
\bes{
\min_{\bm{z} \in \cV^s_K} \cG(\bm{z}),\qquad \cG(\bm{z}) : =  \nm{\bm{A} \bm{z} - \bm{f}}_{2;\cV}.
}
Notice that, upon squaring the objective function, this is equivalent to an algebraic least-squares problem.
To end this section, it is worth mentioning here that, for ease of notation, we write $\cG$ and $\cG'$ for the objective functions in both the known and unknown anisotropy cases. Mathematically, the known anisotropy setting is just a particular case of the unknown anisotropy setting with $\lambda = 0$ and $\Lambda$ and $N$ replaced by $S$ and $s$, respectively.

\section{Banach-valued compressed sensing}\label{ss:HilbertCS}
In this section, we extend relevant aspects of the theory of compressed sensing to the Banach-valued setting, and use these to obtain various error bounds and sample complexity estimates that will be needed later. We refer to \cite{Adcock2022,AdcockEtAl2021MSML,adcock2022efficient,dexter2019mixed} for the classical (scalar- and Hilbert-valued) setting.

\subsection{Weighted sparsity and weighted  rNSP}
 
For positive weights $\bm{w}= (w_{\bm{\nu}})_{\bm{\nu} \in \cF}  > \bm{0}$ and a multi-index set $S\subseteq \cF$, we define the weighted cardinality of $S$ as
\bes{
|S|_{\bm{w}} : = \sum_{ \bnu \in S} w^2_{\bnu}.
}
Note that  $|S|_{\bm{w}}$ may take values in $\bbR\cup \{+\infty\}$.
Now, let $k\geq 0$ and $\Lambda \subseteq \cF$ be a multi-index set. Then, given a $\cV$-valued vector $\bm{z} = (z_{\bm{\nu}})_{\bm{\nu} \in \Lambda}$, we say that $\bm{z}$ is \textit{weighted $(k,\bm{w})$-sparse} if  $| \supp(\bm{z}) |_{\bm{w}} \leq k$, where
\begin{equation}\label{def_support}
 \supp(\bm{z}):= \{ \bnu \in \Lambda: \| z_{\bnu} \|_{\cV} \neq 0\}
 \end{equation}  
is the \textit{support} of the vector $\bm{z}$. Let 
 \begin{equation}\label{def_Sigmakw}
\Sigma_{k,\bm{w}}:= \{ \bm{z} \in \ell^p_{\bm{w}}(\Lambda;\cV) : | \supp(\bm{z}) |_{\bm{w}} \leq k \}
\end{equation}
denote the set of such vectors. Then, for $0 < p \leq 2$, we let
\begin{equation}\label{weighted_best_s_term}
\sigma_{k}(\bm{c})_{p,\bm{w};\cV} = \min \left \{ \nm{\bm{c} - \bm{z}}_{p,\bm{w};\cV} : \bm{z} \in \Sigma_{k,\bm{w}} \right \}
\end{equation}
be the \textit{$\ell^p_{\bm{w}}$-norm weighted best $(k,\bm{w})$-term approximation error} of an arbitrary $\bm{c} \in \ell^p_{\bm{w}}(\Lambda;\cV) $.
Notice that this is equivalent to
\begin{equation}
\label{sigma-k-w-equiv}
\sigma_{k}(\bm{c})_{p,\bm{w};\cV} = \inf \left \{ \nm{\bm{c} - \bm{c}_S}_{p,\bm{w}; \cV} : S \subseteq \Lambda,\ |S|_{\bm{w}} \leq k \right \}.
\end{equation}
Here and elsewhere, for a sequence $\bm{c} = (c_{\bm{\nu}})_{\bm{\nu} \in \cF}$ and a set $S \subseteq \cF$,  we write $\bm{c}_{S}$ for the sequence with $\bm{\nu}$th entry equal to $c_{\bm{\nu}}$ if $\bm{\nu} \in S$ and zero otherwise.
For convenience, for $s \in \bbN$, we also write
\begin{equation}
\label{sigma-equiv}
\begin{split}
\sigma_{s}(\bm{c})_{p;\cV} &= \inf \left \{ \nm{\bm{c} - \bm{z}}_{p; \cV} : \bm{z}\in \ell^{p}(\Lambda; \cV), \, |\supp(\bm{z})| \leq s \right \}
\\
&= \inf \left \{ \nm{\bm{c} - \bm{c}_S}_{p; \cV} :  S \subseteq \Lambda,\ |S| \leq s \right \}
\end{split}
\end{equation}
for the (unweighted)   \textit{$\ell^p$-norm best $s$-term approximation error}.

Now let $\bm{A} \in \cB(\cW^N,\cW^m)$ where $\cW \subseteq \cV$ is a closed  subspace. 
 We say that $\bm{A}$ has the \textit{weighted robust Null Space Property (rNSP) over $\cW$ of order $(k,\bm{w})$  with constants $0 \leq \rho<1$ and $\gamma \geq 0$} if 
\begin{equation*}
\nm{\x_S}_{2;\cV} \leq  \dfrac{\rho \nm{\x_{S^c}}_{1,\bm{w};\cV} }{\sqrt{k}}+\gamma \nm{\bm{A} \x}_{2;\cV} ,\quad \forall \x \in \cW^N,
\end{equation*} 
for any $S \subseteq [N]$ with $|S|_{\bm{w}} \leq k$. See, e.g., \cite[Defn.\ 6.22]{Adcock2022} or \cite[Sec. 4.1]{Rauhut2016}. 
This property implies certain distance bounds in the $\ell^1_{\bm{w}}$- and $\ell^2$-norms. The following lemma is a straightforward extension to Banach spaces from the Hilbert-valued counterpart (see, e.g., \cite[Lem.\  7.4]{adcock2022efficient}). For this reason, we omit its proof.

\begin{lemma}
[Weighted rNSP implies error bounds for inexact minimizers]
\label{lemma-musuboptimal}
Suppose that $\bm{A} \in \cB(\cW ^N,\cW ^m)$  has the weighted rNSP  over a closed subspace $\cW \subseteq \cV$ of order $(k,\bm{w})$ with constants  $0\leq \rho<1$ and $\gamma > 0$. Let $\x \in \cW^N$, $\f \in \cV^m$ and $\bm{e} = \bm{A}\x- \bm{f} \in \cVM$, and consider the minimization problem 
\be{
\label{SRLASSO-CS-sec}
\min_{\bm{z}\in \cW^N} \cG(\bm{z}),\qquad \cG(\bm{z}) : = \lambda \nm{\bm{z}}_{1,\bm{w};\cV} + \nmu{\bm{A} \bm{z} - \bm{f}}_{2;\cV}{,}
}
with parameter
\be{
\label{lambda-bound-rNSP-err}
0 < \lambda \leq  \dfrac{(1+ \rho)^2}{(3+\rho) \gamma } {k}^{-1/2}.
}
Then
\eas{
\nm{\tilde{\x}-\x}_{1,\bm{w};\cV} & \leq C_1 \left( 2 \sigma_{k}(\x)_{1,\bm{w};\cV}
+\frac{\cG(\tilde{\x}) - \cG(\x) }{\lambda}  \right)+ \left( \dfrac{C_1}{\lambda} +C_2 \sqrt{k} \right) \nm{\e}_{2;\cV} ,
\\
\nm{\tilde{\x}-\x}_{2;\cV} & \leq  \dfrac{C_1'}{\sqrt{k}} \left( 2 {\sigma_{k}(\x)_{1,\bm{w};\cV}} 
+\dfrac{\cG(\tilde{\x}) - \cG(\x) }{ \lambda} \right) + \left( \dfrac{C'_1}{\sqrt{k}\lambda} + C'_2 \right)\nm{\e}_{2;\cV} ,
}
for any $\tilde{\x} \in \cW^N$,where the constants are given by
\begin{equation*}
C_1 = \dfrac{(1+\rho)}{(1-\rho)}, \quad C_2 =  \dfrac{2\gamma }{(1-\rho)}, \quad C_1' = \dfrac{(1+\rho)^2}{1-\rho}  \quad \text{ and } \quad C_2' = \dfrac{(3+\rho) \gamma}{1-\rho}.
\end{equation*}
\end{lemma}

Notice that defining the weighted rNSP over a closed subspace $\cW \subseteq \cV$ allows us to assert error bounds for the minimization problem over, for example, the space $\cW = \cV_K$. This fact will be useful later in the proofs.

\subsection{Matrices satisfying the weighted rNSP over Banach spaces}

We now derive explicit conditions for a matrix $\bm{A} \in \bbR^{m \times N}$ of the form \R{def-measMatrix} to give rise to an associated operator $\bm{A} \in \cB(\cV^N,\cV^m)$ (see \eqref{def-inducedBLO}) that satisfies the weighted rNSP over $\cV$.  
\begin{lemma}
[Weighted rNSP over $\bbR$ implies weighted rNSP over $\cV$]
\label{l:from_R_to_V_rNSP}
Suppose that a matrix  $\bm{A} \in \bbR^{m \times N}$ satisfies the weighted rNSP over $\bbR$ of order $(k,\bm{w})$ with    $0 \leq \rho < 1$ 
and $\gamma \geq 0$, and let $s^*=s^*(k) := \max \{ | S| : |S|_{\bm{w}} \leq k, \, S \subseteq [N] \}$. 
Then  the  corresponding operator $\bm{A} \in \cB(\cV^N,\cV^m)$ defined by \eqref{def-inducedBLO}
satisfies the weighted rNSP over $\cV$ of order $(k,\bm{w})$ with  constants $0 \leq \rho' < 1$ and $\gamma' >0$ given by $ \rho' = \sqrt{s^*}\rho$ and  $\gamma ' = \sqrt{s^*} \gamma $, respectively. 
\end{lemma}
\begin{proof}
Let $\bm{v}\in \cV^N $ and $|S|_{\bm{w}} \leq k$. Using   \eqref{norm_Vector} we get 
\begin{equation}\label{eq_rnspBA}
\| \bm{v}_{S}\|_{2;\cV}^2 
= \sum_{i \in S} \| v_i\|_{\cV}^2 
= \sum_{i \in S}  \left( \max_{\overset{v^* \in \cV^*}{\|v^* \|_{\cV^*} =  1}} |v^* (v_i)|  \right)^2 
\leq |S| \left( \max_{\overset{v^* \in \cV^*}{\|v^* \|_{\cV^*} = 1}} \nm{ v^* (\bm{v}_S)}_{2}  \right)^2, 
\end{equation}
where $v^*(\bm{z}) = (v^*(z_i))^{N}_{i=1}$ for $\bm{z} = (z_i)^{N}_{i=1} \in \cV^N$.
Since $\bm{A}$ has the weighted rNSP over $\bbR $ we get
\begin{equation}\label{phi_bound}
\nm{ v^* (\bm{v}_S)}_{2} \leq  \dfrac{\rho \nm{v^* (\bm{v}_{S^c})}_{1,\bm{w}} }{\sqrt{k}}+\gamma \nm{\bm{A} (v^* (\bm{v}))}_{2} ,\quad \forall \bm{v}\in \cV^N, \forall v^*\in \cV^*.
\end{equation}
If $\nm{v^*}_{\cV^*} = 1$, then we also have
\begin{equation*}
\nm{v^* (\bm{v}_{S^c})}_{1,\bm{w}}  \leq \nm{\bm{v}_{S^c}}_{1,\bm{w};\cV} \quad \text{ and } \quad  \nm{\bm{A} (v^* (\bm{v}))}_{2} = \nm{v^* (\bm{A} (\bm{v}))}_{2}  \leq \nm{\bm{A} (\bm{v})}_{2;\cV}.
 \end{equation*} 
Therefore
\begin{equation*}
\nm{ v^* (\bm{v}_S)}_{2}  \leq  \dfrac{\rho \nm{\bm{v}_{S^c}}_{1,\bm{w};\cV} }{\sqrt{k}}+\gamma \nm{\bm{A} (\bm{v})}_{2;\cV} ,\quad \forall \bm{v}\in \cV^N,\ \forall v^*  \in \cV^*, \nm{v^* }_{\cV^*} = 1.
\end{equation*}
Combining this with \eqref{eq_rnspBA} and noting that $|S| \leq s^*(k)$, we deduce that
\begin{equation*}
\| \bm{v}_{S}\|_{2;\cV}  =  \sqrt{{s^*(k)}}   \left(  \dfrac{\rho \nm{\bm{v}_{S^c}}_{1,\bm{w};\cV} }{\sqrt{k}}+\gamma \nm{\bm{A} (\bm{v})}_{2;\cV}   \right),\quad \forall \bm{v}\in \cV^N,
\end{equation*}
as required.
\end{proof}

In order to show that the matrix \R{def-measMatrix}, and therefore the resulting linear operator, {satisfies} the weighted rNSP, we follow the standard approach and first show that it satisfies the weighted restricted isometry property. We say that a matrix $\bm{A} \in \bbR^{m \times N}$ has the \textit{weighted Restricted Isometry Property (RIP) over $\bbR$} of order $(k,\bm{w})$ if there exists a constant $0\leq \delta <1$ such that
\begin{equation}
(1-\delta) \nm{\bm{z}}_{2}^2 \leq \nm{\bm{A} \bm{z}}_{2}^2 \leq (1+\delta) \nm{\bm{z}}_{2}^2,\quad \forall \bm{z}\in \Sigma_{k,\bm{w}}.
\end{equation}
The smallest constant such that this property holds is called  the $k$th \textit{weighted restricted isometry constant} of $\bm{A}$ and is denoted by $\delta_{k,\bm{w}}$. See, e.g., \cite[Def.~6.25]{adcock2021compressive} and \cite[Sec.~4.2]{Rauhut2016}.

As an alternative to the equivalence shown in \cite[Lem.~7.5]{adcock2022efficient} between the weighted RIP over $\bbR$ and the 
weighted RIP over a Hilbert space, Lemma \ref{l:from_R_to_V_rNSP} shows that the weighted RIP over $\bbR$ is a sufficient condition for the  weighted RIP over a Banach space. The extra factor $\sqrt{s^*}$ is {one of the causes} of the extra $m$-dependent factors {-- the other being the absence of Parseval's identity in Banach spaces --} in the final error bound for the Banach-valued case as opposed to the Hilbert-valued case. 

We also need the following result, which shows that the weighted RIP over $\bbR$ implies the weighted rNSP over $\bbR$. See, e.g., \cite[Thm.~6.26]{Adcock2022}.
\begin{lemma}[weighted RIP implies the weighted rNSP]
\label{l:implies-RIP} Let $\bm{w}> \bm{0}$, $k > 0$ and suppose that $\bm{A} \in \bbR^{m\times N}$ has the weighted RIP over $\bbR$ of order $(2k,\bm{w})$ with constant $\delta_{2k,\bm{w}}< (2 \sqrt{2}-1)/7$. Then $\bm{A}$ has the weighted rNSP over  $\bbR$ of order $(k,\bm{w})$ with constants $\rho=2 \sqrt{2} \delta_{2k,\bm{w}} / (1-\delta_{2k,\bm{w}})$ and $\gamma=\sqrt{1+\delta_{2k,\bm{w}}} / (1-\delta_{2k,\bm{w}})$. 
\end{lemma}

We now assert conditions on $m$ under which the measurement matrix \eqref{def-measMatrix} satisfies the weighted RIP. For this, we use the following result, which is an adaptation of \cite[Thm.\ 2.14]{Brugiapaglia2021} (we make several small notational changes herein for consistency with the notation used in this paper; moreover, we replace the logarithmic factor $\log^2(k/\delta^2)$ with $\log^2(k/\delta)$, as revealed by an inspection of the proof).

\begin{lemma}[Weighted RIP for Chebyshev and Legendre polynomials]\label{l:LegMat_RIP_general}
Let $\varrho$ be the tensor-product uniform or Chebyshev measure on $\cU = [-1,1]^{\bbN}$, $\{ \Psi_{\bm{\nu}} \}_{\bm{\nu} \in \cF}$ be the corresponding tensor-product orthonormal Legendre  or {Chebyshev} polynomial basis of $L^2_{\varrho}(\cU)$, $\Lambda = \Lambda^{\mathsf{HCI}}_{n}$ be as in \eqref{HC_index_set_inf} for some $n \geq 1$ and $\bm{y}_1,\ldots,\bm{y}_m$ be drawn independently and identically from the measure $\varrho$. Let  $ c_0$ be a universal constant, $0 < \delta, \epsilon< 1$, $k \geq 1$ and suppose that   
\be{
\label{m-cond-for-wRIP}
m \geq   c_0 \cdot \delta ^{-2} \cdot k \cdot  \left(  \log^2(k/\delta) \cdot \log^{2}(\E n)+\log(2/\epsilon)\right).
}
Then, with probability at least $1-\epsilon$, the matrix $\bm{A}$ defined in \R{def-measMatrix} satisfies the weighted RIP over $\bbR$ of order $(k,\bm{u})$ with constant $\delta_{k,\bm{u}}$, where $\bm{u}$ are the intrinsic weights \eqref{weights_def}.
\end{lemma}
\begin{proof}
We let $N = |\Lambda|$. Then
\bes{
   \nm{\Psi_{\bm{\nu}_j}}_{L^{\infty}_{\varrho}(\cU)} = u_{\bm{\nu}_j},
} 
and therefore the condition $\|\Psi_{\bnu_j} \|_{L^{\infty}(\cU)} \leq u_{\bnu_j}$ required by \cite[Thm.~2.14]{Brugiapaglia2021} holds.
Now, condition \eqref{m-cond-for-wRIP} gives
\begin{equation*}
\begin{split}
c_0  \cdot \delta^{-2} \cdot k \cdot  (\log( \E  N) \cdot \log^2(k/\delta)) 
&\leq c_0 \cdot \delta^{-2} \cdot k \cdot  (\log( \E  N) \cdot \log^2(k/\delta) +  \log (2/\epsilon))\\
&\leq 4 \cdot c_0 \cdot  \delta^{-2}   \cdot k (  \log^2(\E  n) \cdot \log^2(k/\delta )+ \log(2/\epsilon) ) \leq m
\end{split}
\end{equation*}
after relabelling the constant $c_0$ in \R{m-cond-for-wRIP} as $4 c_0$. Here, in the last inequality, we used the estimate $\log(\E N) \leq 4\log^2(\E  n)$, which comes from \eqref{N_bound} and some basic algebra:
\bes{
\log(\E N) \leq \log \left ( \E^2 n^{2 + \log(n)/\log(2)} \right ) \leq \left ( 2 + \frac{\log(n)}{\log(2)} \right ) \log(\E n) \leq 4 \log^2(\E n){.}
}
Therefore, from \cite[Thm.~2.14]{Brugiapaglia2021} we obtain that with probability at least $1-2\exp (-c_1 \delta^{-2}m/k)$ the matrix $\bm{A}$ defined in   \R{def-measMatrix} satisfies the weighted RIP over $\bbR$ of order $(k,\bm{u})$ with constant $\delta_{k,\bm{u}}$ for some $c_1 > 0$. To conclude the result, we notice that
\bes{
m \geq c_1 \delta^{-2} \cdot k \cdot \log(2/\epsilon) \quad \Rightarrow \quad 2 \exp(-c_1\delta^2 m / k) \leq \epsilon.
}
Hence, replacing $c_0$ by $\max \{ c_0 , c_1 \}$ in \R{m-cond-for-wRIP}, we deduce that the result holds with probability at least $1-\epsilon$, as required.
\end{proof}

\begin{remark}\label{rmk:casek}
As formulated above, Lemma \ref{l:LegMat_RIP_general} only considers the case $k \geq 1$. However,  for $k < 1 \leq  \min_{\bm{\nu} \in \mathcal{F}} u_{\bm{\nu}}^2$, the set of weighted $(k,\bm{u})$-sparse vectors is empty. Therefore the RIP of order $(k,\bm{u})$ is trivially satisfied in this case.
\end{remark}

The previous result will be used in the case of unknown anisotropy. In the case of known anisotropy, we require a different argument in order to obtain the better scaling with respect to $m$ in Theorem \ref{t:mainthm2}. As we see later, this case can be analyzed by asserting conditions under which the relevant measurement matrix $\bm{A}$ defined in \R{def-measMatrix-known} has the rNSP of `full' order $k = |S|_{\bm{u}}$, where $S$ is the index set used to construct $\bm{A}$.

\lem{
[Weighted rNSP for Chebyshev and Legendre polynomials in the `full' case]
\label{LS-wrNSP}
Let $\varrho$ be the tensor-product uniform or Chebyshev measure on $\cU = [-1,1]^{\bbN}$, $\{ \Psi_{\bm{\nu}} \}_{\bm{\nu} \in \cF}$ be the corresponding tensor-product orthonormal Legendre  or {Chebyshev} polynomial basis of $L^2_{\varrho}(\cU)$, $S \subset \cF$ and $\bm{y}_1,\ldots,\bm{y}_m$ be drawn independently and identically from the measure $\varrho$. Let $0 < \delta,\epsilon< 1$, $k = |S|_{\bm{u}}$, where $\bm{u}$ are the intrinsic weights \eqref{weights_def}, and suppose that
\be{
\label{m-cond-for-wrNSP-full}
m \geq ((1-\delta)\log (1-\delta)+ \delta)^{-1}\cdot k \cdot \log (k/\epsilon).
}
Then, with probability at least $1-\epsilon$, the matrix $\bm{A} \in \bbR^{m \times s}$, $s = |S|$, defined in \R{def-measMatrix-known} satisfies the weighted rNSP over $\bbR$ of order $(k,\bm{u})$ with constants $\rho = 0$ and $\gamma \leq (1-\delta)^{-1/2}$.
}
\prf{
Since $k = |S|_{\bm{u}}$ is the weighted cardinality of the `full' index set $S$, the wrNSP is equivalent to the condition
\bes{
\nm{\bm{x}}_{2} \leq \gamma \nm{\bm{A} \bm{x}}_2,\quad \forall \bm{x} \in \bbR^s.
}
Define the space $\cP = \spn \{ \Psi_{\bm{\nu}} : \bm{\nu} \in S \} \subset L^2_{\varrho}(\cU)$. Then by Parseval's identity, this inequality is equivalent to
\bes{
\nm{p}_{L^2_{\varrho}(\cU)} \leq \gamma \nm{p}_{\mathsf{disc}},\quad \forall p \in \cP,
}
where $\nm{p}_{\mathsf{disc}} = \sqrt{\frac1m \sum^{m}_{i=1} |p(\bm{y}_i) |^2}$. Thus, by \cite[\S 5.2]{Adcock2022} and \cite[Thm.\ 5.7]{Adcock2022}, we have $\gamma \leq (1-\delta)^{-1/2}$, provided
\be{
\label{m-for-disc-stab-const}
m \geq ((1-\delta)\log (1-\delta)+ \delta)^{-1} \cdot \kappa \cdot \log(s/\epsilon),
}
where $s = |S|$ and $\kappa = \kappa(\cP)$ is defined by (see \cite[Eqn.\ (5.15)]{Adcock2022}) 
\bes{
\kappa = \sup_{\bm{y} \in \cU} \sum_{\bm{\nu} \in S} | \Psi_{\bm{\nu}}(\bm{y} ) |^2.
}
Observe that $\kappa \leq \sum_{\bm{\nu} \in S} u^2_{\bm{\nu}} = |S|_{\bm{u}} = k$ and $s = |S| \leq |S|_{\bm{u}} = k$, since $\bm{u} \geq \bm{1}$. Hence \R{m-for-disc-stab-const} is implied by \R{m-cond-for-wrNSP-full}. This gives the result.
}

\section{Deep neural network approximation }\label{sec:DNN-approx}
 
In this section we detail the second key component of our proofs, which is the approximation of the orthonormal polynomials $\Psi_{\bm{\nu}}$ by DNNs.

\subsection{Approximate multiplication via DNNs}
{Our results are based on three different different DNN architectures (tanh, ReLU and RePU) that emulate the product of $n$ numbers. The first two follow from \cite{Ryck2021tanh} and  \cite{opschoor2019exponential}, respectively. The third is based on \cite{Li2019} and \cite{opschoor2019exponential}.}
\begin{lemma}[Approximate multiplication of $n$ numbers by ReLU and tanh DNNs] \label{lemma_dnn} Let $0 < \delta <1$, $n \in \bbN$ and consider constants $M_1,\ldots,M_n > 0$.
Then there exists a ReLU  ($j=1$) or a tanh ($j=0$) DNN $\Phi_{\delta}^j:  \prod^{n}_{i=1} [-M_i,M_i] \rightarrow \bbR $ satisfying
\begin{equation}\label{eq:ReLUapprox_delta}
\sup_{ |x_i|  \leq M_i } \left | \prod_{i=1}^n x_i - \Phi_{\delta}^j (\x) \right | \leq \delta,\qquad \text{where }\bm{x} = (x_i)^{n}_{i=1},
\end{equation}
for $j \in \lbrace 0,1 \rbrace$. The width and depth are bounded by 
\begin{align*}
     \mathrm{width}(\Phi_{\delta}^{1}  )  &\leq   c_{1,1}     \cdot n,   \\
    \mathrm{depth}( \Phi_{\delta}^{1}  )  & \leq c_{1,2}     \Big(1+ \log(n) \Big[ \log(n \delta^{-1})+ \log(M)   \Big]\Big),
  \end{align*}
in the ReLU case, where $M= \prod_{i}^n M_i$, and 
\begin{equation*}
  \mathrm{width}(\Phi_{\delta}^0  )   \leq c_{1,1}  \cdot n,  \quad
  \mathrm{depth}( \Phi_{\delta}^0  )    \leq c_{1,2}  \cdot  \log_2(n),
  \end{equation*}
  in the tanh case. Here $c_{j,1}$, $c_{j,2}$ are universal constants for $j \in \lbrace 0,1 \rbrace$.
\end{lemma}
\begin{proof}
First, in the tanh case, let $N \geq \max_{i \in [n]} \{ M_i\}$ be such that $\x \in [-N,N]^n$. Then the result in the tanh case is a direct application of \cite[Lem.\ 3.8]{Ryck2021tanh}. We now focus on the ReLU case.
Let $\bm{x} = (x_i)^{n}_{i=1}$ be such that  $|x_i| \leq M_i$ for all $i \in [n]$. Now, write the multiplication of these terms  as
\begin{equation*}
\prod_{i=1}^{n} x_i = \left( \prod_{i=1}^{n} x_i/M_i \right)  \cdot \prod_{i=1}^{n} M_i  = \left( \prod_{i=1}^{n} x_i/M_i \right)  \cdot M.
\end{equation*}
Then, using  \cite[Prop.\ 2.6]{opschoor2019exponential}, for any $\widetilde{\delta}\leq  \delta$  there exists  a ReLU DNN  $\Phi_{\widetilde{\delta},1}$ such that
\begin{equation*}
\left|M \left( \prod_{i=1}^{n} \frac{x_i}{M_i} \right) -  M \cdot\Phi_{\widetilde{\delta},1}\left(\frac{x_1}{M_1},\ldots ,\frac{x_n}{M_n}\right) \right| \leq M\widetilde{\delta}.
\end{equation*}
We now set $\widetilde{\delta} = \delta/M$. Since the composition  of   affine maps is an affine map, 
 we can define a ReLU DNN of the same architecture as 
\begin{equation}
\widetilde{\Phi}_{\delta,M}(x_1,\ldots,x_n)=  M \cdot\Phi_{ {\delta}/M,1}\left(\frac{x_1}{M_1},\ldots ,\frac{x_n}{M_n}\right).
\end{equation}
This implies that
\begin{equation*}
\left|  \prod_{i=1}^{n} {x_i} - \widetilde{\Phi}_{\delta,M}(x_1,\ldots,x_n) \right| \leq  {\delta}.
\end{equation*}
Taking the supremum over $|x_i| \leq M_i$ we obtain \eqref{eq:ReLUapprox_delta}. We now bound the width and depth. From \cite[Prop.\ 2.6]{opschoor2019exponential} notice that there exists  a  constants $c_2>0$ such that
\begin{align*}
\mathrm{depth}(\widetilde{\Phi}_{\delta,M})  \leq  \mathrm{depth}(\Phi_{\delta/M,1})  
&\leq c_2 \left( 1 + \log(n) \log(n M/\delta) \right).
\end{align*}
Now, inspecting the proof of \cite[Prop.\ 2.6]{opschoor2019exponential}, which is based on \cite[Sec.\ 3.2]{Schwab2017} and constructs the DNN for multiplying $n$ numbers as a binary tree, we observe that the product of two numbers involves a maximum of $12$ nodes per layer. Thus, for the product of $n$ numbers, the width is bounded by   
\begin{equation*}
\mathrm{width}(\widetilde{\Phi}_{\delta,M})  \leq  12 \left \lceil \frac{n}{2} \right \rceil.
\end{equation*}
This completes the proof.
\end{proof}
The following lemma asserts the existence of a RePU DNN to calculate the multiplication of two numbers. Its proof can be found in \cite[Lem.\ 2.1]{Li2019} (see also \cite[Appx.\ A]{opschoor2019exponential} and \cite[Thm.\ 2.5]{Li2019}).

\begin{lemma}[Exact multiplication of two numbers by a RePU DNN]\label{l:multi_twoRepU}
For $\ell=2,3,\ldots$, there exists  a RePU neural network $ \bar{\Phi}^{{\ell}}:\bbR^2 \rightarrow \bbR$ such that
\begin{equation*}
\bar{\Phi}^{{\ell}}(x_1,x_2) = x_1 x_2 ,\quad \forall x_1,x_2 \in \bbR.
\end{equation*}
The  width  and depth of this DNN are 
\begin{equation*}
\mathrm{width}(\bar{\Phi}^{{\ell}}) = c_{\ell}, \quad \mathrm{depth}(\bar{\Phi}^{{\ell}}) =1,
\end{equation*}
where  $c_{\ell}$ is a constant depending on $\ell$.
\end{lemma}

We now use similar arguments to those in \cite[Sec.\ 2.3]{opschoor2019exponential}, and notice that the previous lemma implies the existence of RePU neural network for multiplying $n$ different numbers. The next lemma and its proof are based on \cite[Prop.\ 2.6]{opschoor2019exponential}, which, in turn, employs techniques from \cite[Prop.\ 3.3]{Schwab2017}. Basically, the idea here is to construct a neural network $\Phi^{{\ell}}$ as a binary tree of $\bar{\Phi}^{{\ell}}$-networks using Lemma \ref{l:multi_twoRepU}. Unlike in the ReLU and tanh cases (see Lemma \ref{lemma_dnn}), the resulting multiplication is exact and we do not require the assumption $|x_i| \leq M_i$, $i \in [n]$.

\begin{lemma}[Exact multiplication of $n$ numbers by RePU] \label{l:multi_nRepU} 
For $\ell=2,3,\ldots$, there exists a RePU neural network   $\Phi^{{\ell}}:  \bbR^n \rightarrow \bbR $
such that 
\begin{equation}\label{eq:approx_RePu}
  \Phi^{{\ell}} (\x)  =\prod_{i=1}^n x_i ,\quad \forall \x=(x_i)_{i=1}^n \in \bbR^n.
\end{equation}
The width and depth are bounded by
\begin{equation*}
\mathrm{width}( \Phi^{{\ell}} ) \leq c_{ {\ell},1}  \cdot n, \quad \mathrm{depth}( \Phi^{{\ell}})  \leq c_{2}\log_2(n), 
\end{equation*}
where  $c_{{\ell},1}$ and $c_{2}$ are positive constants and only $c_{{\ell},1}$ depends on $\ell$.
\end{lemma}
\begin{proof}
 Let $\widetilde{n}:=  \min \{ 2^k: k \in \bbN, 2^k \geq n \}$. For every $\x=(x_1,\ldots,x_n) \in \bbR^n$ we  define the  vector   $\widetilde{\bm{x}}=(x_1,\ldots,x_n,x_{n+1}, \ldots, x_{\widetilde{n}}) \in \bbR^{\widetilde{n}}$, where $x_{n+1}= \ldots= x_{\widetilde{n}}= 1 $.  Observe that the map $\bm{x} \mapsto \widetilde{\bm{x}}$ is affine and, hence, can be implemented by a suitable choice of weights and biases in the first layer.
Arguing as in \cite[Prop.\ 2.6]{opschoor2019exponential}, let $l \in \bbN$, consider vectors   in $ \bbR^{2l}$ and define the mapping 
\begin{equation*}
R_l(z_1,\ldots, z_{2l}):= \left( \bar{\Phi}^{ {\ell}}(z_1,z_2),\ldots,\bar{\Phi}^{ {\ell}}(z_{2l-1},z_{2l})\right) \in \bbR^{l},
\end{equation*}
where $\bar{\Phi}^\ell$ is as in Lemma~\ref{l:multi_twoRepU}.  Now,  for $k \in \bbN$  we consider the composition 
\begin{equation}\label{eq:def_R}
 \cR^{k}:=  R_{1} \circ R_{2} \circ R_{2^2} \circ \ldots \circ R_{2^{k-1}}.
 \end{equation} 
 Keeping this in mind we now define ${\Phi}^{  {\ell}}: \bbR^n \rightarrow \bbR$ by
\begin{equation*}
{\Phi}^{ {\ell}} (x_1,\ldots, x_n):= \cR^{\log_2(\widetilde{n})}(x_1,\ldots,x_{\widetilde{n}}).
\end{equation*}
We immediately deduce \eqref{eq:approx_RePu}. 
We now estimate the depth and width of ${\Phi}^{ {\ell}} $. First, note that $\widetilde{n}\leq 2n$. Then, following  \cite[Prop.\ 2.6]{opschoor2019exponential} 
and using Lemma \ref{l:multi_twoRepU}, there exists a positive constant $c_2$  such that 
\begin{equation*}
\mathrm{depth}({\Phi}^{ {\ell}})  \leq c_2   \log_2(n).
\end{equation*}
Moreover, from   Lemma \ref{l:multi_twoRepU} and \cite[Lem.\ 2.5]{Li2019},  the construction of the neural network implies the existence of  a constant $c_{ {\ell}}  >0$   depending on $\ell$ such that
\begin{equation*}
\mathrm{width}({\Phi}^{ {\ell}}) \leq c_{ {\ell}} \cdot  n.  
\end{equation*}
This gives the result. 
\end{proof}

\subsection{Emulation of orthogonal polynomials via DNNs}
{Having shown that neural networks can emulate the multiplication of $n$ numbers, we are now able to emulate orthonormal polynomials. We do this by appealing to the fundamental theorem of algebra, which allows us to represent any such polynomial as a product of its roots. This approach is based on \cite{Daws2019b}, which considered only ReLU DNNs and Legendre polynomials. This approach differs from other constructions (see, e.g., \cite[Prop.\ 2.13]{opschoor2019exponential}) which first emulate univariate orthogonal polynomials and then use the previously derived multiplication networks.}

\begin{theorem}
\label{Prop_Ex_NN}  Let $\Lambda \subset \cF$ be a finite multi-index set, $m(\Lambda) = \max_{\bnu \in \Lambda} \|\bnu\|_{1}$ and $\Theta \subset \bbN$, $|\Theta|=n$, satisfy
\bes{
\bigcup_{\bm{\nu} \in \Lambda} \supp(\bm{\nu}) \subseteq \Theta.
}
Let  $\{\Psi_{\bnu}\}_{\bnu \in \cF} \subset L^2_{\varrho}(\cU)$ be the orthonormal Legendre or Chebyshev polynomial basis of $L^2_{\varrho}(\cU)$. Then for  every $0< \delta <1$ there exists a ReLU  ($j=1$), RePU ($j= {\ell}$) or tanh ($j=0$) DNN $\Phi_{\Lambda,\delta}^{j}: \bbR^n \rightarrow \bbR^{|\Lambda|}$,  such that, if $\Phi_{\Lambda,\delta}^j(\bm{z})= (\Phi_{\bnu,\delta}^j(\bm{z}))_{\bnu \in \Lambda}$, $\bm{z}=(z_j)_{j \in \Theta} \in \bbR^n$  and $\cT_{\Theta}$ is as in \eqref{def:cT}, then  
\begin{equation*}
\|\Psi_{\bnu} -\Phi_{\bnu, \delta}^j \circ  \cT_{\Theta}  \|_{L^\infty(\cU)} \leq \delta, \quad \forall \bnu \in \Lambda, \quad  j \in \lbrace 0,1,{\ell} \rbrace.
\end{equation*}
 In the case of the ReLU ($j = 1$) activation function,  the width and depth of this network satisfy  
  \begin{align*}
   \mathrm{width}(\Phi_{\bm{\nu}}^{ 1}  ) & \leq   c_{ 1,1}   \cdot   |\Lambda| \cdot  m(\Lambda),\\
    \mathrm{depth}( \Phi_{\bm{\nu}}^{ 1}  )  & \leq c_{ 1,2}    \cdot   \Big(1+ \log(m(\Lambda))  \cdot \Big[ \log(m(\Lambda) \delta^{-1})+m(\Lambda) +n\Big]\Big).
\end{align*}
 In the case of the RePU ($j =  {\ell}$) or tanh ($j = 0$) activation function,  the width and depth of this network satisfy
  \begin{align*}
  \mathrm{width}(\Phi_{\Lambda,\delta}^j  )  \leq c_{j,1}  \cdot |\Lambda|  \cdot m(\Lambda),  \quad 
  \mathrm{depth}(\Phi_{\Lambda,\delta}^j  ) \leq c_{j,2}  \cdot    \log_2(m(\Lambda)).
\end{align*}
Here $c_{j,1}$, $c_{j,2}$ are universal constants in the ReLU and tanh cases, with only $c_{{\ell},1}$ depending on $\ell = 2,3,\ldots$ in the RePU case.
\end{theorem}
 \begin{proof} We divide the proof into two cases.
 
\pbk
\textit{Case 1: Legendre polynomials.} The univariate Legendre polynomials $\{ \Psi_{\nu} \}_{\nu \in \bbN_0}$ are given by (see, e.g., \cite[Sec.\ 2.2.2]{Adcock2022})
\begin{equation}
\label{Leg-def}
P_{\nu}(y)= \dfrac{1}{2^{\nu} \nu!} \dfrac{\D^{\nu}}{\D y^{\nu}} (y^2-1)^{\nu}  
\quad \text{ and } \quad
 \Psi_{\nu}(y) =  \sqrt{2 \nu+1}  P_{\nu}(y), \quad \forall \nu \in \bbN_0.
\end{equation}
Hence, their multivariate counterparts can be written as 
\begin{equation}\label{eq_poly_Leg}
  \Psi_{\bnu}(\y) =   \prod_{i \in \supp{(\bnu})} \sqrt{2 \nu_i+1}  P_{\nu_i}(y_i),\quad \forall \y \in \cU,\ \bm{\nu} \in \cF,
 \end{equation}
where $\supp(\bnu)$ is as in \R{l0norm}.
Using the fundamental theorem of algebra we may write
 \begin{align}\label{eq:poly_Leg_facts}
  \Psi_{\bnu}(\y) &=    \prod_{i \in \supp (\bnu)} \prod^{\nu_i}_{j=1} \left ( \sqrt{2 \nu_i+1}  d_{\nu_i} \right )^{1/\nu_i}  (y_i - r^{(\nu_i)}_j), \quad \forall {\y} \in  {\cU}, \quad   \bnu\in \cF.
 \end{align}
Here,   $\lbrace {r}_j^{(\nu_i)}\rbrace_{j=1}^{\nu_i}$ are the $\nu_i$ roots of the polynomial $P_{\nu_i}$  and $d_{\nu_i}$ is a scaling factor. Using  \R{Leg-def}, we see that the leading coefficient of $P_{\nu}$ is  $d_{\nu} =  2^{-\nu} \frac{(2\nu)!}{(\nu!)^2}$. Then, by Stirling's formula for factorials $ \sqrt{2 \pi} n^{n+1/2}\E ^{-n} \leq n! \leq \E  n ^{n+1/2}\E ^{-n}$,  this coefficient satisfies
\[
d_0=1 \quad \text{ and } \quad d_{\nu} \leq  \dfrac{\E 2^{\nu}}{\pi \sqrt{2 \nu}}, \quad \nu \in \bbN.
\]
Next,  for $\bnu \in \Lambda$, we define the affine map $\cA_{\bnu}: \bbR^n  \rightarrow \bbR^{\nm{\bnu}_1}$, $\cA_{\bnu} (\bm{y}) = (a_{i,j}(\y))_{i\in \supp(\bnu),j \in [\nu_i]}$, by
\begin{equation}\label{eq:A_y_nu}
a_{i,j}(\y)=  \left ( \sqrt{2 \nu_i+1}  d_{\nu_i} \right )^{1/\nu_i} (y_i - r^{(\nu_i)}_j) ,\quad i \in \supp(\bnu), \ j \in [\nu_i], \qquad \forall \y \in \cU.
\end{equation}
Given $\bnu \in \Lambda$, this allows us to define the product of the $ \nm{\bnu}_1$ terms \R{eq:poly_Leg_facts} that comprise $\Psi_{\bnu}$. It is useful, however, to make the number of factors in this multiplication constant for all $\bnu \in \Lambda$. To this end, we now
 redefine the affine map  $\cA_{\bnu}: \bbR^n  \rightarrow \bbR^{m(\Lambda)}$ by padding the output vector with  $m(\Lambda) - \nm{\bm{\nu}}_1$ terms equal to one. 

Our aim now is to apply Lemma \ref{lemma_dnn} to show that there exists  a DNN that approximates the multiplication of the factors in $\cA_{\bnu}(\y)$. 
To do so, we need to identify bounds $M_k$ for each of the terms in the output vector $\cA_{\bm{\nu}}(\bm{y})$, with $k \in [m(\Lambda)]$. First, notice that the roots of the Legendre polynomials are in $[-1,1]$ for \eqref{eq:poly_Leg_facts}. Then each factor in  \eqref{eq:A_y_nu} is bounded  by
\[
\widetilde{M}_{i}:= 2 (\sqrt{2 \nu_{i}+1})^{1/\nu_{i}} \left(\dfrac{\E  2^{\nu_i}}{\pi \sqrt{2 \nu_i}}\right)^{1/\nu_i}, \quad i \in \supp(\bnu),\  j \in [\nu_i].
\]
Clearly, the other terms in $\cA_{\bm{\nu}}(\bm{y})$ are bounded by one. Therefore, since $|\supp(\bnu)| \leq |\Theta|=n$  we get
\begin{equation*}
\begin{split}
M =  \prod_{k=1}^{m(\Lambda)} M_{k} &= \left(\prod_{i \in \supp(\bnu)} \prod_{j =1}^{ \nu_i} \widetilde{M}_{i} \right) \cdot \left( \prod_{k = 1}^{m(\Lambda)-\|\bnu\|_1} 1 \right) 
 \\
 & = \prod_{i \in \supp (\bnu)} \prod^{\nu_i}_{j=1} 2 (\sqrt{2 \nu_{i}+1})^{1/\nu_{i}} \left(\dfrac{\E  2^{\nu_i}}{\pi \sqrt{2 \nu_i}}\right)^{1/\nu_i} \\ 
&\leq  2^{2\nm{\bm{\nu}}_1} \left(\dfrac{\E }{\pi}\right)^n \prod_{i \in \supp (\bnu)}  \sqrt{ \dfrac{2\nu_{i}+1}{2 \nu_i} } 
\\
 & \leq       2^{2 m(\Lambda)}\left( \dfrac{\E }{\pi}\right)^n\left( \dfrac{3}{2}\right)^{n/2} .
\end{split}
\end{equation*}
This defines the multiplication  of $m(\Lambda)$ factors. Using Lemma \ref{lemma_dnn} (or Lemma \ref{l:multi_nRepU} in the case of RePU, in which case the previous calculation is unnecessary), for any $0 < \delta <1$ there exists a ReLU ($j=1$), a RePU ($j=  {\ell}$) or a tanh ($j=0$) DNN  $\Phi_{\delta,M,\bnu}^{j}$ that approximates  the multiplication of the $m(\Lambda)$ factors defining $\cA_{\bnu}$.   Thus, we define the DNN $\Phi_{\bm{\nu}}^{j} :  \bbR^n \rightarrow \bbR$  by
\[
\Phi_{\bm{\nu}}^j = \Phi_{\delta,M,\bnu}^j \circ {\cA}_{\bnu}, \quad j \in \lbrace 0,1,{\ell} \rbrace .
\]
By construction, we have
\[
\nmu{\Psi_{\bm{\nu}} - \Phi_{\bm{\nu}}^j \circ  \cT_{\Theta} }_{L^{\infty}_{\varrho}(\cU)} \leq \delta, \quad j \in \lbrace 0,1,{\ell}\rbrace, \quad \forall \bnu \in \Lambda.
\]
Applying the bound for $M$, 
and some basic algebra we obtain the respective bounds for the width and depth of each $\Phi_{\bm{\nu}}^j$. Specifically,
  \begin{align*}
   \mathrm{width}(\Phi_{\bm{\nu}}^{ 1}  ) & \leq   c_{ 1,1}    \cdot m(\Lambda),\\
    \mathrm{depth}( \Phi_{\bm{\nu}}^{ 1}  )  & \leq c_{1,2}    \cdot   \Big(1+ \log(m(\Lambda))  \cdot\Big[ \log(m(\Lambda) \delta^{-1})+ m(\Lambda)+n \Big]\Big){,}
\end{align*}
in the ReLU case, and
\begin{equation*}
    \mathrm{width}(\Phi_{\bm{\nu}}^j  )   \leq c_{j,1}  \cdot m(\Lambda), \qquad
    \mathrm{depth}( \Phi_{\bm{\nu}}^j  )   \leq c_{j,2}  \cdot   \log_2(m(\Lambda))
  \end{equation*} 
  otherwise. Observe that we have found DNNs $\Phi^{j}_{\bm{\nu}}$ of the same depth that approximate each polynomial $\Psi_{\bnu}$ for $\bnu \in \Lambda$. We consider now the DNN formed by vertically stacking these DNNs, i.e., $\Phi_{\Lambda,\delta}^j= (\Phi_{\bm{\nu}}^j  )_{\bnu \in \Lambda}$. It follows immediately that the depth and  width of this DNN satisfy
  \begin{align*}
   \mathrm{width}(\Phi_{\bm{\nu}}^{1}  ) & \leq   c_{ 1,1}   \cdot  |\Lambda|  \cdot m(\Lambda),\\
    \mathrm{depth}( \Phi_{\bm{\nu}}^{1}  )  & \leq c_{ 1,2}    \cdot   \Big(1+ \log(m(\Lambda)) \cdot \Big[ \log(m(\Lambda) \delta^{-1})+ m(\Lambda) +n\Big]\Big),
\end{align*}
in the ReLU case, and
\begin{equation*}
    \mathrm{width}(\Phi_{\bm{\nu}}^j  )   \leq c_{j,1}  \cdot |\Lambda| \cdot m(\Lambda), \qquad
    \mathrm{depth}( \Phi_{\bm{\nu}}^j  )   \leq c_{j,2}  \cdot   \log_2(m(\Lambda)).
  \end{equation*} 
in the RePU and tanh cases. This completes the proof for the Legendre polynomials.

\pbk
\textit{Case 2: Chebyshev polynomials.} The orthonormal Chebyshev polynomials are defined by
\begin{equation}\label{eq_poly_c}
  \Psi_{\bnu}(\y) =  2^{\nm{\bnu}_0/2} \prod_{i \in \supp(\bnu)} \cos(\nu_i \mathrm{arccos}(y_i)),\quad \forall \y \in \cU,\ \bnu \in \cF.
 \end{equation}
We can write each factor as a product over the roots of the polynomials $\cos(\nu_i \mathrm{arccos}(y_i))$, to give 
 \begin{align}\label{eq:poly_Leg_facts_cheb}
     \Psi_{\bnu}(\y)  &=  \prod_{i \in \supp(\bnu)} \prod^{\nu_i}_{j=1} \left ( 2^{1 / 2} d_{\nu_i} \right )^{1/\nu_i}  (y_i - r^{(\nu_i)}_j), \quad \forall {\y} \in  {\cU}, \ \bnu \in \cF.
 \end{align}
Define the affine mapping   $\cA_{\bnu}: \bbR^n  \rightarrow \bbR^{m(\Lambda)}$ with entries
\begin{equation*}
a_{i,j}(\y)=   \left ( 2^{1/ 2} d_{\nu_i} \right )^{1/\nu_i}   (y_i - r^{(\nu_i)}_j) ,\quad i \in \supp(\bnu), j \in [\nu_i]{,}
\end{equation*}
where  $d_{\nu} = 2^{\nu - 1}$, and the remaining $m(\Lambda)- \|\bnu\|_1$ entries being equal to one. As in the previous case, the roots $r^{(\nu_i)}_j \in [-1,1]$. Hence we define $M$ as  
\[
M = \prod_{i=1}^{m(\Lambda)}  M_{i}= \prod_{ i \in \supp(\bnu)} \prod^{\nu_i}_{j=1} \left ( 2^{1/ 2} d_{\nu_i} \right )^{1/\nu_i}2 = 2^{2\nm{\bm{\nu}}_1 - \nm{\bm{\nu}}_0/2}  \leq  2^{2 m(\Lambda)}.
\]
Then,  using same arguments as those for the Legendre case, by Lemma  \ref{l:multi_nRepU} and  Lemma \ref{lemma_dnn} for any $0 < \delta <1$ there exists a ReLU ($j= 1$), a RePU ($j=  {\ell}$) or a tanh ($j=0$) DNN  $\Phi_{\delta,\cM}^{j}$ that approximates  the multiplication of $m(\Lambda)$ factors defined    in \eqref{eq:poly_Leg_facts_cheb} with this specific choice of $M$.   Thus, we define the DNNs $\Phi_{\bm{\nu}}^{j} :  \bbR^n \rightarrow \bbR$  by
\[
\Phi_{\bm{\nu}}^j = \Phi_{\delta,{M}}^j \circ {\cA}_{\bnu}, \quad j \in \lbrace 0, 1, {\ell} \rbrace ,
\]
for which the following bound holds
\[
\nm{\Psi_{\bm{\nu}} - \Phi_{\bm{\nu}}^j \circ  \cT_{\Theta} }_{L^{\infty}_{\varrho}(\cU)} \leq \delta, \quad j \in \lbrace 0, 1, {\ell} \rbrace, \quad \forall \bnu \in \Lambda.
\]
We now obtain the result by using the bound for $M$ and the same arguments as in the Legendre case.
 \end{proof}

\section{Polynomial approximation rates for holomorphic functions}\label{s:inf-dim}

In order to derive the desired algebraic rates of convergence in our main theorems, we first need several results on the approximation of infinite-dimensional, holomorphic, Banach-valued functions via polynomials. See, e.g., \cite{Adcock2022,cohen2015high} for an in-depth discussion on this topic. 

The following theorem can be found in \cite[Thm.~A.3]{adcock2022efficient}, and is an amalgamation of various known results in the literature. For further details, see \cite[Rem.\ A.4]{adcock2022efficient} or \cite{Adcock2022,cohen2015high} and references therein.  Note that this result involves lower and anchored sets, which were introduced in Definition \ref{d:lower-anchored-set}.
\thm{
[Best $s$-term decay rates; infinite-dimensional case]
\label{thm:best_s-term_inf-dim}
Let $\varrho$ be the tensor-product uniform or Chebyshev measure on $\cU = [-1,1]^{\bbN}$ and $\{ \Psi_{\bm{\nu}} \}_{\bm{\nu} \in \cF}$ be the corresponding tensor-product orthonormal Legendre  or {Chebyshev} polynomial basis of $L^2_{\varrho}(\cU)$.
Let $0 < p <1$, $\varepsilon >0$, $\bm{b} \in \ell^p(\bbN)$ with $\bm{b} > \bm{0}$ and $f \in \cH(\bm{b},\varepsilon)$, where $\cH(\bm{b},\varepsilon)$ is as in \R{B-b-eps-def}. Let $\bm{c} = (c_{\bm{\nu}})_{\bm{\nu} \in \cF}$ be the Chebyshev or Legendre coefficients of $f$, as in \R{f-exp}. Then the following best $s$-term decay rates hold:
\begin{itemize}
\item[(i)] For any $p \leq q < \infty$ and $s \in \bbN$, there exists a lower set $S \subset \cF$ of size $|S| \leq s$ such that 
\bes{
\sigma_s(\bm{c})_{q;\mathcal{V}} \leq \nmu{\bm{c} - \bm{c}_S}_{q;\cV} \leq \nmu{\bm{c} - \bm{c}_S}_{q,\bm{u};\cV} \leq   C\cdot s^{1/q-1/p}{,}
}
where $\sigma_{s}(\bm{c})_{q;\cV}$ is as in \R{sigma-equiv} (with $\Lambda = \cF$), $\bm{u}$ is as in \R{weights_def} and $C = C (\bm{b},\varepsilon,p) >0$ depends on $\bm{b}$, $\varepsilon$ and $p$ only.
\item[(ii)] For any $p \leq q \leq 2$ and $k >0$, there exists a set $S \subset \cF$ with $|S|_{\bm{u}} \leq k$ such that
$$
\sigma_k(\bm{c})_{q,\bm{u};\mathcal{V}} \leq \nmu{\bm{c} - \bm{c}_S}_{q,\bm{u};\cV}
\leq C \cdot k^{1/q-1/p}, 
$$
where $\sigma_k(\bm{c})_{q,\bm{u};\mathcal{V}}$ is as in \R{sigma-k-w-equiv} (with $\Lambda = \cF$), $\bm{u}$ is as in \R{weights_def} and $C = C(\bm{b},\varepsilon,p) >0$ depends on $\bm{b}$, $\varepsilon$ and $p$ only.
\item[(iii)] Suppose that $\bm{b}$ is monotonically nonincreasing. Then, for any $p \leq q < \infty$ and $s \in \bbN$, there exists an anchored set $S \subset \cF$ of size $|S| \leq s$ such that 
\bes{
\sigma_s(\bm{c})_{q;\mathcal{V}} \leq \nmu{\bm{c} - \bm{c}_S}_{q;\cV} \leq \nmu{\bm{c} - \bm{c}_S}_{q,\bm{u};\cV} \leq   C\cdot s^{1/q-1/p}{,}
}
where $\sigma_{s}(\bm{c})_{q;\cV}$ is as in \R{sigma-equiv} (with $\Lambda = \cF$), $\bm{u}$ is as in \R{weights_def} and $C = C(\bm{b},\varepsilon,p) >0$ depends on $\bm{b}$, $\varepsilon$ and $p$ only.
\end{itemize}
}

Recall that in our main theorems, we consider the case where $\bm{b}$ belongs to the monotone $\ell^p$ space $\ell^p_{\mathsf{M}}(\bbN)$. For this we require the following corollary of Theorem \ref{thm:best_s-term_inf-dim}:
\cor{
\label{cor:anchored_case_lpM_space}
Let $0 < p <1$, $\varepsilon >0$, $\bm{b} \in \ell^p_{\mathsf{M}}(\bbN)$ and $f \in \cH(\bm{b},\varepsilon)$, where $\cH(\bm{b},\varepsilon)$ is as in \R{B-b-eps-def}. Let $\bm{c} = (c_{\bm{\nu}})_{\bm{\nu} \in \cF}$ be the Chebyshev or Legendre coefficients of $f$, as in \R{f-exp}. Then, for any $p \leq q < \infty$ and $s \in \bbN$, there exists an anchored set $S \subset \cF$ of size $|S| \leq s$ such that 
\bes{
\sigma_s(\bm{c})_{q;\mathcal{V}} \leq \nmu{\bm{c} - \bm{c}_S}_{q;\cV} \leq \nmu{\bm{c} - \bm{c}_S}_{q,\bm{u};\cV} \leq   C\cdot s^{1/q-1/p},
}
where $\sigma_{s}(\bm{c})_{q;\cV}$ is as in \R{sigma-equiv} (with $\Lambda = \cF$), $\bm{u}$ is as in \R{weights_def} and $C = (\bm{b},\varepsilon,p) >0$ depends on $\bm{b}$, $\varepsilon$ and $p$ only. 
}
\prf{Let $\tilde{\bm{b}}$ be the minimal monotone majorant of $\bm{b}$, defined in \R{min-mon-maj}. We first claim that
\be{
\label{set-inclusion}
\cR(\tilde{\bm{b}},\varepsilon) \subseteq \cR(\bm{b},\varepsilon),
}
where $\cR$ is as in \R{def_R_b_e}.
Let $\bm{\rho} \geq \bm{1}$ be such that
\bes{
\sum^{\infty}_{j=1} \left ( \frac{\rho_j + \rho^{-1}_j}{2} -  1 \right ) \tilde{b}_j \leq \varepsilon.
}
Then, since $\tilde{b}_j \geq b_j$ for all $j$, we have 
\bes{
\sum^{\infty}_{j=1} \left ( \frac{\rho_j + \rho^{-1}_j}{2} -  1 \right ) b_j \leq \varepsilon.
}
Thus,
\eas{
\cR(\tilde{\bm{b}},\varepsilon) &= \bigcup \left\lbrace \cE(\brho): \brho \geq \bm{1},\  \sum_{j=1}^{\infty} \left( \dfrac{\rho_j+\rho_j^{-1}}{2} -1 \right) \tilde{b}_j \leq  \varepsilon \right\rbrace
\\
& \subseteq  \bigcup \left\lbrace \cE(\brho): \brho \geq \bm{1},\  \sum_{j=1}^{\infty} \left( \dfrac{\rho_j+\rho_j^{-1}}{2} -1 \right) b_j \leq  \varepsilon \right\rbrace = \cR(\bm{b},\varepsilon),
}
as required.

Now let $f \in \cH(\bm{b},\varepsilon)$, i.e., $f$ is holomorphic in $\cR(\bm{b},\varepsilon)$ and satisfies $\nmu{f}_{L^{\infty}_{\varrho}(\cR(\bm{b},\varepsilon))} \leq 1$. It follows from \R{set-inclusion} that $f \in \cH(\tilde{\bm{b}},\varepsilon)$. Since $\tilde{\bm{b}}$ is monotonically nonincreasing and $\ell^p$-summable, part (iii) of Theorem \ref{thm:best_s-term_inf-dim} now immediately implies the result.
}

\rem{[Dependence of the set $S$ on $\bm{b}$, $\varepsilon$ only]
\label{rem:S-idpt-f}
As stated, the various sets $S$ described in these two results depend on the function $f$ being approximated. In fact, an inspection of the proofs of these results  (see the references mentioned above) reveals that they only depend on the holomorphy parameters  $\bm{b}$ and $\varepsilon$. This holds because the proofs rely on bounds for the polynomial coefficients $c_{\bm{\nu}}$ that depend on $\bm{\nu}$, $\bm{b}$ and $\varepsilon$ only. We will use this observation later in the proofs of Theorems \ref{t:mainthm2} and \ref{t:mainthm2H}. 
}

\section{Proofs of main theorems}\label{S:proofs}

We are now ready to prove Theorems \ref{t:mainthm1}-\ref{t:mainthm2H}. The general idea comes from the proof of the main results in \cite[Sec. B.4]{AdcockEtAl2021MSML}. We first show that the polynomial matrix $\bm{A} \in \bbR^{m \times N}$ has the wrNSP. Then, we show that its approximation via DNNs $\bm{A}' \in \bbR^{m \times N}$ has the wrNSP by using a perturbation result from \cite[Lem.\ 12]{AdcockEtAl2021MSML} (see also \cite[Lem.\ 8.5]{adcock2021compressive}). This implies that the operator $\bm{A}' \in \cB(\cV^N, \cV^m)$ has the wrNSP by Lemma \ref{l:LegMat_RIP_general}.   Next, after splitting the error into various terms, we use Lemma \ref{lemma-musuboptimal} in combination with the results in \S \ref{s:inf-dim} to derive the desired error bounds. Finally, we use the results of \S \ref{sec:DNN-approx} to estimate the depths and widths of the DNN architectures.

Throughout these proofs, we use the notation $a \lesssim b$ to mean that there exists a constant $c>0$ independent of $a$ and $b$ such that $a \leq cb$.

\subsection{Unknown anisotropy case}

We commence with the unknown anisotropy case.
\begin{proof} [Proof of Theorem \ref{t:mainthm1}] 
The proof is divided in several steps. 

\pbk
\textit{Step 1: Problem setup.} Let $\Theta$, $|\Theta|=n$ be as in \eqref{def_n}, $\Lambda =\Lambda^{\mathsf{HCI}}_{n}$ be as in \eqref{HC_index_set_inf}, $N = |\Lambda|$ and $\delta>0$ be a constant whose value will be chosen later in Step 4. Let $\Phi_{\Lambda,\delta}$ be as in  Theorem \ref{Prop_Ex_NN} 
and consider the class of DNNs \eqref{class_DNN}. Then, as shown in \S \ref{S:formulation_recovery}--\ref{S:Unk_recov}, we can reformulate the DNN training problem \eqref{trainingprob1} as the Banach-valued, weighted SR-LASSO problem \eqref{wsr-LASSO_NN}.

\vspace{1pc} \noindent
\textit{Step 2: Establishing the weighted rNSP.} Let $\bm{A}' \in \bbR^{m \times N}$ be given by \eqref{def_A_NN_general}. We now prove that the induced operator $\bm{A}' \in \cB(\cV^N, \cV^m)$ has the rNSP over $\cV$ of order $(k,\bm{u})$ with constants $\gamma' > 0$ and $0<\rho'<1$ to be specified. We do this first by establishing the wrNSP for $\bm{A}$, and then by using a perturbation result \cite[Lem.\ 12]{AdcockEtAl2021MSML} to establish it for $\bm{A}'$.

First, define the weighted sparsity parameter
\begin{equation}\label{def_k}
k:= \sqrt{\dfrac{m}{c_0L}},
\end{equation}
where $L = L(m,\epsilon)$ is as in \eqref{def_L} and $c_0 \geq 1$ is a universal constant. Observe that $m \geq m/L \geq m/(c_0 L) = k^2$, since $m \geq 3$ by assumption and therefore $L(m,\epsilon) \geq 1$ for all $0<\epsilon<1$.
Our aim now is to apply Lemma \ref{l:LegMat_RIP_general} to show that $\bm{A}$ has the weighted RIP over $\bbR$ of order $(2k,\bm{u})$. Let   $\bar{c}_0$   be the constant considered therein (related, in turn, to the constants in \cite[Thm.~2.14]{Brugiapaglia2021}). 
Note that we now use the notation $\bar{c}_0$ to avoid confusion with the constants $c_0$  in Theorem \ref{t:mainthm1}. Set
\bes{
\bar{\delta} = \frac{1}{(4 \sqrt{2} \sqrt{k} + 1) }.
} 
Consider \eqref{m-cond-for-wRIP}. Since $ k \leq \sqrt{m}$ and $m \geq 3$, we have $\log^2(k / \bar{\delta}) \lesssim \log^2(m)$ and since $n = \lceil m / (c_0 L ) \rceil \leq 2 m$ we have $\log(\E n) \lesssim \log(m)$. Hence, using the fact that $\log(m) \gtrsim 1$ once more, we deduce that 
\bes{
 \log^2(k / \bar{\delta})   \cdot  \log^2 ( e n) +\log(4/\epsilon) \lesssim \log^4(m) +\log( \epsilon^{-1})  = L(m,\epsilon).
}
We now assume that
\be{
\label{k-assumption}
k \geq 1.
}
In particular, this implies that $m/(c_0 L) \geq 1$. We discuss the case $k <1$ at the end of the proof. Using this, we get
\eas{
\bar{c}_0  \cdot \bar{\delta}^{-2} \cdot 2 k \cdot  \left( \log^2( k / \bar{\delta})   \cdot  \log^2( e n)+\log(4/\epsilon) \right)  & \leq c_0 \cdot k^2 \cdot L(m,\epsilon) = m,
}
for a suitably-large choice of the universal constant $c_0$. It follows that condition \eqref{m-cond-for-wRIP}, with $k$ and $\epsilon$ replaced by $2k$ and $\epsilon/2$, respectively, holds. Therefore, with probability at least $1-\epsilon/2$, the matrix $\bm{A}$ has the weighted RIP over $\bbR$ of order $(2k,\bm{u})$ with constant 
\be{
\label{delta_choose}
\delta_{2k,\bm{u}} =   \bar{\delta} =  \frac{1}{4 \sqrt{2} \sqrt{k} + 1}.
}
We now seek to apply Lemma \ref{l:implies-RIP}. Notice that, with this value of $\delta_{2k,\bm{u}}$, we have
\bes{
2 \sqrt{2} \frac{\delta_{2k,\bm{u}}}{1-\delta_{2k,\bm{u}}} = \frac{1}{2 \sqrt{k}},\qquad \frac{\sqrt{1+\delta_{2k,\bm{u}}}}{1-\delta_{2k,\bm{u}}} \leq \frac{3}{2}.
}
Here, in the second step, we used the fact that $k \geq 1$, by assumption. Thus, with probability at least $1-\epsilon/2$, $\bm{A} \in \bbR^{m \times N}$ has the weighted rNSP over $\bbR$ of order $(k,\bm{u})$ with constants 
 \begin{equation}
\rho= \frac{1}{2\sqrt{k}},\quad \gamma = \frac{3}{2}.
 \end{equation}  
 Next, we turn our attention to the matrix $\bm{A}'$. It is a short argument (see Step 3 of the proof of  \cite[Thm.\ 5]{AdcockEtAl2021MSML}) based on the definition of $\Phi_{\bnu,\delta,\Theta}$ to show that
\begin{equation}\label{eq_A_AtUB}
\nmu{\bm{A}- \bm{A}'}_{2} \leq \sqrt{N} \delta.
\end{equation}
Now suppose that $\delta$ satisfies  
\begin{equation}\label{delta_cond_1}
\sqrt{N} \delta \leq  \tilde{\delta} : = \dfrac{2}{3(3+4k)}.
\end{equation}
Later in Step 4 we will ensure that this condition is fulfilled. Then, using a straightforward extension of \cite[Lem. 8.5]{adcock2021compressive} from the unweighted to the weighted case  we deduce that, with probability at least $1-\epsilon/2$, $\bm{A}'$ has the weighted rNSP over $\bbR$ of order $(k,\bm{u})$ with constants
\bes{
\frac{\rho+\gamma \tilde{\delta} \sqrt{k}}{1-\gamma \tilde{\delta}} = \frac{3}{4 \sqrt{k}} : = \tilde{\rho},\qquad \frac{\gamma}{1-\gamma \tilde{\delta}} \leq \frac{7}{4} : = \widetilde{\gamma}.
}
Here, in the second step we used the fact that $k \geq 1$ once more. 
Finally, we now apply Lemma \ref{l:from_R_to_V_rNSP}. First note that $s^*(k) \leq k$ since $\bm{u}\geq 1$. 
Hence, with probability at least $1-\epsilon/2$, the corresponding operator $\bm{A}' \in \cB (\cV^N, \cV^m)$ has the weighted rNSP over $\cV$ of order $(k,\bm{u})$ with constants 
\be{
\label{rho_tau_prime}
\sqrt{k} \widetilde{\rho}= \frac34 : = \rho',\qquad  \sqrt{k} \widetilde{\gamma} \leq 2\sqrt{k} : = \gamma'.
}

\vspace{1pc} \noindent
\textit{Step 3: Estimating the error.} First, we recall that $\cP_{K} : \cV \rightarrow \cV_{K}$ is a bounded linear operator and $\pi_{K} = \max \{ \nm{\cP_{K}}_{\cV \rightarrow \cV_{K}},1\}$. Let $f \in \cH(\bm{b},\varepsilon)$ and consider its expansion \R{f-exp}. As in \eqref{f_exp_trunc}, let $f_{\Lambda} = \sum_{\bm{\nu} \in \Lambda} c_{\bm{\nu}} \Psi_{\bm{\nu}}$ be the truncated expansion of $f$. For convenience, we now define
\be{
\label{def-ELambdas}
E_{\Lambda,2}(f) = \nm{f -f_{\Lambda}  }_{L^2_{\varrho}(\cU ; \cV)} ,\quad E_{\Lambda,\infty}(f) = \nm{f -f_{\Lambda}  }_{L^{\infty}_{\varrho}(\cU ; \cV)}.
}
We now derive an error bound for $f - f_{\hat{\Phi},\Theta}$, where $\hat{\Phi}$ is an approximate minimizer of \eqref{trainingprob1} and $f_{\hat{\Phi},\Theta}$ is as in \eqref{fPhi_DNN}. Write  $\hat{\Phi} = \widehat{\bm{C}}^{\top} \Phi_{\Lambda,\delta}$ for $\widehat{\bm{C}} \in \bbR^{N \times K} $ and let $\hat{\bm{c}} = (\hat{c}_{\bm{\nu}})_{\bm{\nu} \in \Lambda}$ be the corresponding approximate minimizer of \eqref{wsr-LASSO_NN} defined via the relation \eqref{zZrelation}. Set
\bes{
f_{\hat{\Psi}} = \sum_{\bm{\nu} \in \Lambda} \hat{c}_{\bm{\nu}} \Psi_{\bm{\nu}}.
}
{Then
\begin{align*}
&\nmu{f -f_{\hat{\Phi},\Theta}}_{L^2_{\varrho}(\cU ; \cV)}  
\\
&~~~ \leq   \nmu{f - \cP_{K}(f) }_{L^2_{\varrho}(\cU; \cV)} 
          + \nmu{ \cP_{K}(f)-\cP_{K}(f_{\Lambda})  }_{L^2_{\varrho}(\cU ; \cV)}
          + \nmu{\cP_{K}(f_{\Lambda}) - f_{\hat{\Psi}}}_{L^2_{\varrho}(\cU ; \cV)} 
          + \nmu{f_{\hat{\Psi}} - f_{\hat{\Phi},\Theta} }_{L^2_{\varrho}(\cU ; \cV)}
\\
 & ~~~ =:   A_1 + A_2 + A_3+A_4.
\end{align*}
We now bound the terms $A_1$, $A_2$, $A_3$ and $A_4$ in several substeps.}

\vspace{1pc} \noindent \textit{Step 3(i):  Bounding $A_1$.} {First, we have that $A_1 \leq E_{\mathsf{disc}}$, where $E_{\mathsf{disc}}$ is as in \R{E123_def_1}.}

\vspace{1pc} \noindent \textit{Step 3(ii):  Bounding $A_2$.}  {Using the linearity of $\cP_K$ and the fact that it is a bounded  operator, we have
\begin{equation}\label{A1bound}
A_2 \leq \pi_K  \nmu{ f-f_{\Lambda}  }_{L^2_{\varrho}(\cU ; \cV)} = \pi_K E_{\Lambda,2}(f).
\end{equation}}

\vspace{1pc} \noindent {\textit{Step 3(iii):  Bounding $A_3$.}  Let $\bm{u}$  be the intrinsic weights in \eqref{def:int_weights}. Then we have} 
\begin{equation*}
{A_3 = \nm{\cP_{K}(f_{\Lambda}) - f_{\hat{\Psi}} }_{L^2_{\varrho}(\cU ; \cV)} 
\leq \nm{\cP_{K}(f_{\Lambda}) - f_{\hat{\Psi}} }_{L^{\infty}(\cU ; \cV)}  \leq  \nm{ \bm{c}_{\Lambda,{K}} - \hat{\bm{c}} }_{1,\bm{u};\cV},}
\end{equation*}
where $\bm{c}_{\Lambda,{K}} = (\cP_K(c_{\bnu}))_{\bnu \in \Lambda} $.  We now apply Lemma \ref{lemma-musuboptimal} to the problem \eqref{wsr-LASSO_NN}. In Step 2 we showed that, with probability at least $1-\epsilon/2$, $\bm{A}' \in \cB(\cV_K^N,\cV_K^m)$ has the weighted rNSP of order $(k, \bm{u})$ with constants $\rho'$ and $\gamma'$ given by \eqref{rho_tau_prime} and $k$ given by \eqref{def_k}. Hence, this lemma gives
\begin{equation*}
{\nm{ \bm{c}_{\Lambda,K}  - \hat{\bm{c}}}_{1,\bm{u};\cV}   
\leq  C_1 \left( 2 {\sigma_{k}(\bm{c}_{\Lambda,K} )_{1,\bm{u};\cV}} 
+\frac{\cG'(\hat{\bm{c}}) - \cG'(\bm{c}_{\Lambda,K} ) }{ \lambda} \right) + \left( \dfrac{C_1}{\lambda} + C_2 \sqrt{k} \right)\nm{\bm{A}' \bm{c}_{\Lambda,K}  - \bm{f}}_{2;\cV} ,}
\end{equation*}
with probability at least $1-\epsilon/2$, where  {$C_1 =   {(1+\rho')}/{(1-\rho')}   $, $C_2 = {2 \gamma'}/{(1-\rho')}  $} and $\cG'$ is as in \eqref{wsr-LASSO_NN}.  Notice that this holds provided $\lambda \leq C'_1 / (C'_2 \sqrt{k})${,  where  $C'_1 =   {(1+\rho')^2}/{(1-\rho')}   $, $C'_2 = {(3+\rho') \gamma'}/{(1-\rho')}  $}. Using the values for $\gamma'$ and $\rho'$ we get
\begin{equation*}
\dfrac{1}{6\sqrt{k}} =  \left(\dfrac{1}{\gamma'} \right) \frac13 <  \dfrac{1}{\gamma'} \left[ 1 - \dfrac{2}{(3+\rho')}\right] = \dfrac{1}{\gamma'}\left[  \dfrac{(1+\rho')}{(3+\rho')}\right] \leq \dfrac{1}{\gamma'}\left[  \dfrac{(1+\rho')^2}{(3+\rho')}\right],
\end{equation*}
which implies that
\begin{equation}\label{lambda_def}
\lambda : = \dfrac{1}{6 \sqrt{m/L}} = \frac{1}{6 \sqrt{c_0} k} \leq \dfrac{1}{\sqrt{k}}\cdot\dfrac{1}{6 \sqrt{k}} \leq \dfrac{1}{\sqrt{k}}\cdot\dfrac{1}{ \gamma'}\left[  \dfrac{(1+\rho')^2}{(3+\rho')} \right ] =\dfrac{C_1'}{C_2'\sqrt{k} },
\end{equation}
as required. Here we used the fact that $c_0 \geq 1$.
Now,  since $\bm{\hat{c}}$ is an approximate minimizer  and $\bm{c}_{\Lambda,K}  \in \cV_K^N$ is feasible for \eqref{wsr-LASSO_NN}, we have $\cG'(\hat{\bm{c}}) - \cG'(\bm{c}_{\Lambda,K}  ) \leq E_{\mathsf{opt}}$, where $E_{\mathsf{opt}}$ is as in \eqref{def:Eopt}. 
Also, due to the values of $\rho'$ and $\gamma'$ given by \eqref{rho_tau_prime}, we notice that $C_1,C'_1 \lesssim 1$ and $C_2,C'_2 \lesssim \sqrt{k}$. Substituting this into the previous bound and noticing that $1/\lambda \lesssim k$, we obtain
\begin{equation*}
{\nm{ \bm{c}_{\Lambda,K}  - \hat{\bm{c}}}_{1,\bm{u};\cV}     \lesssim     {\sigma_{k}(\bm{c}_{\Lambda,K})_{1,\bm{u};\cV}} 
 + {k} E_{\mathsf{opt}}+{k}\nm{\bm{A}' \bm{c}_{\Lambda,K} - \bm{f}}_{2;\cV} ,}
\end{equation*}
with probability at least $1-\epsilon/2$.
Consider the first term. Since $\cP_K$ satisfies \R{ch-def}, \R{sigma-k-w-equiv} implies that 
\be{
\label{sigma-k-Ph}
\sigma_{k}( \bm{c}_{\Lambda,K})_{1,\bm{u};\cV} = \inf \left \{ \sum_{\bm{\nu} \in \Lambda \backslash S} u_{\bm{\nu}}\| \cP_K({c}_{\bnu})\|_{\cV} : S \subseteq \Lambda,\ |S|_{\bm{u}} \leq k \right \} \leq   \pi_K \sigma_{k}(\bm{c}_{\Lambda})_{1,\bm{u};\cV}.
}
Hence
\begin{equation}
\label{SRLASSO_gives_this}
{\nm{ \bm{c}_{\Lambda,K}  - \hat{\bm{c}}}_{1,\bm{u};\cV} \lesssim \pi_K  {\sigma_{k}(\bm{c}_{\Lambda})_{1,\bm{u};\cV}}  + {k}\nm{\bm{A}' \bm{c}_{\Lambda,K}- \bm{f} }_{2;\cV} +{k} E_{\mathsf{opt}},}
\end{equation}
with probability at least $1-\epsilon/2$.
We now estimate the second term. Let $i = 1,\ldots,m$ and write  
\eas{
\sqrt{m} \left ( \bm{A}' \bm{c}_{\Lambda,K} - \bm{f} \right )_i &= \sum_{\bm{\nu} \in \Lambda}  \cP_K({c}_{\bnu})\Phi_{\bm{\nu} , \delta, n}(\bm{y}_i) - f(\bm{y}_i) - n_i
\\
& = \sum_{\bm{\nu} \in \Lambda} \cP_K({c}_{\bnu})\left ( \Phi_{\bm{\nu} , \delta,n}(\bm{y}_i) - \Psi_{\bm{\nu}}(\bm{y}_i) \right ) 
+ \sum_{\bm{\nu} \in \Lambda} \cP_K( c_{\bnu}) \Psi_{\bm{\nu}}(\bm{y}_i) - f(\bm{y}_i) - n_i.
}
Then, using Theorem \ref{Prop_Ex_NN}, the triangle inequality  and the fact that $\cP_K$ is a bounded linear operator, we get
\eas{
\|  \sqrt{m}  ( \bm{A}' \bm{c}_{\Lambda,K}  - \bm{f}  )_i  \|_{\cV} 
& \leq \sum_{\bm{\nu} \in \Lambda} \nm{  \cP_K({c}_{\bnu}) }_{\cV} \delta 
 + \nm{ \sum_{\bm{\nu} \in \Lambda}   \cP_K({c}_{\bnu}) \Psi_{\bm{\nu}}(\bm{y}_i) -  f(\bm{y}_i) }_{\cV} 
 + \nm{n_i}_{\cV}
\\
& \leq \delta \sum_{\bm{\nu} \in \Lambda} \nm{ \cP_K({c}_{\bnu}) }_{\cV} 
+ \nm{ \sum_{\bm{\nu} \not\in \Lambda} \cP_K(c_{\bnu})\Psi_{\bm{\nu}}(\bm{y}_i)  }_{\cV}  \hspace{-0.3cm}
+ \nm{f(\y_i) - \cP_K(f)(\y_i)}_{\cV}  
+ \nm{n_i}_{\cV}
\\
& \leq \pi_K \sqrt{N} \delta \nm{\bm{c}_{\Lambda}}_{2;\cV }
+\pi_K  \nm{ \sum_{\bm{\nu} \not\in \Lambda}  c_{\bnu} \Psi_{\bm{\nu}}(\bm{y}_i) }_{\cV} 
+ \nm{f- \cP_K(f)}_{L^{\infty}_{\varrho}(\cU;\cV)} 
+ \nm{n_i}_{\cV}
\\ 
& =  \pi_K \left( \sqrt{N} \delta \nm{\bm{c}_{\Lambda}}_{2;\cV}
+ \nmu{f(\y_i) - f_{\Lambda}(\y_i) }_{\cV} \right )
+ E_{\mathsf{disc}}
+ \nm{n_i}_{\cV}  ,
}
where $E_{\mathsf{disc}}$ is as in \R{E123_def_1}. {Notice that, by \eqref{f-exp}, the Cauchy-Schwarz inequality and the orthonormality of $\{\Psi_{\bnu}\}_{\bnu \in \cF}$, we have
\begin{equation}
\nm{\bm{c}_{\bnu}}_{\cV } =   \nm{ \int_{\cU} f(\bm{y}) \Psi_{\bnu} (\bm{y}) \D \varrho (\bm{y}) }_{\cV } \leq \int_{\cU} \|f(\bm{y})\|_{\cV}|\Psi_{\bnu}(\bm{y})| \D \varrho (\bm{y})
 \leq \|f\|_{L^{2}_{\varrho}(\cU; \cV)} \cdot 1 \leq 1,\quad \forall \bnu \in \Lambda.
\end{equation} }
{In the last inequality we used the fact that $ \|f\|_{L^{2}_{\varrho}(\cU; \cV)} \leq \|f\|_{L^{\infty}(\cU; \cV)} \leq 1$.} Hence
\bes{
{\nmu{\bm{A}' \bm{c}_{\Lambda,K} - \bm{f} }_{2;\cV } \leq \pi_K {N} \delta + \pi_K \sqrt{\frac1m \sum^{m}_{i=1}  \nmu{f(\y_i) - f_{\Lambda}(\y_i) }^2_{\cV} } + E_{\mathsf{disc}} + E_{\mathsf{samp}},}
}
where $E_{\mathsf{samp}}$ is as in \R{E123_def_1}.
Now, since $m  = c_0 \cdot L \cdot k^2 \geq 2 \cdot k^2 \cdot \log(4/\epsilon)$ for a sufficiently large choice of the universal constant $c_0$, the arguments in \cite[Lem.\ 7.11]{Adcock2022}  imply that
\begin{equation}
\label{eq:disc_1}
 \sqrt{\frac1m \sum^m_{i=1} \nmu{f(\y_i) - f_{\Lambda}(\y_i)}^2_{\cV}} \leq \sqrt{2} \left ( \frac{E_{\Lambda,\infty}(f) }{{k}} + E_{\Lambda,2}(f)\right ),
\end{equation}
with probability at least $1-\epsilon/2$, where $E_{\Lambda,2}(f)$ and  $E_{\Lambda,\infty}(f)$ are as in \R{def-ELambdas}.  We deduce that
\begin{equation*}
{\nm{\bm{A}' \bm{c}_{\Lambda,K}- \bm{f} }_{2;\cV }
 \lesssim \pi_K \left ( {N}\delta 
 +  \frac{E_{\Lambda,\infty}(f) }{{k}} + E_{\Lambda,2}(f) \right )
 +E_{\mathsf{disc}} +E_{\mathsf{samp}},}
\end{equation*}
with probability at least $1-\epsilon/2$.
Substituting this into \eqref{SRLASSO_gives_this} and applying the union bound now yields
\begin{equation}
\label{A2bound}
{A_3   \lesssim \pi_K \sqrt{k}\left ( \frac{\sigma_{k}(\bm{c}_{\Lambda})_{1,\bm{u};\cV}}{\sqrt{k}} + \sqrt{k} {N} \delta + \frac{E_{\Lambda,\infty}(f)}{\sqrt{k}} + \sqrt{k} E_{\Lambda,2}(f) \right ) + {k} \left ( E_{\mathsf{opt}}+E_{\mathsf{samp}} +  E_{\mathsf{disc}} \right),}
\end{equation}
with probability at least $1-\epsilon$.

\vspace{1pc} \noindent \textit{Step 3(iv):  {Bounding $A_4$.}} Recalling that $f_{\hat{\Phi},\Theta} = \sum_{\bm{\nu} \in \Lambda} \hat{c}_{\bm{\nu}} \Phi_{\bm{\nu}, \delta, \Theta}$ and that $\bm{u} \geq \bm{1}$, we first write
\bes{
A_4 = \nmu{f_{\hat{\Psi}} - f_{\hat{\Phi},\Theta}}_{L^2_{\varrho}(\cU ; \cV)} \leq \sum_{\bm{\nu} \in \Lambda} \nm{\Psi_{\bm{\nu}} - \Phi_{\bm{\nu},\delta,\Theta} }_{L^2_{\varrho}(\cU)}   {u}_{\bnu}\nm{\hat{c}_{\bm{\nu}}}_{\cV} \leq \delta \nm{\hat{\bm{c}}}_{1,\bm{u};\cV}.
}
Recall that $\hat{\bm{c}}$ is an approximate minimizer of \eqref{wsr-LASSO_NN}. Hence
\bes{
\lambda \nm{\hat{\bm{c}}}_{1,\bm{u};\cV} \leq \lambda \nm{\bm{0}}_{1,\bm{u};\cV} + \nmu{\bm{A}' \bm{0} - \bm{f} }_{2;\cV} + E_{\mathsf{opt}}= \nm{\bm{f}}_{2;\cV}+ E_{\mathsf{opt}},
}
where $\bm{0} \in \cV^N_K$ is the zero vector. Using the definitions of $\bm{f}$ and $\lambda$ in \R{def-measMatrix} and \R{lambda_def}, respectively, we see that
\bes{
\nm{\hat{\bm{c}}}_{1,\bm{u};\cV} \lesssim k \left ( \nm{\bm{e}}_{2;\cV } + \nm{f}_{L^{\infty}_{\varrho}(\cU ; \cV)} + E_{\mathsf{opt}} \right ) \leq k \left ( E_{\mathsf{samp}}+ 1 + E_{\mathsf{opt}} \right ).
}
Note that $\delta k \lesssim 1$ due to  \eqref{delta_cond_1}. Since $k \geq 1$, we get
\be{
\label{A3bound}
A_4 \lesssim k \delta + {k} \left ( E_{\mathsf{samp}}+ E_{\mathsf{opt}} \right ).
}

\pbk
\textit{Step 3(v): Final bound.}
Combining the bound for $A_1$ with the estimates  \eqref{A1bound}, \eqref{A2bound} and \eqref{A3bound} and using the facts that $\pi_K \geq 1$ and $k \geq 1$ we deduce that
\begin{equation} \label{everything_but_the_poly_bounds}
\begin{split}
\nmu{f -f_{\hat{\Phi},\Theta}}_{L^2_{\varrho}(\cU ; \cV)} 
\lesssim & ~\pi_K  {\sqrt{k}}\left ( \frac{\sigma_{k}(\bm{c}_{\Lambda})_{1,\bm{u};\cV}}{\sqrt{k}}  +\sqrt{k}\delta {N} + \frac{E_{\Lambda,\infty}(f)}{\sqrt{k}} + \sqrt{k} E_{\Lambda,2}(f) \right )
\\
&+{k}\left(E_{\mathsf{samp}}+ E_{\mathsf{opt}}+  E_{\mathsf{disc}} \right).
\end{split}
\end{equation}
This concludes Step 3.

\vspace{1pc}\noindent
\textit{Step 4: Establishing the algebraic rates.} Here, we bound the first four terms in  \eqref{everything_but_the_poly_bounds}. Recall the definition of $k$  in \eqref{def_k}. Then part (ii) of Theorem \ref{thm:best_s-term_inf-dim} with $q = 1 > p$ gives
\bes{
\frac{\sigma_{k}(\bm{c}_{\Lambda})_{1,\bm{u};\cV}}{\sqrt{k}} \leq C (\bm{b},\varepsilon,p)  \cdot k^{1/2-1/p} = C (\bm{b},\varepsilon,p)  \cdot  \left ( \frac{m}{c_0 L} \right )^{\frac12( \frac12-\frac{1}{p})},
}
where $C(\bm{b},\varepsilon,p)>0$ depends on $\bm{b}$, $\epsilon$ and p only.
Next, define the following term:
\bes{
E_{\Lambda}(f) = \frac{ E_{\Lambda,\infty}(f) }{\sqrt{k}} +  \sqrt{k}E_{\Lambda,2}(f).
}
As noted, the set $\Lambda = \Lambda_{n}^{\mathsf{HCI}}$ contains the union of all anchored sets of size at most $n$ (see \cite[Prop.\ 2.18]{Adcock2022}). We now use Corollary  \ref{cor:anchored_case_lpM_space}  with $s=n$ and $q=1$. This implies that there exists an anchored set $S \subset \cF$ of size $|S| \leq n$ such that
\begin{equation*}
E_{\Lambda,\infty}(f) = \nm{f - f_{\Lambda}}_{L^{\infty}_{\varrho}(\cU ; \cV)}  \leq   {\nmu{\bm{c}-\bm{c}_{\Lambda}}_{1,\bm{u};\cV}} \leq \nmu{\bm{c}-\bm{c}_{S}}_{1,\bm{u};\cV} \leq C (\bm{b},\varepsilon,p)    \cdot n^{1-1/p},
\end{equation*}
{where in the first inequality we used \eqref{eq:truncweight}}. Similarly, we also have
\begin{equation}\label{eq:E_lambda_C}
 E_{\Lambda,2}(f)={\nmu{f - f_{\Lambda}}_{L^{2}_\varrho(\cU;\cV)} \leq   {\nmu{\bm{c}-\bm{c}_{\Lambda}}_{1,\bm{u};\cV}} \leq \nmu{\bm{c}-\bm{c}_{S}}_{1,\bm{u};\cV} \leq C (\bm{b},\varepsilon,p)    \cdot n^{1-1/p} }.
\end{equation}
Therefore 
\begin{equation}\label{eq:def_p}
E_{\Lambda}(f) \leq C(\bm{b},\varepsilon,p) \cdot \left ( k^{-1/2} \cdot n^{1-1/p} + k^{1/2} \cdot {n^{1-1/p}} \right )
{\leq C(\bm{b},\varepsilon,p) \cdot k^{1/2} \cdot n^{1-1/p}.}
\end{equation}
Since {$p\leq 1/2$, the exponent $1-1/p$ is negative}. Using the definitions of $n$ and $k$ in \eqref{def_n} and \R{def_k}, respectively, we see that $n \geq k^2$. Hence
\begin{equation*}
E_{\Lambda}(f) \lesssim C(\bm{b},\varepsilon,p) \cdot {k^{5/2-2/p}} \leq C(\bm{b},\varepsilon,p) \cdot k^{1/2-1/p} = C(\bm{b},\varepsilon,p) \cdot \left ( \frac{m}{c_0 L} \right )^{\frac12(\frac12-\frac1p)}.
\end{equation*}
Here, in the penultimate step we use the fact that $k \geq 1$ by assumption {and $p \leq 1/2$}.
Returning to \eqref{everything_but_the_poly_bounds}, we deduce that
\begin{equation}\label{error_proof_1}
\nmu{f - f_{\hat{\Phi},\Theta}}_{L^2_{\varrho}(\cU ; \cV)} \leq \pi_K \cdot C(\bm{b},\varepsilon,p) \cdot \left ( \frac{m}{c_0 L} \right )^{{\frac12 (1- \frac{1}{p})}} + {k} \left(E_{\mathsf{disc}}+E_{\mathsf{samp}}+ E_{\mathsf{opt}}\right),
\end{equation}
provided $\delta$ satisfies {$k \delta {N} \leq k^{1-\frac1p}$}. Hence it suffices to to choose
\bes{
\delta \leq {N^{-1}    k  ^{- \frac{1}{p}} }.
} 
Therefore, recalling \eqref{delta_cond_1}, we now set 
\be{
\label{delta_def}
\delta =  \min \left \{ \frac{2}{3 (3+4k)\sqrt{N}} , {\dfrac{k^{- \frac{1}{p}}}{N}} \right \}.
}
In this way, using the definition of $E_{\mathsf{app}, \mathsf{UB}}$  in \eqref{E123_def_1}, \eqref{error_proof_1} and the bound $k \lesssim \sqrt{m}$   we get
\bes{
\nmu{f -f_{\hat{\Phi},\Theta}}_{L^2_{\varrho}(\cU ; \cV)} \lesssim {E_{\mathsf{app},\mathsf{UB}}}+{m^{1/2}} \cdot (E_{\mathsf{disc}}+E_{\mathsf{samp}}+ E_{\mathsf{opt}}),
}
as required.

\vspace{1pc} \noindent \textit{Step 5: Bounding the {width and depth} of the DNN architecture.} 
 We have now established the main error bound \R{main_err_bd}. In this penultimate step, we derive the bounds for the {width and depth} of the class of DNNs $\cN$. To do this,   we follow similar arguments to those in Step 6 of the proof of \cite[Thm.\ 5]{AdcockEtAl2021MSML}. Using \eqref{delta_def} and the facts that $k \geq 1$ and $p < 1$, we first see that
\be{
\label{delta-LB}
\delta \gtrsim   {N^{-1}} k^{-\frac1p} \quad \Rightarrow \quad \log(\delta^{-1}) \lesssim  \log(N)+ \dfrac{1}{p}\log(k).
}
From the definition of  $\Lambda^{\mathsf{HCI}}_{n}$  in  \eqref{HC_index_set_inf}  notice that $m(\Lambda)= \max_{\bm{\nu}\in\Lambda}\|\bm{\nu}\|_1 \leq n$. We now apply Theorem~\ref{Prop_Ex_NN} with the ReLU activation function and the choice $\Theta =[n]$ as in \R{def_n}. Notice that this choice is valid, since every $\bm{\nu} \in \Lambda = \Lambda^{\mathsf{HCI}}_n$ satisfies $\supp(\bm{\nu}) \subseteq [n]$. We deduce that the width and depth of the network $\cN = \cN^{ 1}$ satisfies  
  \begin{align*}
   \mathrm{width}(\cN^{ 1}  ) & \lesssim   N  n, \\
    \mathrm{depth}( \cN^{ 1}  )  & \lesssim      \Big(1+ \log(n)\Big[  \log(n )+ \log(N)+p^{-1}\log(k)
   + n \Big]\Big),
\end{align*} 
Noticing that, $3 \leq m$, $k^2 \leq n \leq m $ and $N \lesssim n^{2 + \log_2(n)}$, $\log(N) \lesssim \log^2(n)$ (these follow from \R{N_bound}) now gives the result in the ReLU case. On the other hand, for the RePU or hyperbolic tangent activation function, the width and depth of this network satisfy
  \begin{align*}
  \mathrm{width}(\cN^j  )  \leq c_{j,1}  \cdot n^{ 3+ \log_{2}(n)},  \quad 
  \mathrm{depth}(\cN^j  ) \leq c_{j,2}  \cdot  \log_2(n ),
\end{align*}
where   $c_{j,1}$, $c_{j,2}$ are universal constants for the  hyperbolic tangent  activation function  ($j=0$) and $c_{j,1},c_{j,2}$ depend on $ {\ell}$ for the RePU activation function ($j= {\ell}$).  The bounds in these cases now follow from the fact that $n \leq m$.

\vspace{1pc} \noindent
\textit{Step 6: The case $k<1$.} So far, we  have assumed that  $k \geq 1$. We now address the case $k<1$. In this case, since $p<1$ we have
\begin{equation}\label{eq:kless_1implies}
k=\sqrt{\dfrac{m}{c_0L}}<1 \Rightarrow  1 < \left(\dfrac{m}{c_0L}\right)^{\frac12(1-\frac1p)}.
\end{equation}
Next, we skip Step 2 of the above argument, and go directly to Step 3. Once more, we observe that
\begin{align*}
&\nmu{f -f_{\hat{\Phi},\Theta}}_{L^2_{\varrho}(\cU ; \cV)}  
 \leq  {  A_1 + A_2 + A_3+A_4.}
\end{align*}
We now bound the various terms.

\vspace{1pc} \noindent
\textit{Step 6(i):  Bounding $A_1$.} Once more we have $A_1 \leq E_{\mathsf{disc}}$.

\vspace{1pc} \noindent
\textit{Step 6(ii):  Bounding $A_2$.} Using {\R{A1bound} and \eqref{eq:E_lambda_C} with $n=1$} we get
\begin{equation*}
A_2 \leq \pi_K {E_{\Lambda,2}(f)} \leq \pi_K { C(\bm{b},\varepsilon,p).}
\end{equation*}

\vspace{1pc} \noindent
\textit{Step 6(iii):  Bounding $A_3$.} Once more, using   \eqref{eq:E_lambda_C} with $n=1$ and triangle inequality we obtain
\begin{align*}
\nmu{\cP_K( f_{\Lambda}) - f_{\hat{\Psi}}}_{L^2_{\varrho}(\cU ; \cV)} 
&\leq \nmu{\cP_K( f_{\Lambda}) -\cP_K( f) +\cP_K( f) - f_{\hat{\Psi}}}_{L^2_{\varrho}(\cU ; \cV)}\\
& \leq \pi_K \cdot C(\bm{b},\varepsilon,p) + \pi_K \|f\|_{L^{\infty}(\cU;\cV)}  { +\|\hat{\bm{c}}\|_{1,\bm{u};\cV}.}
\end{align*}
Since $\hat{\bm{c}}$ is an approximate minimizer, following the same analysis as  Step 3(iv)  gives
\bes{
\nm{\hat{\bm{c}}}_{1,\bm{u};\cV} \lesssim 
k (E_{\mathsf{samp}}+1+E_{\mathsf{opt}}).
}
Therefore $A_2 \lesssim k (E_{\mathsf{samp}}+1+E_{\mathsf{opt}}) +\pi_K$.

\vspace{1pc} \noindent
\textit{Step 6(iv):  Bounding $A_4$.} This step is almost identical and gives $A_4 \lesssim k \delta (E_{\mathsf{samp}}+1+E_{\mathsf{opt}})$.

\vspace{1pc} \noindent
\textit{Step 6(iv):  Final bound.} Combining the previous estimates and using the fact that $\pi_K \geq 1$, $\delta k<k<1$ (the first inequality follows from \R{delta_def}), we deduce that
\begin{equation*}
\nmu{f -f_{\hat{\Phi},\Theta}}_{L^2_{\varrho}(\cU ; \cV)} 
\lesssim
\pi_K(C(\bm{b},\varepsilon,p)+1)+ k(E_{\mathsf{samp}}+ E_{\mathsf{opt}}) +E_{\mathsf{disc}}.
\end{equation*}
The condition \R{eq:kless_1implies} and the fact that $m \geq 3$ give that ${m^{1/2} }> 1 > k$. Using \R{eq:kless_1implies} once more, we deduce that
\begin{equation*}
\nmu{f -f_{\hat{\Phi},\Theta}}_{L^2_{\varrho}(\cU ; \cV)} 
\lesssim \pi_K \cdot \left(\dfrac{m}{c_0L}\right)^{{\frac12(1-\frac1p)}} + {m^{1/2}}(E_{\mathsf{samp}}+ E_{\mathsf{opt}} +E_{\mathsf{disc}}).
\end{equation*}
The error bound \R{main_err_bd} now follows in this case with $C = C(\bm{b},\varepsilon,p)+1$, depending on $\bm{b}$, $\varepsilon$ and $p$ only.

\vspace{1pc} \noindent \textit{Step 6(v): Bounding the {width and depth} of the DNN architecture.} It remains to bound the {width and depth} of the DNN architecture in the case $k < 1$. Since $m/(c_0L)<1$, by \eqref{def_n} we see that $n=1$ in this case.
Recall also that $\delta$ is defined by \R{delta_def}. Hence {$\delta \gtrsim 1/N$} in this case, since $k < 1$  and $N \geq 1$. We deduce that $\log(\delta^{-1}) \lesssim \log(N)$. We now argue as in Step 5, using this bound and the fact that $n = 1 < m$. Thus, the bounds still hold in this case.
\end{proof}

\begin{proof}[Proof of Theoorem \ref{t:mainthm1H}]
The proof is similar  to that of Theorem \ref{t:mainthm1}, except that {it uses Parseval's identity} and \cite[Lem. 7.5]{adcock2022efficient} instead of Lemma \ref{l:from_R_to_V_rNSP}. 

\pbk
\textit{Step 1: Problem setup.} This step is identical.

\vspace{1pc} \noindent
\textit{Step 2: Establishing the weighted rNSP.}  In this case, we define
\begin{equation}\label{def_kUH}
k:= \dfrac{m}{c_0L},
\end{equation}
where $L = L(m,\epsilon)$ is as in \eqref{def_L} and $c_0 \geq 1$ is a universal constant. Observe that $m \geq m/L \geq m/(c_0 L) = k$. Let  $\bar{c}_0$  be the constant considered in Lemma~\ref{l:LegMat_RIP_general}. Set $\bar{\delta} = \left({4 \sqrt{2}+1}\right)^{-1}$. Observe that in this case we do not need the assumption $k>1$ as in the previous case. Indeed, since $k \leq m$, $m \geq 3$ and  $n \leq 2m$ we deduce that
\eas{
\bar{c}_0  \cdot \bar{\delta}^{-2} \cdot 2 k \cdot  \left( \log^2(k/\bar{\delta})   \cdot  \log^2( e n) +\log(4/\epsilon) \right)  & \leq c_0 \cdot k \cdot L(m,\epsilon ) = m,
}
for a suitably-large choice of the universal constant $c_0$. By similar arguments to those of the previous theorem, considering Lemma \ref{l:LegMat_RIP_general} and Remark \ref{rmk:casek}, we deduce that, with probability at least $1-\epsilon/2$, the matrix $\bm{A}$ has the weighted RIP over $\bbR$ of order $(2k,\bm{u})$ with constant 
\be{
\label{delta_chooseUH}
\delta_{2k,\bm{u}} =  \bar{\delta} = \dfrac{1}{4 \sqrt{2}+1}.
}
Hence, by Lemma \ref{l:implies-RIP} it has the weighted rNSP over $\bbR$ of order $(k,\bm{u})$ with constants
 \begin{equation}\label{eq:gammaUH}
\rho= \frac{1}{2},\quad
\gamma = \frac32.
 \end{equation}  
Note that \eqref{eq_A_AtUB} holds in this case, since this property pertains to the matrices $\bm{A}$ and $\bm{A}'$ and not the associated linear operators. We now assume  that $\delta$ satisfies  
\begin{equation}\label{delta_cond_1UH}
\sqrt{N} \delta \leq \tilde{\delta} : = \dfrac{2}{3(3+4\sqrt{k})}.
\end{equation}
 By the same extension of \cite[Lem. 8.5]{adcock2021compressive}  and   from the values of $\rho$ and  $\gamma$  in \eqref{eq:gammaUH} the matrix   $\bm{A}'$ has the weighted rNSP of order $(k,\bm{u})$ over $\bbR$ with constants  
\begin{equation}\label{rho_tau_primeUH}
\frac{\rho+\gamma \tilde{\delta} \sqrt{k}}{1-\gamma \tilde{\delta}} =  \frac34 : = \tilde{\rho},\qquad \frac{\gamma}{1-\gamma \tilde{\delta}} \leq \frac94 : = \widetilde{\gamma},
\end{equation}
with probability at least $1-\epsilon/2$.  
Then, applying \cite[Lem.\ 7.5]{adcock2022efficient} the corresponding operator $\bm{A}' \in \cB (\cV^N, \cV^m)$ satiesfies the weighted rNSP over $\cV $ of order $(k,\bm{u})$ with constants ${\rho}'<1$ and ${\gamma'}>0$, where  ${\rho}' =  \widetilde{\rho}$ and ${\gamma}' = \widetilde{\gamma} $, with probability $1-\epsilon/2$.

\vspace{1pc} \noindent
\textit{Step 3: Estimating the error.} {The setup, Step 3(i) and Step 3(ii) are identical. For Step 3(iii), using Parseval's identity, we may write
\begin{equation}
A_3= \nm{\cP_{K}(f_{\Lambda}) - f_{\hat{\Psi}} }_{L^2_{\varrho}(\cU ; \cV)} = \nm{ \bm{c}_{\Lambda,{K}} - \hat{\bm{c}} }_{2;\cV},
\end{equation}
where $\bm{c}_{\Lambda,{K}} = (\cP_K(c_{\bnu}))_{\bnu \in \Lambda} $.  We now apply Lemma \ref{lemma-musuboptimal} to the problem \eqref{wsr-LASSO_NN}. This lemma gives
\begin{equation*}
\nm{ \bm{c}_{\Lambda,K}  - \hat{\bm{c}}}_{2;\cV}   \leq  \dfrac{C_1'}{\sqrt{k}} \left( 2 {\sigma_{k}(\bm{c}_{\Lambda,K} )_{1,\bm{u};\cV}} 
+\frac{\cG'(\hat{\bm{c}}) - \cG'(\bm{c}_{\Lambda,K} ) }{ \lambda} \right) + \left( \dfrac{C'_1}{\sqrt{k}\lambda} + C'_2 \right)\nm{\bm{A}' \bm{c}_{\Lambda,K}  - \bm{f}}_{2;\cV} ,
\end{equation*}
with probability at least $1-\epsilon/2$, where  $C'_1 =   {(1+\rho')^2}/{(1-\rho')}   $, $C_2 = {(3+\rho') \gamma'}/{(1-\rho')}  $ and $\cG'$ is as in \eqref{wsr-LASSO_NN}.} The values {of $\rho'$ and $\gamma'$}  in \eqref{rho_tau_primeUH} give
\begin{equation*}
\label{lambda-boundUH}
\lambda : = \frac{1}{6 \sqrt{m/L}} = \frac{1}{6 \sqrt{c_0 k}} \leq \frac{1}{6 \sqrt{k}} <  \dfrac{(1+ \rho')^2}{(3+\rho') \gamma' }    \frac{1}{\sqrt{k}}.
\end{equation*}
{Now, since $m  = c_0 \cdot L \cdot k \geq 2 \cdot k \cdot \log(4/\epsilon)$ for a sufficiently large choice of the universal constant $c_0$, once more the arguments in \cite[Lem.\ 7.11]{Adcock2022}  imply that
\begin{equation}\label{eq:disc_2}
 \sqrt{\frac1m \sum^m_{i=1} \nmu{f(\y_i) - f_{\Lambda}(\y_i)}^2_{\cV}} \leq \sqrt{2} \left ( \frac{E_{\Lambda,\infty}(f) }{\sqrt{k}} + E_{\Lambda,2}(f)\right ),
\end{equation}
with probability at least $1-\epsilon/2$. Notice that \eqref{eq:disc_2}   is slightly different to \eqref{eq:disc_1}}.
{Following similar arguments to Step 3(iii), we deduce that $A_3$} satisfies
\begin{equation}
\label{A2bound2}
{A_3}   \lesssim \pi_K \left ( \frac{\sigma_{k}(\bm{c}_{\Lambda})_{1,\bm{u};\cV}}{\sqrt{k}} + {N} \delta + \frac{E_{\Lambda,\infty}(f)}{\sqrt{k}} + E_{\Lambda,2}(f) \right )+ E_{\mathsf{opt}} +E_{\mathsf{samp}} +E_{\mathsf{disc}} 
\end{equation}
with probability at least $1-\epsilon$. Step 3(iv) is essentially the same except that we use the bounds $\lambda \lesssim 1/\sqrt{k}$ and $\sqrt{k} \delta \lesssim 1$ in this case to get the bound
\bes{
A_4 \lesssim \sqrt{k} \delta ( E_{\mathsf{samp}} + 1 + E_{\mathsf{opt}} ) \lesssim {N} \delta + E_{\mathsf{samp}}  + E_{\mathsf{opt}}.
}
Hence, combining the estimates  and using the fact that $k \leq n \leq N$,  we deduce that
\be{
\label{everything_but_the_poly_bounds2}
\begin{split}
\nmu{f -f_{\hat{\Phi},\Theta}}_{L^2_{\varrho}(\cU ; \cV)} 
\lesssim &~
\pi_K \left ( \frac{\sigma_{k}(\bm{c}_{\Lambda})_{1,\bm{u};\cV}}{\sqrt{k}} +  {N} \delta + 
  \frac{  E_{\Lambda,\infty}(f) }{\sqrt{k}} + E_{\Lambda,2}(f) \right )
  \\
   & +E_{\mathsf{samp}} + E_{\mathsf{opt}}+E_{\mathsf{disc}}.
   \end{split}
   }
This concludes Step 3.

\vspace{1pc}\noindent
\textit{Step 4: Establishing the algebraic  rates.}  By the same arguments, except using \R{def_kUH}, we get
\begin{equation*}
\frac{\sigma_{k}(\bm{c}_{\Lambda})_{1,\bm{u};\cV}}{\sqrt{k}} \leq C (\bm{b},\varepsilon,p)  \cdot k^{1/2-1/p} = C (\bm{b},\varepsilon,p)  \cdot  \left ( \frac{m}{c_0 L} \right )^{ ( \frac12-\frac{1}{p})}.
\end{equation*}
Similarly as in Step 4, there exists an anchored set $S \subset \cF$ of size $|S| \leq n$ such that
\begin{equation*}
E_{\Lambda,\infty}(f) = \nm{f - f_{\Lambda}}_{L^{\infty}_{\varrho}(\cU ; \cV)}  \leq   {\nmu{\bm{c}-\bm{c}_{\Lambda}}_{1,\bm{u};\cV}} \leq \nmu{\bm{c}-\bm{c}_{S}}_{1,\bm{u};\cV} \leq C (\bm{b},\varepsilon,p)    \cdot n^{1-1/p}.
\end{equation*}
Now, using   Corollary  \ref{cor:anchored_case_lpM_space}  with $q=2$ we get
\begin{equation*}
 E_{\Lambda,2}(f)=\nmu{f - f_{\Lambda}}_{L^{2}_\varrho(\cU;\cV)} \leq   {\nmu{\bm{c}-\bm{c}_{\Lambda}}_{2;\cV}}  \leq C (\bm{b},\varepsilon,p)    \cdot n^{1/2-1/p}.
\end{equation*}
Hence
\begin{equation*}
E_{\Lambda}(f)= \frac{ E_{\Lambda,\infty}(f) }{\sqrt{k}} +  E_{\Lambda,2}(f) \lesssim C(\bm{b},\varepsilon,p) \cdot \left ( k^{-1/2} n^{1-1/p} + n^{1/2-1/p} \right ) \lesssim C(\bm{b},\varepsilon,p) \cdot \left ( \frac{m}{c_0 L} \right )^{\frac12-\frac1p}.
\end{equation*}
Here, in the final step, we used the definitions of $k$ and $n$ and the fact that $p < 1$.
Therefore, using \R{everything_but_the_poly_bounds2} we conclude that
\begin{equation*}
\begin{split}
\nmu{f -f_{\hat{\Phi},\Theta}}_{L^2_{\varrho}(\cU ; \cV)} &\lesssim C(\bm{b},\varepsilon,p) \cdot \pi_K \cdot \left ( \frac{m}{c_0 L} \right )^{(\frac12- \frac{1}{p})} +  E_{\mathsf{samp}} + E_{\mathsf{opt}}+E_{\mathsf{disc}},
\end{split}
\end{equation*}
provided {$\delta \leq N^{-1} k^{1/2-1/p}$}. Hence, in view of \R{delta_cond_1UH}, we now set {$\delta = \min \left \{ \frac{2}{3(3+4 \sqrt{k}) \sqrt{N}} , \frac{k^{\frac12-\frac1p}}{{N}} \right \}$}.

\vspace{1pc}\noindent
\textit{Step 5: Bounding the {width and depth} of the DNN architecture.} This step is essentially the same. Since bound \R{delta-LB} remains valid, the only possible difference is in the various universal constants.

\pbk
Note that in this proof we do not need the assumption $k<1$. Hence, Step 6 is not necessary.
\end{proof}

\subsection{Known anisotropy case}

\begin{proof}[Proof of Theorem \ref{t:mainthm2}] We proceed in similar steps to those of the previous two theorems.

\pbk
\textit{Step 1: Problem setup.} Let $S \subset \cF$ be a finite index set and write $s = |S|$. We will choose a suitable $S$ in Step 4 below. Now let $\Theta \subset \bbN$ be any set for which
\be{
\label{Theta-def-known}
 \bigcup_{\bm{\nu} \in S} \mathrm{supp}(\bm{\nu}) \subseteq \Theta.
}
Notice that the left-hand side is a finite set, since $S$ is finite and any multi-index $\bm{\nu} \in \cF$ has only finitely many nonzero terms. Hence $\Theta$ can be chosen as a finite set.  We make a precise choice of $\Theta$ in Step 4 once we have defined $S$. Next, let
$\Phi_{S,\delta}$ be as in Theorem \ref{Prop_Ex_NN}, where $\delta > 0$ will also be chosen in Step 4, and 
consider the class of DNNs \eqref{class_DNN} with $S$ in place of $\Lambda$ and $s$ in place of $N$.

\pbk
\textit{Step 2: Establishing the weighted rNSP.} The main difference in this case is the use of Lemma \ref{LS-wrNSP} to assert the weighted rNSP instead of Lemma \ref{l:implies-RIP}.
Let $\bm{A},\bm{A}' \in \bbR^{m \times s}$ be given by \eqref{def-measMatrix-known} and set
\be{
\label{def_kBK}
\bar{k} := \frac{m}{11 L} \leq \frac{m}{2},
}
where $L = L(m,\epsilon) \geq 1$ is as in \R{def_L-known}. We now make the following assumption:
\be{
\label{k-assume-known}
\bar{k} \geq k : = |S|_{\bm{u}}.
}
Later, when we construct the set $S$ in Step 4 we will verify that this holds. We now apply Lemma \ref{LS-wrNSP} with $\delta = 2/5$. Notice that
\bes{
m = 11 \cdot \bar{k} \cdot L(m,\epsilon) \geq ((1-\delta) \log(1-\delta) + \delta)^{-1} \cdot k \cdot \log(2k/\epsilon).
}
Hence, with probability at least $1-\epsilon/2$, the matrix $\bm{A}$ has the weighted rNSP over $\bbR$ of order $(k,\bm{u})$ with constants $\rho = 0$ and $\gamma = \sqrt{5/3}$. Or equivalently (recall the proof of Lemma \ref{LS-wrNSP}), the bound
\be{
\label{wrNSP-known-case}
\nm{\bm{x}}_2 \leq \sqrt{5/3} \nm{\bm{A} \bm{x}}_2,\quad \forall \bm{x} \in \bbR^{s},
}
holds with probability at least $1-\epsilon/2$. Now, much as before, we have
\begin{equation}\label{eq_A_AtBK}
\nmu{\bm{A}- \bm{A}'}_{2} \leq \sqrt{s} \delta.
\end{equation}
Suppose that
\be{
\label{delta_cond_1BK}
\sqrt{s} \delta \leq \frac{\sqrt{3}}{2 \sqrt{5}}.
}
Then, if \R{wrNSP-known-case} holds, we have
\bes{
\nm{\bm{x}}_2 \leq \sqrt{5/3} \nm{\bm{A}' \bm{x}}_2 + \nm{\bm{x}}_2 / 2,\quad \forall \bm{x} \in \bbR^{s},
}
which implies that
\bes{
\nm{\bm{x}}_2 \leq 2 \sqrt{5/3} \nm{\bm{A}' \bm{x}}_2,\quad \forall \bm{x} \in \bbR^{s}.
}
We deduce that, with probability at least $1-\epsilon/2$, $\bm{A}'$ has the weighted rNSP over $\bbR$ of order $(k,\bm{u})$ with constants $\rho = 0$ and $\gamma = 2 \sqrt{5/3}$. Finally, applying Lemma \ref{l:from_R_to_V_rNSP} (and recalling that $\bm{u}\geq 1$), we deduce that the corresponding operator $\bm{A}' \in \cB (\cV^{s}, \cV^m)$ satisfies the weighted rNSP over $\cV $ of order $(k,\bm{u})$ with constants $\rho' = 0$ and ${\gamma}' = 2 \sqrt{k} \sqrt{5/3}$, with the same probability.

\pbk
\textit{Step 3: Estimating the error.} Let $f \in \cH(\bm{b},\varepsilon)$. This step is again similar to Step 3 of the proof of Theorem \ref{t:mainthm1}. We first write
\bes{
\nmu{f -f_{\hat{\Phi},\Theta}}_{L^2_{\varrho}(\cU ; \cV)} \leq {A_1 + A_2 + A_3 + A_4,}
}
with {$A_1$, $A_2$, $A_3$ and $A_4$} defined in the same way, except with $\Lambda$ replaced by $S$ throughout.

\pbk
\textit{Step 3(i): Bounding $A_1$.} This step is identical and gives $A_1 \leq E_{\mathsf{disc}}$.

\pbk
\textit{Step 3(ii): Bounding $A_2$.} {This step is identical and gives $A_2 \leq \pi_K E_{S,2}(f)$.}

\pbk
\textit{Step 3(iii): Bounding $A_3$.}
{Let $\bm{u}$ be the intrinsic weights in \eqref{def:int_weights}, using the definition of $k$ in \eqref{k-assume-known} and {the Cauchy-Schwarz inequality} we get}
\begin{equation*}
{A_3 \leq \nmu{\bm{c}_{S,K} - \hat{\bm{c}}}_{1,\bm{u};\cV}= \sum_{\bnu \in S} u_{\bnu}   \nmu{\bm{c}_{\bnu,K} - \hat{\bm{c}}_{\bnu}}_{\cV} 
\leq \left( \sum_{\bnu \in S} u_{\bnu}^2 \right )^{1/2} \left (\sum_{\bnu \in S}   \nmu{\bm{c}_{\bnu,K} - \hat{\bm{c}}_{\bnu}}_{\cV}^2 \right)^{1/2}}
\end{equation*}
{
and therefore
\begin{equation*}
A_3 \leq \sqrt{k}\nmu{\bm{c}_{S,K} - \hat{\bm{c}}}_{2;\cV}.
\end{equation*}
}
We now use the weighted rNSP for $\bm{A}'$ to deduce that
\begin{equation*}
A_3 \lesssim {k} \nmu{\bm{A}'(\bm{c}_{S,K} - \hat{\bm{c}})}_{2;\cV} \leq {k} \left ( \nmu{\bm{A}' \bm{c}_{S,K} - \bm{f} }_{2;\cV} + \nmu{\bm{A}' \hat{\bm{c}} - \bm{f} }_{2;\cV} \right ).
\end{equation*}
Now $\hat{\bm{c}}$ is an approximate minimizer of \R{def:G_LS_probelmA} and $\bm{c}_{S,K} \in \cV^{s}_K$ is feasible. Therefore
\begin{equation*}
A_3 \lesssim {k} \left ( 2 \nmu{\bm{A}' \bm{c}_{S,K} - \bm{f} }_{2;\cV} + E_{\mathsf{opt}} \right ).
\end{equation*}
Via the same arguments as before, we now bound
\begin{equation*}
\nmu{\bm{A}' \bm{c}_{S,K} - \bm{f} }_{2;\cV} \leq \pi_K {s} \delta + \pi_K \sqrt{\frac1m \sum^{m}_{i=1} \nmu{f(\y_i) - f_{S}(\y_i) }^2_{\cV}} + E_{\mathsf{disc}} + E_{\mathsf{samp}}.
\end{equation*}
Now, it follows from \R{def_kBK} and \R{k-assume-known} that $m = 11 \bar{k} L \geq 11 k (\log(m) + \log(1/\epsilon) ) \geq 2 k \log(4/\epsilon)$. Hence, minor modification of \cite[Lemma 7.11]{Adcock2022} gives 
\begin{equation*}
 \sqrt{\frac1m \sum_i \nmu{f(\y_i) - f_{{S}}(\y_i)}^2_{\cV}}  \leq \sqrt{2} \left ( \frac{E_{S,\infty}(f) }{\sqrt{k}} +E_{S,2}(f)  \right )
\end{equation*}
with probability at least $1-\epsilon/2$. Combining this with the previous bound, we deduce that
\begin{equation}
{A_3 \lesssim \pi_K \sqrt{k}\left ( \sqrt{k} {s} \delta  +E_{S,\infty}(f) +\sqrt{k} E_{S,2}(f) \right ) + {k} \left ( E_{\mathsf{opt}} + E_{\mathsf{disc}} + E_{\mathsf{samp}} \right ).}
\end{equation}
with probability at least $1-\epsilon$.

\pbk
\textit{Step 3(iv). {Bounding $A_4$.}} As in the corresponding step in the previous proofs, we first write
\begin{equation*}
A_{4} = \nmu{f_{\hat{\Psi}} - f_{\hat{\Phi},\Theta}}_{L^2_{\varrho}(\cU ; \cV)}  \leq \sum_{\bm{\nu} \in S} \nm{\Psi_{\bm{\nu}} - \Phi_{\bm{\nu},\delta,\Theta} }_{L^2_{\varrho}(\cU)}   \nm{\hat{c}_{\bm{\nu}}}_{\cV} \leq \delta   \sqrt{s} \nm{\hat{\bm{c}}}_{2;\cV}.
\end{equation*}
Since $\bm{A}'$ has the weighted rNSP and $\bm{\hat{c}}$ is an approximate minimizer, we get
\begin{equation*}
\begin{split}
\nm{\hat{\bm{c}}}_{2;\cV}& \lesssim \sqrt{k} \nm{\bm{A}' \hat{\bm{c}}}_{2;\cV} 
\\
& \leq  \sqrt{k} (\nm{\bm{A}' \hat{\bm{c}}-\bm{f}}_{2;\cV}+ \nm{\bm{f}}_{2;\cV})\\
& \leq  \sqrt{k} (\nm{\bm{A}'\bm{0}-\bm{f}}_{2;\cV}+ \nm{\bm{f}}_{2;\cV}+E_{\mathsf{opt}}) 
\\
& \leq \sqrt{k}  (E_{\mathsf{samp}}+1+ E_{\mathsf{opt}}).
\end{split}
\end{equation*}
Note that $\delta \sqrt{s} \lesssim 1$ due to \R{delta_cond_1BK}. Hence we obtain
\begin{equation*}
A_4 \lesssim \sqrt{s} \sqrt{k} \delta + \sqrt{k} (  E_{\mathsf{samp}}+E_{\mathsf{opt}}).
\end{equation*}

\pbk
\textit{Step 3(v). Final bound.} Combining the estimates for {$A_1$, $A_2$, $A_3$ and $A_4$} from the previous substeps and noticing that $\pi_K \geq 1$ by definition and $k = |S|_{\bm{u}} \geq 1$ (since $\bm{u}\geq \bm{1}$), we obtain
\begin{equation}
\label{Step3-known-bound}
{\nmu{f -f_{\hat{\Phi},\Theta}}_{L^2_{\varrho}(\cU ; \cV)} \lesssim
  \pi_K \left ( {k} {s} \delta  +\sqrt{k} E_{S,\infty}(f) +{k} E_{S,2}(f) \right ) 
 + {k} \left ( E_{\mathsf{opt}} + E_{\mathsf{disc}} + E_{\mathsf{samp}} \right ) .}
\end{equation}

\pbk
\textit{Step 4: Establishing the algebraic rates.} Suppose that $\bm{b} \in \ell^p(\bbN)$ (we address the case $\bm{b} \in \ell^p_{\mathsf{M}}(\bbN)$ in Step 6 below). We now make a suitable choice of $S$ so as to obtain the desired algebraic rates of convergence.

We first apply part (ii) of Theorem \ref{thm:best_s-term_inf-dim} with $\bar{k}$ in place of $k$. Let $q = 1$. Then this guarantees the existence of a set $S_1$ with {$|S_1|_{\bm{u}} \leq \bar{k}$} such that
\begin{equation}
{
\nmu{\bm{c} - \bm{c}_{S_1}}_{2; \cV} \leq \nmu{\bm{c} - \bm{c}_{S_1}}_{1,\bm{u}; \cV} \leq C(\bm{b}, \varepsilon, p ) \cdot \bar{k}^{1-1/p}.}
\end{equation}
We now define
\be{
\label{S-known-define}
S = {S_1  \cap \Lambda},\quad \text{where }
\Lambda = \Lambda^{\mathsf{HC}}_{\lceil \bar{k} \rceil,\infty} = \left \{ \bm{\nu} = (\nu_k)^{\infty}_{k=1} \in \cF : \prod_{k : \nu_k \neq 0} (\nu_k + 1) \leq \lceil \bar{k} \rceil \right \}.
}
Observe that {$|S|_{\bm{u}} \leq |S_1|_{\bm{u}} \leq \bar{k}$}. Therefore \R{k-assume-known} holds for this choice of $S$. Note also that $S$ is independent of $f \in \cH(\bm{b},\varepsilon)$ and depends only on $\bm{b}$, $\varepsilon$ (see Remark \ref{rem:S-idpt-f}).

Having defined $S$, we now bound
\begin{equation*}
{E_{S,2}(f) \leq E_{S,\infty}(f) }=\nmu{f - f_S}_{L^{\infty}_{\varrho}(\cU;\cV)} \leq \nmu{\bm{c} - \bm{c}_{S}}_{1,\bm{u};\cV} 
\leq C(\bm{b},\varepsilon,p) \cdot \bar{k}^{1-1/p}+\nmu{\bm{c}-\bm{c}_{\Lambda}}_{1,\bm{u};\cV}
\end{equation*}
Now, the set $\Lambda$ is precisely the union of all lower sets (see Definition \ref{d:lower-anchored-set}) of size at most $\lceil \bar{k} \rceil$ (see, e.g., \cite[Prop.\ 2.5]{Adcock2022}). Hence, by part (i) of Theorem \ref{thm:best_s-term_inf-dim},
\bes{
{\nmu{\bm{c} - \bm{c}_{\Lambda}}_{1,\bm{u};\cV} \leq C(\bm{b},\varepsilon,p)  \cdot  \bar{k}^{1-1/p}.}
}
Since $k \leq \bar{k}$, we deduce that
\begin{equation*}
{\sqrt{k}E_{S,\infty}(f) + {k} E_{S,2}(f)  \leq C(\bm{b},\varepsilon,p) \cdot \bar{k}^{2-1/p}.}
\end{equation*}
Substituting this bound into \R{Step3-known-bound} gives
\begin{equation*}
{\nmu{f -f_{\hat{\Phi},\Theta}}_{L^2_{\varrho}(\cU ; \cV)} \lesssim
  C(\bm{b},\varepsilon,p) \cdot \pi_K \left ( {k} {s} \delta  + \bar{k}^{2-\frac1p} \right ) 
 + {k} \left ( E_{\mathsf{opt}} + E_{\mathsf{disc}} + E_{\mathsf{samp}} \right ) }.
\end{equation*}
We now set 
\begin{equation}\label{delta_defBK}
\delta = \min \left \{ \frac{\sqrt{3}}{2 \sqrt{5} \sqrt{\bar{k}}} ,{ \bar{k}^{-\frac1p} }\right \}.
\end{equation}
Notice that \R{delta_cond_1BK} holds for this choice of $\delta$, since $s = |S| \leq |S|_{\bm{u}} = k \leq \bar{k}$. Substituting this into the previous expression and using the definition \R{def_kBK} of $\bar{k}$ and \R{k-assume-known} now gives
\be{
\label{known-bound-we-want}
\nmu{f -f_{\hat{\Phi},\Theta}}_{L^2_{\varrho}(\cU ; \cV)} \lesssim
  E_{\mathsf{app},\mathsf{KB}} 
 + m\left ( E_{\mathsf{opt}} + E_{\mathsf{disc}} + E_{\mathsf{samp}} \right ) ,
}
as required.

\vspace{1pc} \noindent \textit{Step 5: Bounding the {width and depth} of the DNN architecture.} We first consider $\Theta$.  Recall that $\Theta$ must satisfy \R{Theta-def-known}. By construction, any $\bm{\nu} \in S$ is also an element of $\Lambda$, and therefore $2^{\nm{\bm{\nu}}_0} \leq \lceil \bar{k} \rceil$. Hence $|\supp(\bm{\nu}) | = \nm{\bm{\nu}}_0 \leq \log_2( \lceil \bar{k} \rceil)$. Since $|S| = s$, we deduce that
\bes{
\left |  \bigcup_{\bm{\nu} \in S} \mathrm{supp}(\bm{\nu}) \right | \leq s \log_2(\lceil \bar{k} \rceil).
}
Observe that $s = |S| \leq |S|_{\bm{u}} \leq \bar{k}$, since $\bm{u}\geq 1$. Using the definition \R{def_kBK} of $\bar{k}$, the fact that $m \geq 3$ and the definition \R{def_L-known} of $L$, we see that
\bes{
s \log_2(\lceil \bar{k} \rceil) \leq \frac{m}{11 L} \frac{\log(m)}{\log(2)} \leq \frac{m}{11 \log(2)}  \leq n,
}
where $n$ is as in \R{def_n_known}. Thus, we now choose $\Theta$ as any set of size $n$ that satisfies \R{Theta-def-known}.

We now estimate the {width and depth} of the DNN architecture. First, observe that
\[
 \log(\delta^{-1}) \lesssim  p^{-1} \log(\bar{k}) \leq p^{-1} \log(m).
\]
In addition, due to the choice \R{S-known-define}, we have
\be{
\label{mS-Lambda-bound}
m(S) =  \max_{\bm{\nu}\in S}\|\bm{\nu}\|_1 \leq  \max_{\bm{\nu}\in\Lambda}\|\bm{\nu}\|_1\leq \lceil \bar{k} \rceil \leq m.
}
Hence, applying  Theorem~\ref{Prop_Ex_NN} with the set $S$ in place of $\Lambda$ and $\Theta$ as chosen above, we deduce that the width and depth in the case of the ReLU activation function satisfy  
\bes{
   \mathrm{width}(\cN^{  1} )  \lesssim     m ^2, \quad
    \mathrm{depth}( \cN^{ 1} )   \lesssim  \left(1+ \log(m) \left ( p^{-1} \log(m)+ m \right ) \right ) .
}
Here, we also used the facts that $s = |S| \leq k \leq \bar{k} \leq m$ and $n \lesssim m$.
 Now, for either the RePU or tanh activation function, we have
  \begin{align*}
  \mathrm{width}(\cN^j)  \leq c_{j,1}  \cdot m^2,  \quad 
  \mathrm{depth}(\cN^j ) \leq c_{j,2}  \cdot  \log_2(m ),
\end{align*}
where  $c_{j,1}$, $c_{j,2}$ are universal constants for the  tanh  activation function  ($j=0$) and $c_{j,1},c_{j,2}$ depend on ${\ell}$ for the RePU activation function ($j={\ell}$).  This gives the desired bounds.

\pbk
\textit{Step 6: Modifying the proof in the case $\bm{b} \in \ell^p_{\mathsf{M}}(\bbN)$.} In this case, we replace the definition of $S$ in \R{S-known-define} with
\bes{
S = S_1  \cap \Lambda,\quad \text{where }
\Lambda = \Lambda^{\mathsf{HCI}}_{\lceil \bar{k} \rceil},
}
and $\Lambda^{\mathsf{HCI}}_{\lceil \bar{k} \rceil}$ is as in \R{HC_index_set_inf}. Recall from the discussion in \S \ref{S:Unk_recov} that this set contains all anchored sets of size at most $\lceil \bar{k} \rceil$. Thus, we may argue as in Step 4, but using Corollary \ref{cor:anchored_case_lpM_space} instead of Theorem \ref{thm:best_s-term_inf-dim} to bound the error $\bm{c} - \bm{c}_{\Lambda}$. Doing so, and using exactly the same value for $\delta$ yields an identical bound \R{known-bound-we-want}.

We now modify Step 5 accordingly. By definition of $\Lambda^{\mathsf{HCI}}_{\lceil \bar{k} \rceil}$, any multi-index $\bm{\nu} \in S$ must satisfy $\mathrm{supp}(\bm{\nu}) \subseteq \{1,\ldots,\lceil \bar{k} \rceil \}$. It follows from \R{def_kBK} that $\lceil \bar{k} \rceil \leq n$, where $n$ is as in \R{Theta-n-choice-monotone}. Hence we may take $\Theta$ as in \R{Theta-n-choice-monotone}. Finally, we note that \R{mS-Lambda-bound} also holds for this choice of $S$. Thus the bounds for the widths and depths of the various DNN classes hold in this case as well.
\end{proof}
\begin{proof}[Proof of Theoorem \ref{t:mainthm2H}] 
The proof involves several  modifications to that of the previous theorem. {Step 1 is identical.} In Step 2, instead of Lemma \ref{l:from_R_to_V_rNSP} we use  \cite[Lem.\ 7.5]{adcock2022efficient} to deduce that the operator $\bm{A}' \in \cB (\cV^s, \cV^m)$ has the weighted rNSP over $\cV $ of order $(k,\bm{u})$ with $k$-independent constants $\rho'  = 0 $ and   ${\gamma'} = 2 \sqrt{5/3}$, with probability at least $1-\epsilon/2$.

\pbk
{\textit{Step 3: Estimating the error.} Let $f \in \cH(\bm{b},\varepsilon)$. Consider the same setup as before, with
\bes{
\nmu{f -f_{\hat{\Phi},\Theta}}_{L^2_{\varrho}(\cU ; \cV)} \leq A_1 + A_2 + A_3+A_4,
}
and $A_1$, $A_2$, $A_3$ and $A_4$ defined in the same way. Step 3(i) and Step 3(ii) are identical.}

\pbk
{\textit{Step 3(iii): Bounding $A_3$.} As in the corresponding step in the proof of Theorem \ref{t:mainthm1H}, using Parseval's identity we first write
\bes{
A_3 = \nmu{\bm{c}_{S,K} - \hat{\bm{c}}}_{2;\cV}.
}
Following the same arguments as before, noticing that $\gamma' \lesssim 1$, we deduce that
\begin{equation*}
A_3 \lesssim \pi_K \left (  \sqrt{s} \delta  +\dfrac{E_{S,\infty}(f)}{\sqrt{k}}  +E_{S,2}(f) \right ) +   E_{\mathsf{opt}} + E_{\mathsf{disc}} + E_{\mathsf{samp}}.
\end{equation*}
with probability at least $1-\epsilon$.}

\pbk{
\textit{Step 3(iv). Bounding $A_4$.} Using the same arguments as in the corresponding step in the previous proofs,  we obtain
\begin{equation}
A_4 \lesssim \sqrt{s}   \delta +   E_{\mathsf{samp}}+E_{\mathsf{opt}}.
\end{equation}
}

\pbk{
\textit{Step 3(v). Final bound.} Combining the estimates for $A_1$, $A_2$, $A_3$ and $A_4$ from the previous substeps  we obtain
\begin{equation*}
\nmu{f -f_{\hat{\Phi},\Theta}}_{L^2_{\varrho}(\cU ; \cV)} \lesssim
  \pi_K \left (  \sqrt{s} \delta  +\dfrac{E_{S,\infty}(f)}{\sqrt{k}}   + E_{S,2}(f) \right ) 
 + E_{\mathsf{opt}} + E_{\mathsf{disc}} + E_{\mathsf{samp}}  .
\end{equation*}}

\pbk
{
\textit{Step 4: Establishing the algebraic rates.} Suppose that $\bm{b} \in \ell^p(\bbN)$.
We now apply part (ii) of Theorem \ref{thm:best_s-term_inf-dim} with $\bar{k}/2$ in place of $k$. Let $q = 1$. Then this guarantees the existence of a set $S_1$ with $|S_1|_{\bm{u}} \leq \bar{k}/2$ such that
\bes{
\nmu{\bm{c} - \bm{c}_{S_1}}_{1,\bm{u}; \cV} \leq C(\bm{b}, \varepsilon, p ) \cdot \bar{k}^{1-1/p}.
}
Similarly, letting $q = 2$, we obtain a set $S_2$ with $|S_2|_{\bm{u}} \leq \bar{k}/2$ such that
\bes{
\nmu{\bm{c} - \bm{c}_{S_2}}_{2; \cV} \leq C(\bm{b}, \varepsilon, p ) \cdot \bar{k}^{1/2-1/p}.
}
Instead of \eqref{S-known-define}, we now define
\begin{equation}
\label{S-known-define_2}
S = (S_1 \cup S_2 ) \cap \Lambda,\quad \text{where }
\Lambda = \Lambda^{\mathsf{HC}}_{\lceil \bar{k} \rceil,\infty} = \left \{ \bm{\nu} = (\nu_k)^{\infty}_{k=1} \in \cF : \prod_{k : \nu_k \neq 0} (\nu_k + 1) \leq \lceil \bar{k} \rceil \right \}.
\end{equation}
Observe that $|S|_{\bm{u}} \leq |S_1|_{\bm{u}} + |S_2|_{\bm{u}} \leq \bar{k}$. Therefore \R{k-assume-known} holds for this choice of $S$. Once more $S$ is independent of $f \in \cH(\bm{b},\varepsilon)$ and depends only on $\bm{b}$, $\varepsilon$.
}

{
Having defined $S$, we now bound 
\begin{equation*}
E_{S,\infty}(f) =\nmu{f - f_S}_{L^{\infty}_{\varrho}(\cU;\cV)} \leq \nmu{\bm{c} - \bm{c}_{S}}_{1,\bm{u};\cV} 
\leq C(\bm{b},\varepsilon,p) \cdot \bar{k}^{1-1/p}+\nmu{\bm{c}-\bm{c}_{\Lambda}}_{1,\bm{u};\cV}
\end{equation*}
and, by Parseval's identity,
\begin{equation*}
E_{S,2}(f) =\nmu{f - f_S}_{L^2_\varrho(\cU;\cV)} = \nmu{\bm{c}-\bm{c}_S}_{2;\cV} 
\leq C(\bm{b}, \varepsilon, p ) \cdot \bar{k}^{1/2-1/p}+\nmu{\bm{c}-\bm{c}_{\Lambda}}_{2;\cV}.
\end{equation*}
Following similar arguments as before, by part (i) of Theorem \ref{thm:best_s-term_inf-dim}, we get
\bes{
\nmu{\bm{c} - \bm{c}_{\Lambda}}_{2;\cV} \leq C(\bm{b},\varepsilon,p) \cdot \bar{k}^{1/2-1/p},\qquad \nmu{\bm{c} - \bm{c}_{\Lambda}}_{1,\bm{u};\cV} \leq C(\bm{b},\varepsilon,p)  \cdot  \bar{k}^{1-1/p}.
}
Since $k \leq \bar{k}$, we deduce that
\begin{equation*}
\dfrac{E_{S,\infty}(f)}{\sqrt{k}} + E_{S,2}(f)  \leq C(\bm{b},\varepsilon,p) \cdot \bar{k}^{1/2-1/p}.
\end{equation*}
Substituting this bound into \R{Step3-known-bound} gives
\begin{equation*}
\nmu{f -f_{\hat{\Phi},\Theta}}_{L^2_{\varrho}(\cU ; \cV)} \lesssim
  \pi_K \left ( \sqrt{s} \delta  + \bar{k}^{1/2-1/p} \right ) 
 +  E_{\mathsf{opt}} + E_{\mathsf{disc}} + E_{\mathsf{samp}} .
\end{equation*}
Arguing in the same way as Step 4, and making a similar choice as in \eqref{delta_defBK}  for $\delta$, we see that
\begin{equation*}
\nmu{f -f_{\hat{\Phi},\Theta}}_{L^2_{\varrho}(\cU ; \cV)} \lesssim
  E_{\mathsf{app},\mathsf{KH}} 
 +   E_{\mathsf{opt}} + E_{\mathsf{disc}} + E_{\mathsf{samp}},
\end{equation*}
as required.
}

{
Step 5 is identical. For Step 6,  we replace the definition of $S$ in \eqref{S-known-define_2} with
\begin{equation}
S = (S_1 \cup S_2 ) \cap \Lambda,\quad \text{where }
\Lambda = \Lambda^{\mathsf{HCI}}_{\lceil \bar{k} \rceil},
\end{equation}
and $\Lambda^{\mathsf{HCI}}_{\lceil \bar{k} \rceil}$ is as in \R{HC_index_set_inf}.  Finally, we note that \R{mS-Lambda-bound} also holds for this choice of $S$. Thus the bounds for the widths and depths of the various DNNs hold in this case as well.}
\end{proof}

\section{Conclusions and open problems}\label{s:conclusions}
 
This paper developed a theoretical foundation for DL of infinite-dimensional, Banach-valued, holomorphic functions from limited data. Our main contribution was four \textit{practical existence theorems} showing the existence of DNN architectures and training procedures based on (regularized) $\ell^2$-loss functions that learn such functions with explicit error bounds that are independent of the dimension and account for all primary error sources.
We achieved this by using DNNs to emulate certain polynomial approximation methods. Our main theoretical contributions were the nontrivial extensions to the infinite-dimensional and Banach-valued function settings, and the consideration of both known and unknown parametric anisotropy.
In particular, we narrowed a key gap in the theory of approximating with DNNs in Hilbert spaces and Banach spaces, by showing near-optimal algebraic decay rates with respect to the amount of training data $m$. In all cases, we also demonstrated robustness of the training procedure to the other sources of error in the training problem.

There are several interesting directions for future research. First, whether or not the rates of decay of the approximation error can be improved in the Banach-valued case is an open problem. We conjecture that they can, and optimal rates can be shown for Banach-valued function approximation with DNNs{, if there is a way to avoid using the  $\ell^1_{\bm{u}}$-norm to bound the coefficients in Step 3(iii)}. However, this will require a different approach,  and is the subject of a future work. It is worth noting, however, that optimal rates can be shown for certain spaces $\cV_K$.  In particular, following a different proof strategy \cite[Sec.~13.2.1]{adcock2021compressive}, one could  use the basis $\{\varphi_i\}_{i=1}^K$ to assert  a weighted RIP-type bound of the form
\begin{equation*}
\alpha_K (1-\delta)\| \bm{v}\|_{2;\cV} 
\leq
\|\bm{A} \bm{v}\|_{2;\cV} 
\leq
\beta_K (1+\delta)\| \bm{v}\|_{2;\cV}
\end{equation*}
for all $s$-sparse vectors $\bm{v} \in \cV_K^N$, where $\alpha_K$ and $\beta_K$ depend the space $\cV_K$. However, this dependence leads to (typically undesirable) convergence rates of the type  $(\varpi_K m) ^{1/2-1/p}$, where $\varpi_K$ depends on $\beta_K/\alpha_K$. The construction and implementation of suitable discretizations -- i.e., those for which $\beta_K/\alpha_K \lesssim 1$ -- is nontrivial and go beyond the scope of this work (see, e.g., \cite{urban2009wavelet} and references therein for more information on this topic).

{Second, our results only consider error bounds in the $L^2_{\varrho}(\cU;\cV)$-norm. It is interesting to extend them to the $L^{\infty}_{\varrho}(\cU;\cV)$-norm.  We conjecture that similar upper bounds hold in this norm. In particular, in the Hilbert-valued case, we expect the approximation error to behave like $(m/L)^{1-1/p}$ instead of $(m/L)^{1/2-1/p}$. However, while the latter rate is optimal  (up to the polylogarithmic factor $L$) for the $L^2_{\varrho}(\cU;\cV)$-norm, it is currently unknown whether the former rate is optimal for the $L^{\infty}_{\varrho}(\cU;\cV)$-norm. Indeed, \cite{Adcock2022learning} only considers optimal approximation rates in the $L^2_{\varrho}(\cU;\cV)$-norm and it is unclear if it can be extended to the $L^{\infty}_{\varrho}(\cU;\cV)$-norm.}

Third, as described in \cite{Bhattacharya2021}, existing DL approaches for parametric DEs are often not robust when the mesh (e.g., a finite difference or finite element mesh) used to simulate the training data is refined, since the DNN architecture depends on the mesh size. In our work, the DNN architecture depends on the space $\cV_K$. This could correspond to the mesh used to generate the data, leading to a non-mesh invariant approach. But, as in \cite{Bhattacharya2021}, $\cV_K$ could also be constructed in a different manner, e.g., via PCA in the case of Hilbert spaces, leading to a mesh-invariant scheme. There are various open problems in this direction; for example, how to compute a reduced dimension space $\cV_K$ when $\cV$ is a Banach space, or how to perform adaptive mesh refinement (see, e.g., \cite{Eigel2021b,Eigel2015}).

{Fourth, our main focus in this work has been to establish practical existence theorems that are nearly optimal in the sense of the algebraic convergence of the approximation error with respect to $m$. We have not strived to optimize the width and depths of the corresponding DNNs. In the RePU and tanh cases, the depth grows logarithmically in $m$, which is reasonable in practice. However, the width bounds are $\ord{m^2}$ in the known anisotropy case and $\ord{m^{3+\log_2(m)}}$ in the unknown case. The latter, in particular, grows superalgebraically in $m$. In the known anisotropy case, it may be possible to reduce this quadratic scaling by finding a more efficient way to emulate polynomials using DNNs. However, the primary reason for the superalgebraic growth in the unknown case stems not from the specific emulation procedure, but from the need to emulate all polynomials in the large index set \R{HC_index_set_inf}: recall the cardinality bound \R{N_bound}. Finding a way to avoid forming all polynomials in this index set would be useful not just for improving the width bounds in these practical existence theorems. It would also be extremely helpful for the underlying compressed sensing-based polynomial approximation schemes, as these schemes suffer from high computational cost for precisely this reason \cite{adcock2022efficient}. }

Finally, this work focuses on Banach-valued function approximation problems that may arise from the solution of a parametric PDE parametrized by an infinite vector $\bm{y}$. As noted, a different approach involves learning the operator directly in terms of, for instance, the diffusion coefficient in the case of \R{eq:diffusion}. Recent work \cite{Herrmann2022} has established expression rate bounds for such operator learning problems. It would be interesting to see if our work could be extended in this direction to show the existence of sample-efficient training strategies for generalized operator learning problems.

\section*{Acknowledgments}

BA acknowledges the support of the Natural Sciences and Engineering Research Council of Canada of Canada (NSERC) through grant RGPIN-2021-611675. SB acknowledges the support of the Natural Sciences and Engineering Research Council of Canada (NSERC) through grant RGPIN-2020-06766, the Fonds de recherche du Québec Nature et Technologies (FRQNT) through grant 313276, the Faculty of Arts and Science of Concordia University, and the CRM Applied Math Lab. 
ND acknowledges support from a PIMS postdoctoral fellowship. {The authors would like to thank Dinh D\~ung and the two anonymous referees for helpful comments and suggestions.}

\small
\bibliographystyle{plain}
\bibliography{refb}

\end{document}